\documentclass[graybox]{svmult}

\usepackage{mathptmx}       
\usepackage{helvet}         
\usepackage{courier}        
\usepackage{type1cm}        
\usepackage{makeidx}         
\usepackage{graphicx}        
\usepackage{multicol}        
\usepackage[bottom]{footmisc}
\usepackage{amsmath}
\usepackage{amssymb}
\usepackage{cases}
\newcommand{\mybinom}[3][0.8]{\scalebox{#1}{$\dbinom{#2}{#3}$}}
\DeclareMathOperator*{\argmin}{arg\,min}

\makeindex             
                       
                       \usepackage{etoc}

\graphicspath{{./images/}}
\begin{document}

\title*{Modelling interactions between active and passive agents moving through heterogeneous environments}
\titlerunning{Modelling interactions between active (\dots) through heterogeneous environments} 
\author{Matteo Colangeli, Adrian Muntean, Omar Richardson and Thoa Thieu}
\institute{Matteo Colangeli\at Universit\`{a} degli Studi dell'Aquila,  Via Vetoio, 67100 L’Aquila, Italy; \email{matteo.colangeli1@univaq.it} \and Adrian Muntean \at Karlstad University, Universitetsgatan 2, Karlstad, Sweden; \email{adrian.muntean@kau.se}
\and Omar Richardson \at Karlstad University, Universitetsgatan 2, Karlstad, Sweden; \email{omar.richardson@kau.se}\and Thieu Thi Kim Thoa \at Gran Sasso Science Institute, Viale Francesco Crispi 7, 67100 LAquila, Italy; \email{thikimthoa.thieu@gssi.it}}

\maketitle

\abstract{We study the dynamics of interacting agents from two distinct inter-mixed  populations: One population includes active agents that follow a predetermined velocity field, while the second population contains exclusively passive agents, i.e. agents that have no preferred direction of motion. The orientation of their local velocity is affected by repulsive interactions with the neighboring agents and environment. We present two models that allow for a qualitative analysis of these mixed systems. We show that the residence times of this type of systems containing mixed populations is strongly affected by the interplay between these two populations. After showing our modeling and simulation results, we conclude with a couple of mathematical aspects concerning the well-posedness of our models.}




\section{Introduction}

Unlike fluid flows, pedestrian flows are rarely uniform. Hence, their motion is difficult to predict accurately. The main source of non-uniformity stems from the fact that pedestrian flows are \textquotedblleft thinking flows\textquotedblright, i.e., both agent-agent interactions and agent-structure interactions are always active and are much more complex than the standard Van der Waals-like (attraction-repulsion) interactions which govern to a large extent the molecular description of fluids and gases. In this framework, we consider a particular type of non-uniformity. Looking at a heterogeneous environment (e.g. a complex office building), we consider our target pedestrian flow to contain the dynamics of interacting agents from two distinct populations: 
\begin{itemize}
    \item \textit{active agents}, knowing where to go (they are aware of a predetermined optimal velocity field leading towards the exits), 
 \item \textit{passive agents}, randomly exploring the environment (they have no information about the exit routes, but base their motion on interaction with other agents). 
\end{itemize}
We are particularly interested in investigating  what mechanisms can be responsible for  the  minimization of the residence time of the pedestrians when an emergency evacuation situation has occurred, for instance, due to the unexpected occurrence of a fire that produces a significant amount of smoke. Our standing assumption is that the use of a purely macroscopic crowd model, which encodes the motion of a uniform flow, is prone to underestimate the residence time and does not properly capture crowd interaction.

In this chapter, we present conceptually different crowd dynamics models that describe the joint evolution of such passive and active agents. 
One of the models employs systems of nonlinear stochastic differential equations of motion one-way coupled with the diffusive-convective dynamics of the smoke, while an other model is a lattice-gas-type approach based on a Monte Carlo stochastic dynamics.  
Both models give estimates of the residence time of the particles as well as of the local occupancy (local pedestrian densities).  
When treating such scenarios, the complexity of the work is high. 
One of the difficulties is the handling of agent-structure interaction. 
It is worth noting that even if all agents were active and their wanted path is known {\em a priori}, if their number is sufficiently high, given a certain prescribed internal geometry of the facility, the agent-agent and agent-structure interactions normally lead to clogging or to the faster-is-slower effect; see e.g. \cite{Zuriguel2014} and \cite{Garcimartin2015}.
Another difficulty is to handle the presence of the fire, and consequently, of the smoke and of the increased discomfort the agents feel. 
We refer the reader to \cite{OmarMSC} for one possible way of treating the presence of obstacles and to \cite{Omar2017} for hints on how to introduce the fire physics in the evolution equations describing the dynamics of the crowd. 
In this framework, we focus exclusively on the effect of knowledge of the geometry on the actual dynamics of the agents.


After reviewing a number of relevant related contributions, we proceed with the description of  two closely-related modeling scenarios where the type of models previously mentioned apply. Then we solve the models numerically and illustrate the typical behavior of the output: positions, residence times, discomfort values, etc.  We also discuss a few basic aspects concerning the mathematical well-posedness of one crowd model related to Model 1.  
We close the chapter with a discussion section where we also include hints towards further potential contributions in this context.
The results reported here should be seen as preliminary. More efforts are currently invested to develop these research directions. 

 





\section{Related contributions}
\label{sec:related_contributions}
Escape evacuation and social human behaviour are closely connected. In an emergency situation, building occupants require information about the surrounding environment and social interactions in order to evacuate successfully. The experiments in \cite{Horiuchi1986} can serve as a typical example for the relevance of distinguishing between two groups of occupants: regular users of the building and those less familiar with it. 

In the model presented in \cite{Chu2014}, the building occupants are modelled
as agents who decide their evacuation actions on the basis of their infrastructural knowledge and their interactions with the social groups and the neighboring crowd therein. The authors showed that both familiar agents with the geometry building and social influence can dramatically impact on egress performance. 
As a multi-agent evacuation simulation tool, the ESCAPES system (presented in \cite{Tsai2011}) describes a realistic spread of knowledge to model two types of different knowledge: exit knowledge together with event knowledge.
The conclusions made based on these models are supported by experimental findings such as those reported in \cite{ronchi17}, where an evacuation was performed and the exit choice of participants was investigated, as well as the effect of the evacuation geometry.

Commonly, agent-based crowd models are based on developing individual trajectories.
Yet for dense crowds, additional dynamics come into play. This has been observed in, for instance, \cite{corbetta14}, where the interaction in dense crowds has been measured and analyzed, obtaining statistics for aggregate dynamics.
These \emph{macroscopic} properties have been observed from a theoretical perspective as well in e.g. \cite{luding07}.
One way to bridge the gap between models for regular and dense crowds is to use models defined on different spatial scales, giving rise to a so-called multiscale model. In \cite{cristiani11}, a multiscale model is proposed in terms of a granular flow formulation to display both microscopic as well as macroscopic crowd behaviour. For an investigation of handling contacts in such flows of granular matters applied to crowds, we refer the reader to the work of Maury and co-authors, compare  \cite{Faure2015}. All these papers assume that the exits are visible. For study cases when the walking environment is not visible due to the lack of light, we refer the reader to \cite{Cialela}.  There the main question is whether the grouping of the agents (involving higher coordination costs and information overload) has  a chance to favorize an eventually  quicker evacuation. From a different perspective, interesting connections to crisis management issues are made in \cite{Bellomo2016_1} and references cited therein. 

We refer the reader to further related contributions on modeling crowd dynamics as reported, for instance, in \cite{Nguyen2013}, \ \cite{Nguyen2015}, \ \cite{Anh2012}, as well as in \cite{Bellomo2015} \ and \cite{Colombo2015_2}.  

\section{Agent-based dynamics (Model 1)}
\label{sec:model_1}
In this section we introduce an agent-based model in a continuous two-dimensional multiple connected region $\Omega$, containing obstacles with a fixed location, a fire that produces smoke, and an exit. 
$\Omega$ represents the environment in which the crowd is present and tries to find the fastest way to the exit, avoiding any obstacles and the fire. The crowd is represented by the two aforementioned groups, active and passive agents. At time $t=0$, the crowd starts to evacuate from $\Omega$. In the rest of this section, $\Omega$ refers to the geometry displayed in Figure~\ref{fig:example}.

\begin{figure}[ht]
    \centering
    \includegraphics[width=0.5\textwidth]{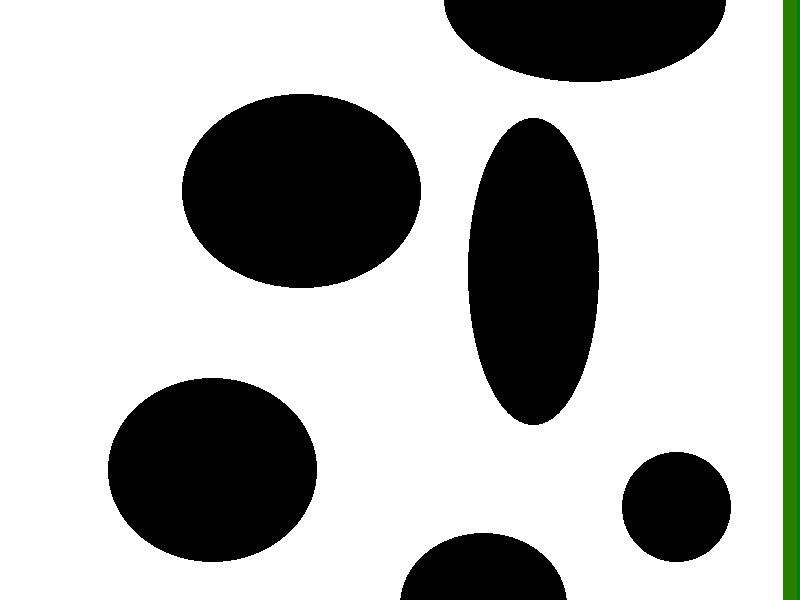}
    \caption{Basic geometry for our case study cf. Model 1. Agents are initialized in a random location within the geometry and have to reach the exit (green) while avoiding the obstacles (black).}
    \label{fig:example}
\end{figure}

Active agents have a perfect knowledge of the environment and the locations of the obstacles, but are not aware of the location of the fire prior to experiencing sensory cues. Passive agents have no information on the environment and follow their neighbours to reach the exit.
A similar model as the one described below has been presented in \cite{Omar2017}.

Active and passive agents are seen as members from the sets $X_A = \{a_1,...,a_{N_A}\}$ and $X_B = \{b_1,...,b_{N_B}\}$, respectively.
The dynamics governing their motion are described in the following sections. 
\subsection{Active agents}

The motion of the active agents is governed by a potential field model proposed by Hughes in \cite{hughes02} and adapted in \cite{treuille06}. It functions similarly to a floor field function, its counterpart in lattice models presented in e.g. \cite{tan15} and \cite{cao2014}. We modify the potential field model to account for the presence of obstacles and the effects of fire and smoke.

The potential field agrees with the principle of \emph{minimization of effort}, serving as a dynamic generalized distance transform. Let $\vec{x}$ be an arbitrary point selected in $\Omega$. We introduce a \emph{marginal cost field} $u(\vec{x})>0$, defined as
\[ u(\vec{x}) = \alpha + u_{\mathrm{obs}}(\vec{x}) + wH(\vec{x}).\]
The marginal cost field represents the effort of moving through a certain location and consists of a base level of constant walking effort $\alpha>0$, information on the geometry and the obstacles $u_{\mathrm{obs}}$, and information on the fire source $wH$. Here, $w$ takes value 1 if the agent is aware of the location of the fire and 0 otherwise.

Let $S$ be a path going from point $\vec{x}_p$ to point $\vec{x}_q$. Then the effort of walking on the path $S$ can be expressed as
\[\int_S u(\vec{\xi})d\vec{\xi}=\int_S\alpha + u_{\mathrm{obs}}(\vec{\xi}) + wH(\vec{\xi})d\vec{\xi}.\]
At the beginning of the simulation, $w$ is 0 for all agents. When an active agent experiences a significant increase in temperature because of his proximity to the location of the fire, $w$ is set to 1 and $S$ changes, and as a result, the fire is avoided.

Let $G\subset \Omega$ be the set of all inaccessible locations in the geometry (i.e. those parts of $\Omega$ covered by obstacles). Then for all $\vec{x}\in \Omega$, the geometry information (i.e. the obstacle cost field) can be expressed as
\begin{equation}
    u_{\mathrm{obs}}(\vec{x}) =
    \begin{cases}
        \infty &\mbox{if }\vec{x} \in G,\\
        \frac{1}{|d(\vec{x},G)|} & \mbox{if }\vec{x} \notin G\mbox{ and }d(\vec{x},G) \leq r_G,\\
        0 &\mbox{if } d(\vec{x},G) > r_G,
    \end{cases}
    \label{eq:pot_obs}
\end{equation}
where $r_G$ is a parameter of the order of the size of the agents. The obstacle cost makes sure that obstacle locations are inaccessible, and $r_G$ adds a tiny layer of repulsion around each obstacle to ensure the basic fact that agents do not run into walls.

The preferred path $S^*$ for an agent with location $\vec{x}_p$ and motion target $\vec{x}_q$ is determined as
\begin{equation*}
    S^* = \argmin_S \int_S u(\vec{\xi})d\vec{\xi},
\end{equation*}
where we minimize over the set of all possible motion paths $S$ from $\vec{x}_p$ to $\vec{x}_q$.
In this framework, the active agents are aware of all exits, and the optimal path $S^*$ is made available by means of the potential function $\Phi$, a solution to the equation
\begin{equation}
    \left|\left|\nabla \Phi(\vec{x})\right|\right| = u(\vec{x}),
    \label{eq:eikonal}
\end{equation}
where $||\cdot||$ denotes the standard Euclidean norm. Passive agents do not have access to the optimal paths.

Figure~\ref{fig:potential_standard} and Figure~\ref{fig:evac_path_standard} display the potential field and the corresponding paths for our case study.
Figure~\ref{fig:potential_fire} and Figure~\ref{fig:evac_path_fire} display the adaption active agents make as soon as they become aware of the fire locations and take an alternative route out.

\begin{figure}[ht]
\centering
\begin{minipage}{.5\textwidth}
        \centering
    \includegraphics[width=\textwidth]{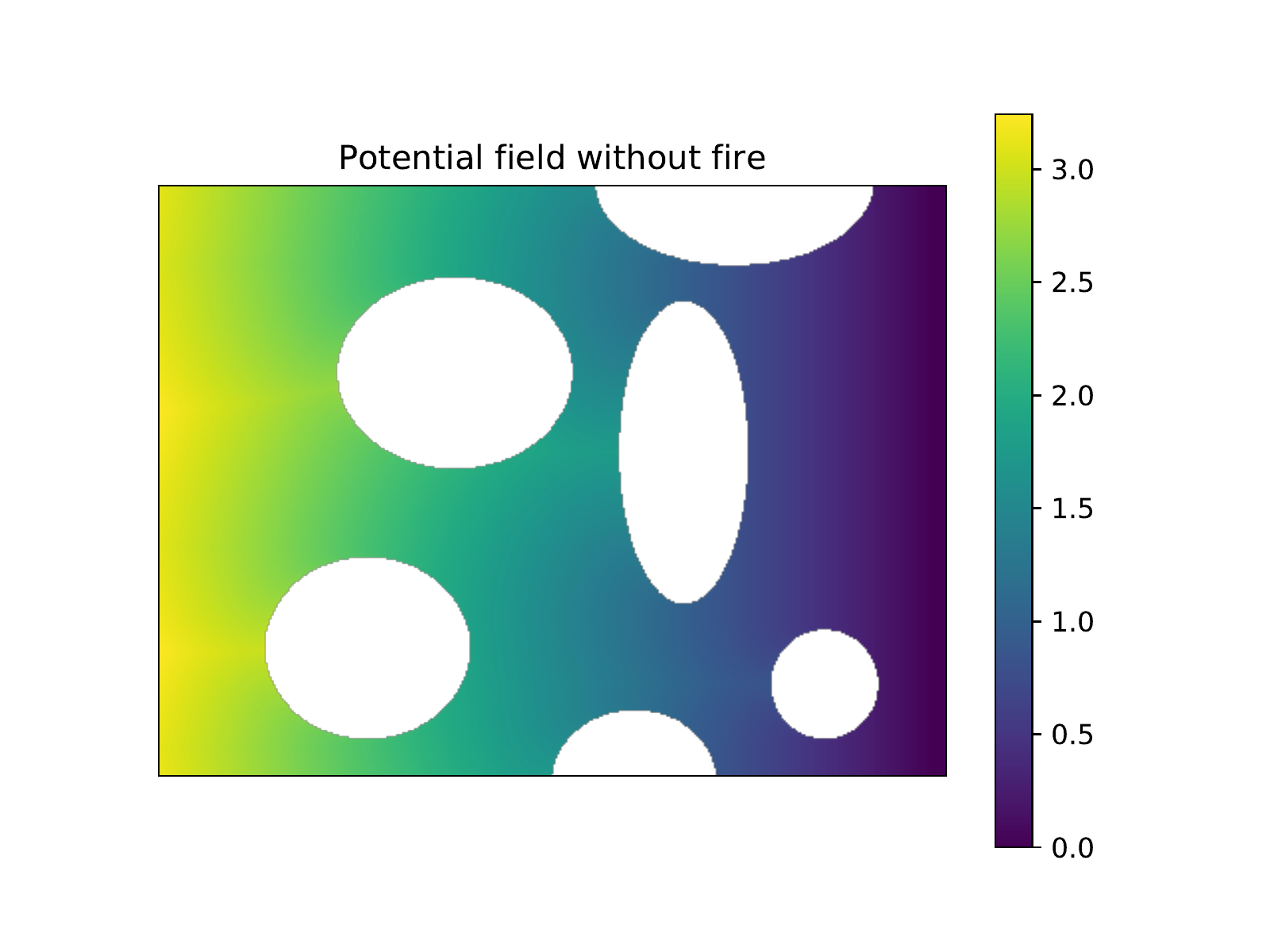}
    \caption{Potential field $\Phi$ for the environment of the case study, not taking any fire into account.}
    \label{fig:potential_standard}
\end{minipage}%
\hfill
\begin{minipage}{.4\textwidth}
        \centering
        \includegraphics[width=\textwidth]{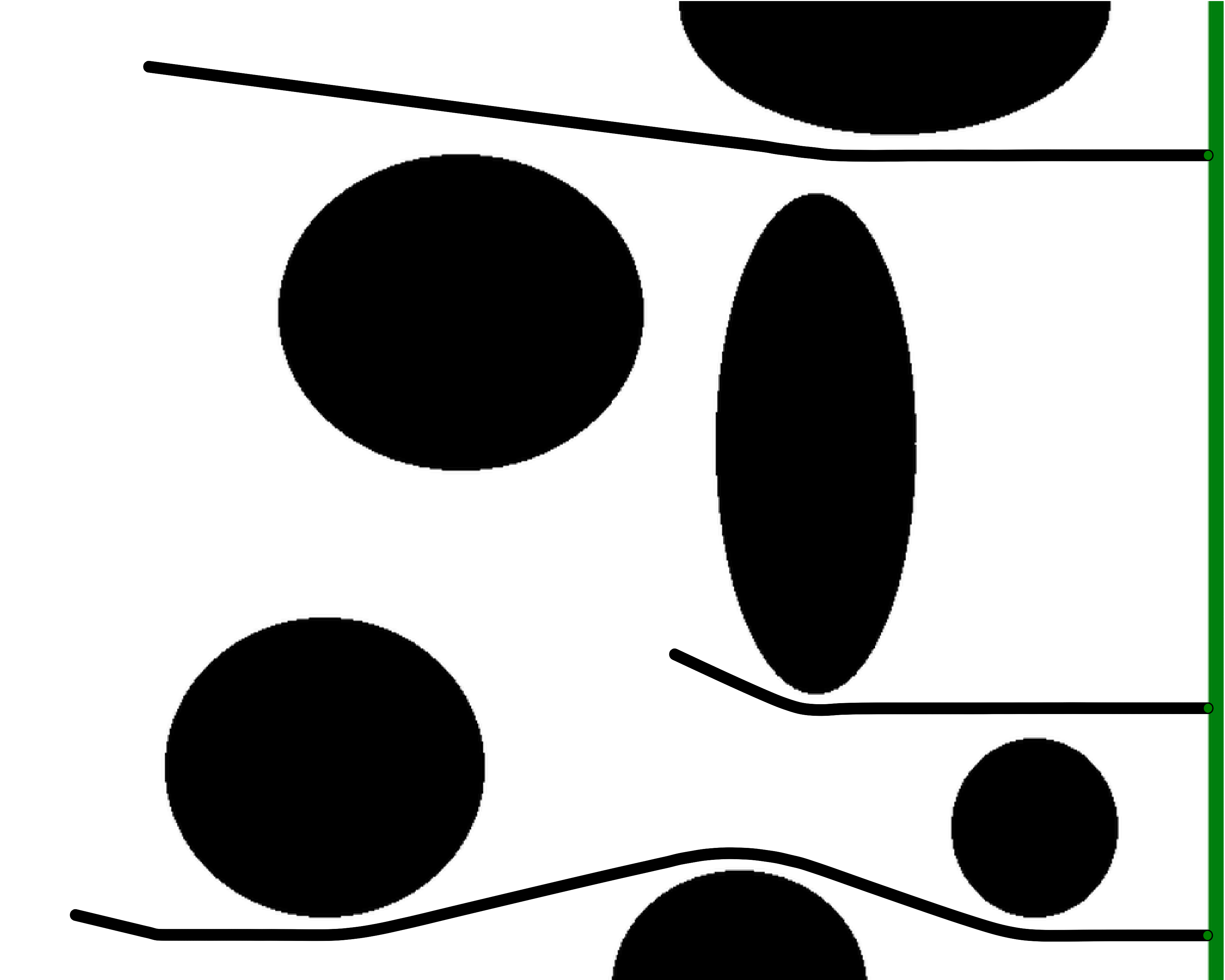}
        \caption{Paths generated from the potential field in Figure~\ref{fig:potential_standard}.}
        \label{fig:evac_path_standard}
\end{minipage}
\end{figure}

\begin{figure}[ht]
\centering
\begin{minipage}{.5\textwidth}
        \centering
    \includegraphics[width=\textwidth]{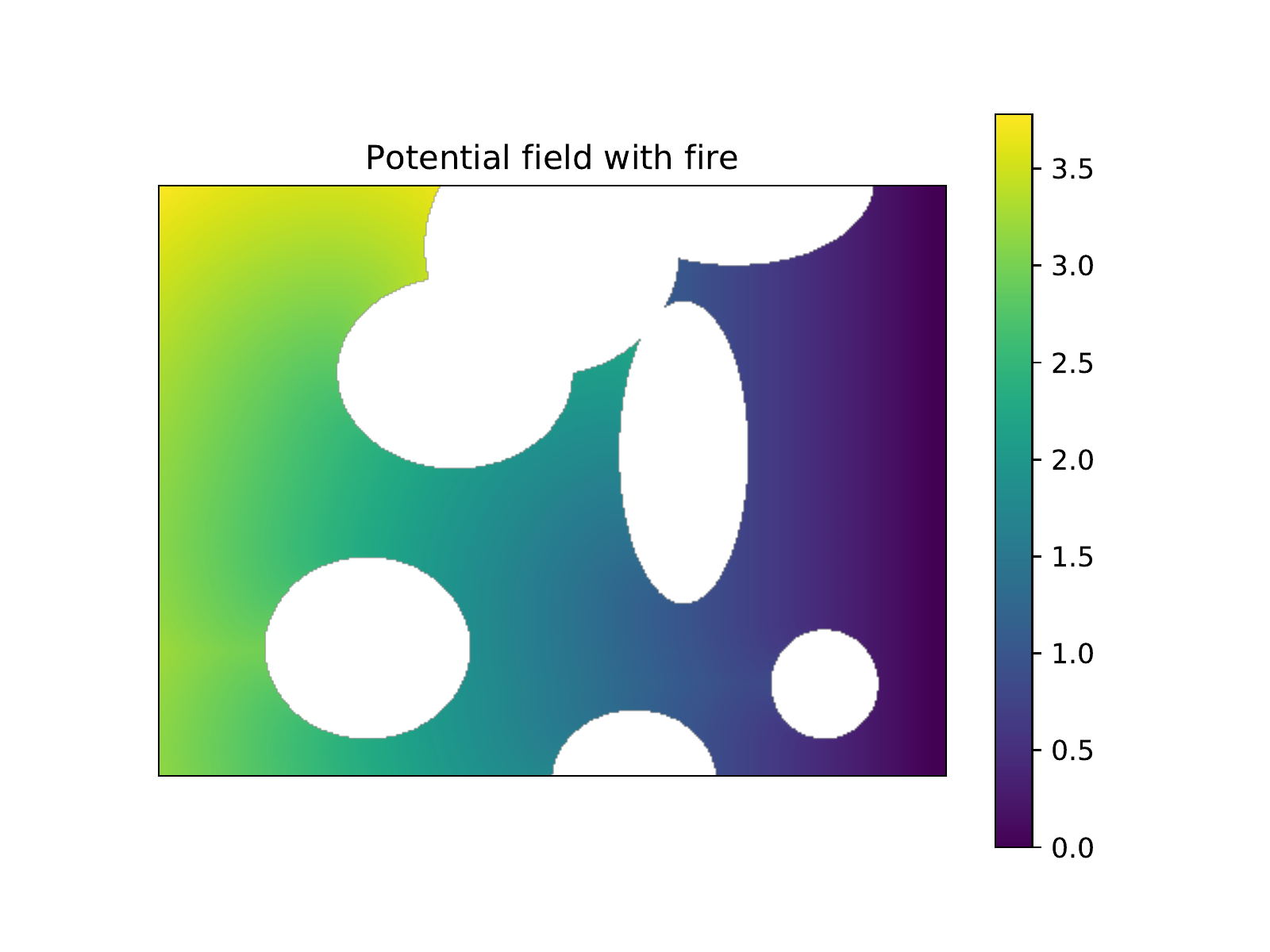}
    \caption{Potential field $\Phi$ for the environment of the case study for agents aware of the fire location.}
    \label{fig:potential_fire}
\end{minipage}%
\hfill
\begin{minipage}{.4\textwidth}
        \centering
        \includegraphics[width=\textwidth]{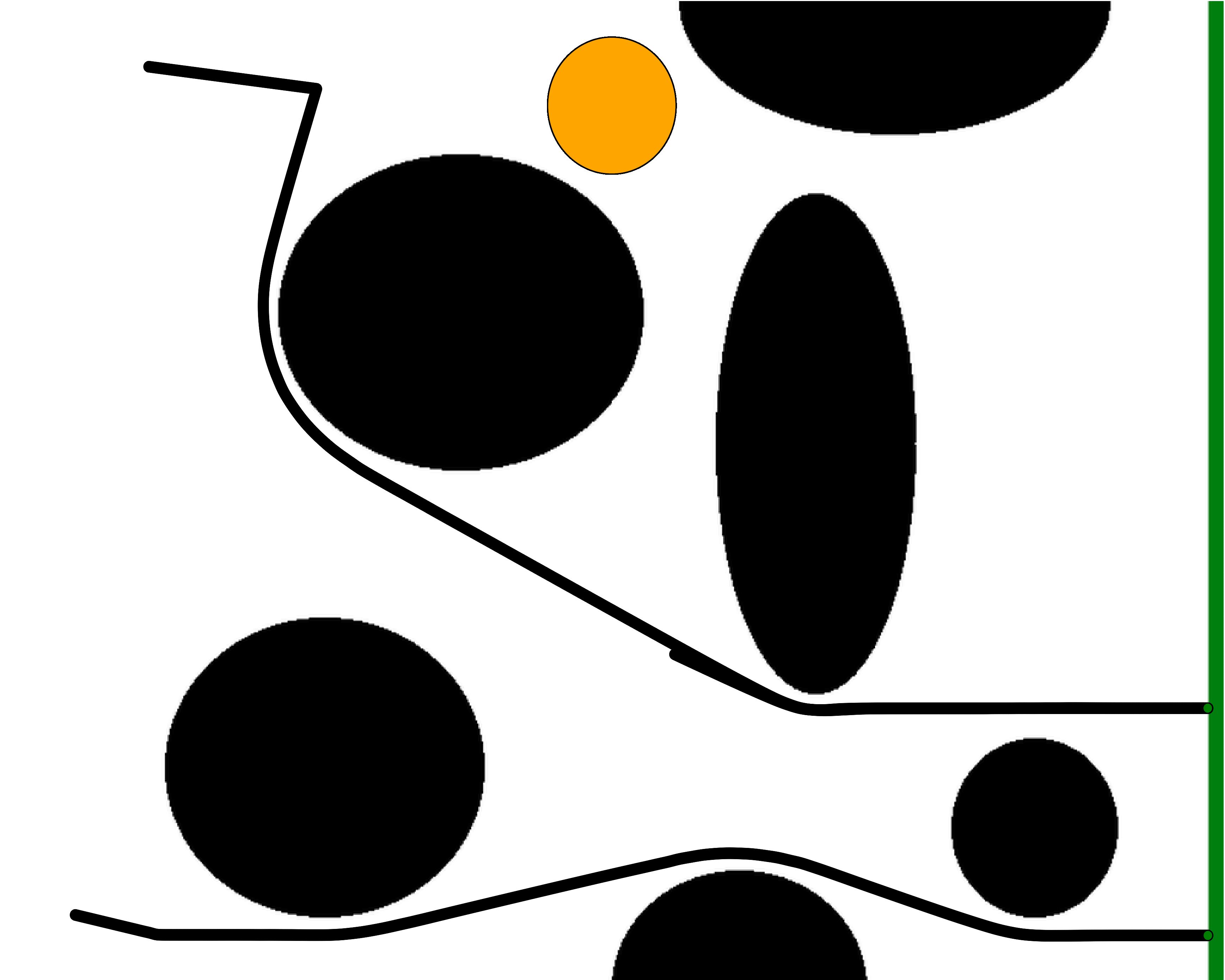}
        \caption{Paths generated from the potential field in Figure~\ref{fig:potential_fire}, avoiding the fire.}
        \label{fig:evac_path_fire}
\end{minipage}
\end{figure}

Let $\vec{x}_{a_i}(t)$ denote the position of active agent $i$ at time $t$. We express their motion within the geometry $\Omega$ by
\begin{equation}
\begin{cases}
    \displaystyle
        \frac{d\vec{x}_{a_i}}{dt} &= -v_s(\vec{x}_{a_i},t)\frac{\nabla\Phi(\vec{x}_{a_i}) - \nabla p(\vec{x}_{a_i},t)}{||\nabla\Phi(\vec{x}_{a_i}) - \nabla p(\vec{x}_{a_i},t)||},\\
        \vec{x}_{a_i}(0) &= \vec{x}_{a_i,0},
\end{cases}
\label{eq:x_a}
\end{equation}
where $\vec{x}_{a_i,0}$ represents the initial configuration of the active agents and $v_s$ represents a predefined walking speed.
In \eqref{eq:x_a}, $p$ represents a given discomfort term that influences agent interactions at the macroscopic scale. 
The discomfort measures how much agents locally have to deviate from their ideal velocity. We are on purpose vague concerning this macroscopic discomfort. In a follow-up publication, we will tackle a multiscale non-uniform crowd model where $p$ will be part of the solution to a macro-micro flow problem. 

\subsection{Passive agents}

Since we assume in this context that passive agents are unfamiliar with their environment, it is reasonable to postulate that to obtain information, they rely solely on neighbouring agents. 
This is a modelling assumption which has been confirmed for e.g. primates in \cite{meunier06}. This idea has already been applied in other crowd dynamics models (cf. e.g. \cite{helbing00}).

To model this strategy, we choose to apply a Cucker-Smale-like model which averages the velocity of nearby agents (an idea introduced originally in \cite{cucker07}).
A Brownian term $\mathbf{B}_i$ is added to this swarming-like model to represent disorienting and chaotic effects which inherently appear while moving through an unknown environment. 

We denote the positions and velocities of agent $i$ from population $X_B$ as $\vec{x}_{b_i}$ and $\vec{v}_{b_i}$, and positions and velocities from member $j$ of the complete set $X_A \cup X_B$ as $\vec{x}_j$ and $\vec{v}_j$, respectively. 

We express the motion of passive agents in the following way
\begin{equation}
\begin{cases}
    \frac{d\vec{v}_{b_i}}{dt} &= \sum_{j\in X} (\vec{v}_{j} - \vec{v}_{b_i})w_{ij}- \nabla {H}(\vec{x}_{b_i},t)\\
     &+ \frac{\vec{v}_{b_i} - \nabla p}{||\vec{v}_{b_i} - \nabla p||}\Upsilon\left(s(\vec{x}_{b_i},t)\right)+ \mathbf{B}_i(t),\\
     \frac{d\vec{x}_{b_i}}{dt}&=\vec{v}_{b_i},\\
    \vec{v}_{b_i}(0) &= \vec{v}_{b_i,0},\\
    \vec{x}_{b_i}(0) &= \vec{x}_{b_i,0}.
\end{cases}
\label{eq:x_b}
\end{equation}
In \eqref{eq:x_b}, $w_{ij}$ are weight factors, decreasing as a function of distance, defined as
\begin{equation}
    w_{ij} \sim \frac{1}{r_s^2}\exp\left(-\frac{|\vec{x_{b_i}} - \vec{x_{j}}|^2}{r_s^2}\right).
    \label{eq:weight_factor}
\end{equation}
In~\eqref{eq:weight_factor}, $r_s$ is the sight radius in the agents' location, affected by smoke level $s(\vec{x},t)$ (see Section~\ref{sec:smoke}). It should also be noted that we do not take into account those walls that block the transfer of information between agents, since they are ignored in \eqref{eq:weight_factor}. However, in the simulations described in the next section, the size of the walls generally exceeds the size of the interaction radius. 
The term $\Upsilon\left(s(\vec{x}_{b_i},t)\right)$ is simply an {\em a priori} known normalization factor depending of the smoke level; one can take $\Upsilon\left(s(\vec{x}_{b_i},t)\right)=1$ just for simplicity. In~the context of \eqref{eq:weight_factor}, the gradient in the discomfort level\footnote{Here,  we assume that the discomfort is perceptible, known.} $\nabla p$  bounds asymptotically the speed of the passive agents.

Note that, based on \eqref{eq:x_b}, passive agents follow a set of coupled second-order differential equations (a social force-like model), while following \eqref{eq:x_a}, the active agents are expected to respect a set of coupled first-order differential equations (a social velocity-like model). We believe that the 'social inertia' is much higher in the case of passive agents, so we keep the classical Langevin structure of the balance of forces, while for the active agents we choose an overdamped version. 

Another important observation is that in this model, passive agents do not know which of the other agents are active, and which are passive themselves; they follow others indiscriminately.

\subsection{Smoke effects}
\label{sec:smoke}
In addition to repelling the agents, the fire produces smoke which propagates in $\Omega$ and reduces the visual acuity of the agents.
The creation and propagation of the smoke is modelled as a diffusion-dominated reaction-advection-diffusion process.

The smoke density $s(\vec{x},t)$, is assumed to respect the following equation
\begin{equation}
\displaystyle
\begin{cases}
    \partial_ts = \operatorname{div}(D\nabla s) - \operatorname{div}(\vec{v}s) + y_s H(\vec{x})&\mbox{ in } \Omega \setminus G,\\
    (-D\nabla s+ \vec{v}s)\cdot \vec{n} = 0 &\mbox{ on } \partial \Omega \cup \partial G,\\
    s(\vec{x},0) = 0 &\mbox{ in } \Omega,
\end{cases}
\label{eq:smoke}
\end{equation}
where $D>0$ represents the smoke diffusivity, determined by the environment, $\vec{n}$ is the outer normal vector to $\partial \Omega \cup \partial G$, $\mathbf{v}$ is a given drift corresponding to, for instance, ventilation systems or indoor airflow, while $H(\vec{x})$ encodes the shape and intensity of the fire, viz.
\begin{equation}
    H(\vec{x}) = \begin{cases}
        R &\mbox{ if }|\vec{x} - \vec{x}_0| < r_0\\
        0 &\mbox{ otherwise }
    \end{cases}.
    \label{eq:fire}
\end{equation}
In our context, $D>0$ is the molecular diffusion coefficient for the smoke and a slight space dependence in $D$ is allowed. At a later stage, maybe eventually also an $s$-dependence of $D$ can be foreseen, if one would replace \eqref{eq:smoke} by an averaged version where the free motion paths and the geometry
 are perceived as some sort of \textquotedblleft homogenized\textquotedblright \ porous medium.
 
Figure~\ref{fig:smoke_prop} illustrates a snapshot of the smoke density for our case study.

\begin{figure}[ht]
    \centering
    \includegraphics[width=0.6\textwidth]{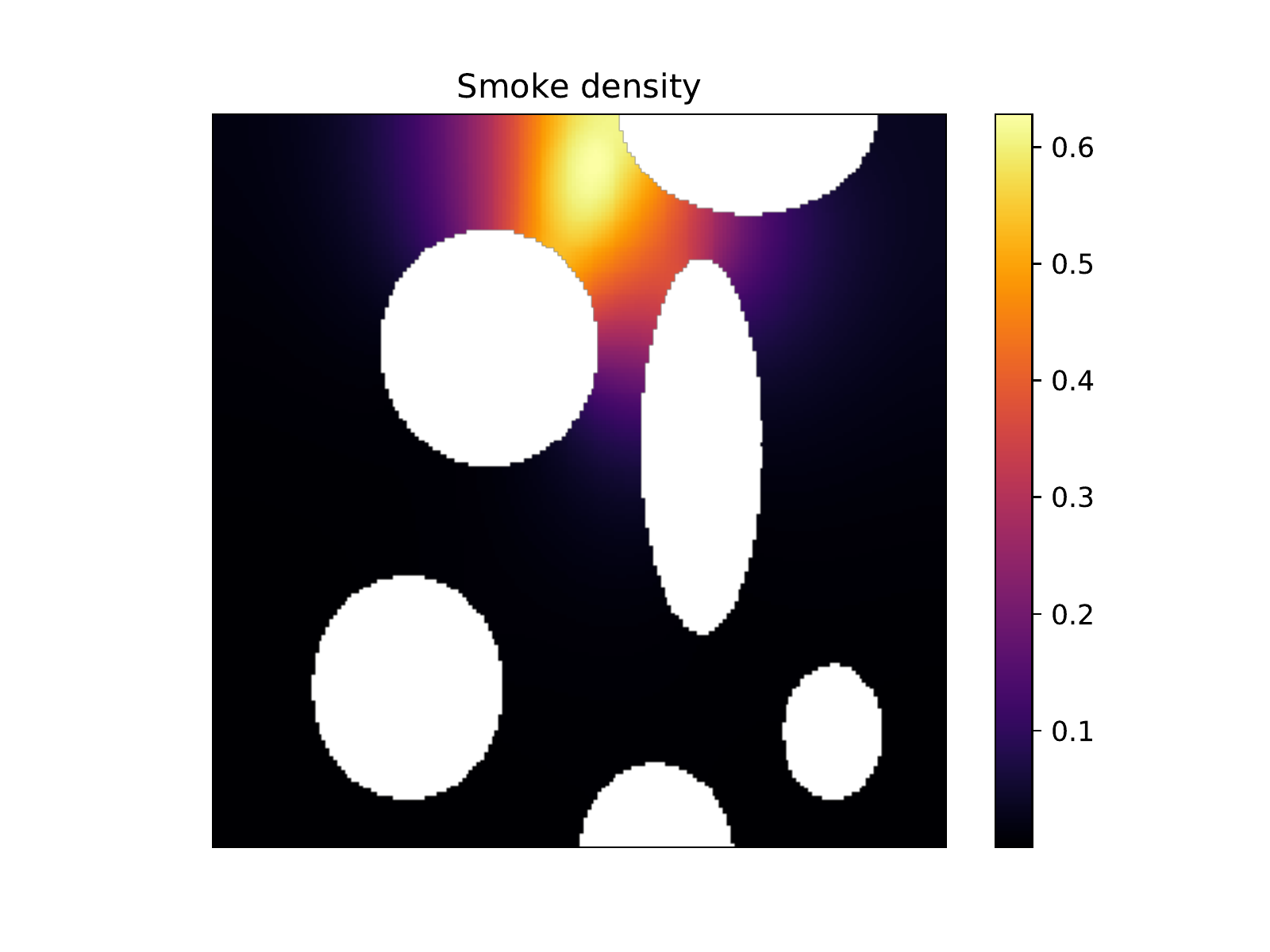}
    \caption{Smoke density in the environment at $t=60$.}
    \label{fig:smoke_prop}
\end{figure}

\section{Results Model 1: Agent-based dynamics}
\label{sec:results_1}
This section contains our numerical results obtained using the agent-based dynamics described in the previous section.

The results are run in crowd simulation prototyping application \emph{Mercurial} (\cite{mercurial}). This is an open-source framework developed in Python and Fortran to simulate hybrid crowd representations as the one described in Section~\ref{sec:model_1}. It provides both agent-based- and continuum-level visualizations and supports the design of arbitrary two-dimensional geometries. More details on the structure and implementation of \emph{Mercurial} are found in \cite{OmarMSC}.

Figure~\ref{fig:example} shows the geometry of our case study.
It has a fairly simple structure to ensure the exit can be reached even without environment knowledge.
However, the placement of the obstacles is such that zones of congestion easily occur and paths to the exit will necessarily have to be curved.

The simulation was run twice with 1000 agents, varying the ratio between active and passive agents. The first run (Case 1), of which a snapshot is presented in Figure~\ref{fig:active_configuration}, contains a total of 800 active agents and 200 passive agents. The second run (Case 2), illustrated with a snapshot in Figure~\ref{fig:passive_configuration}, contains a total of 200 active agents and 800 passive agents.

\begin{figure}[ht]
\centering
\begin{minipage}{.45\textwidth}
        \centering
        \includegraphics[width=0.8\textwidth]{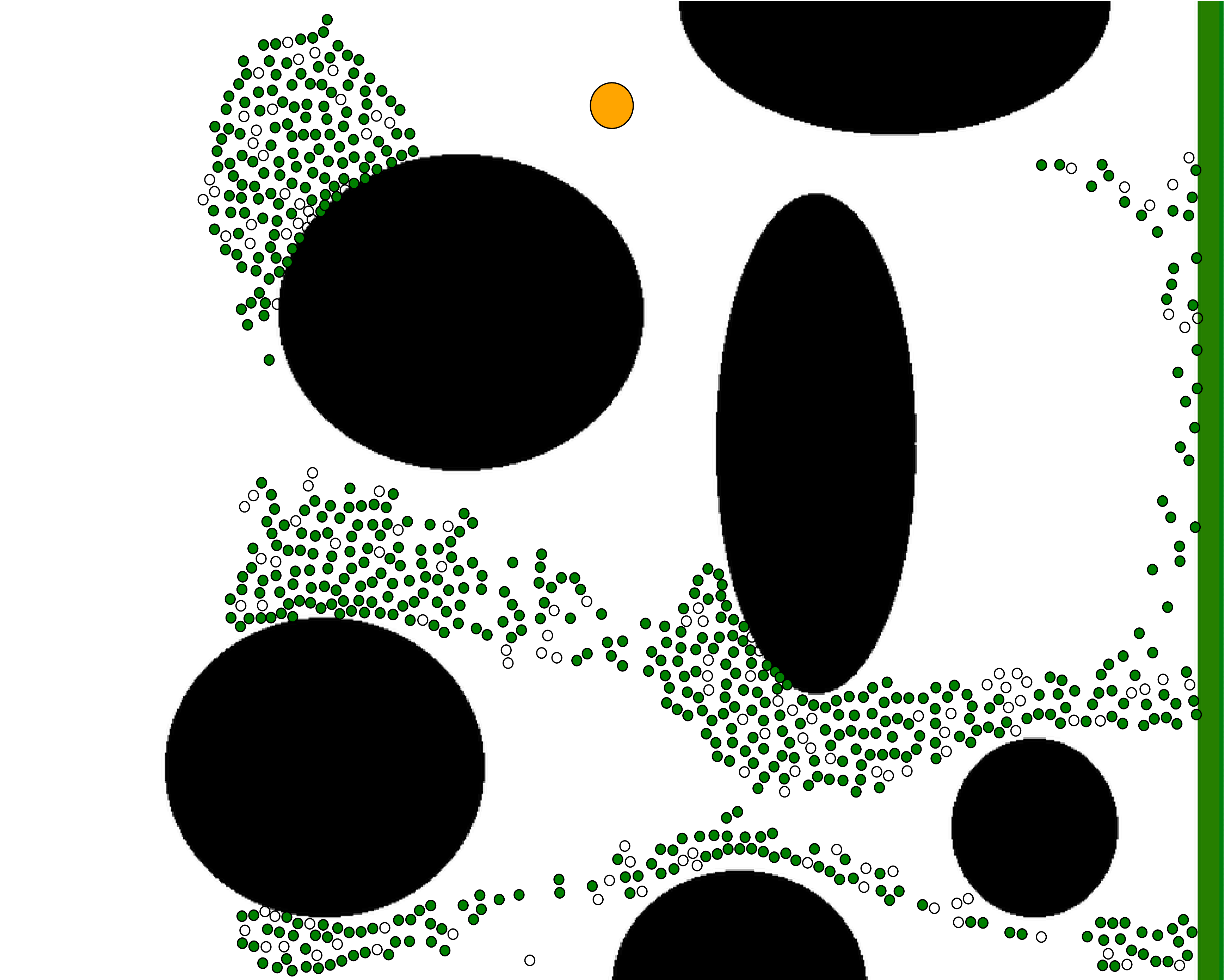}
        \caption{Snapshot after $t=17.9$ for a population with 80\% active agents. Active agents are displayed as filled circles, passive agents as open circles. The larger circle represents the center of the fire.}
    \label{fig:active_configuration}
\end{minipage}%
\hfill
\begin{minipage}{.5\textwidth}
        \centering
        \includegraphics[width=\textwidth]{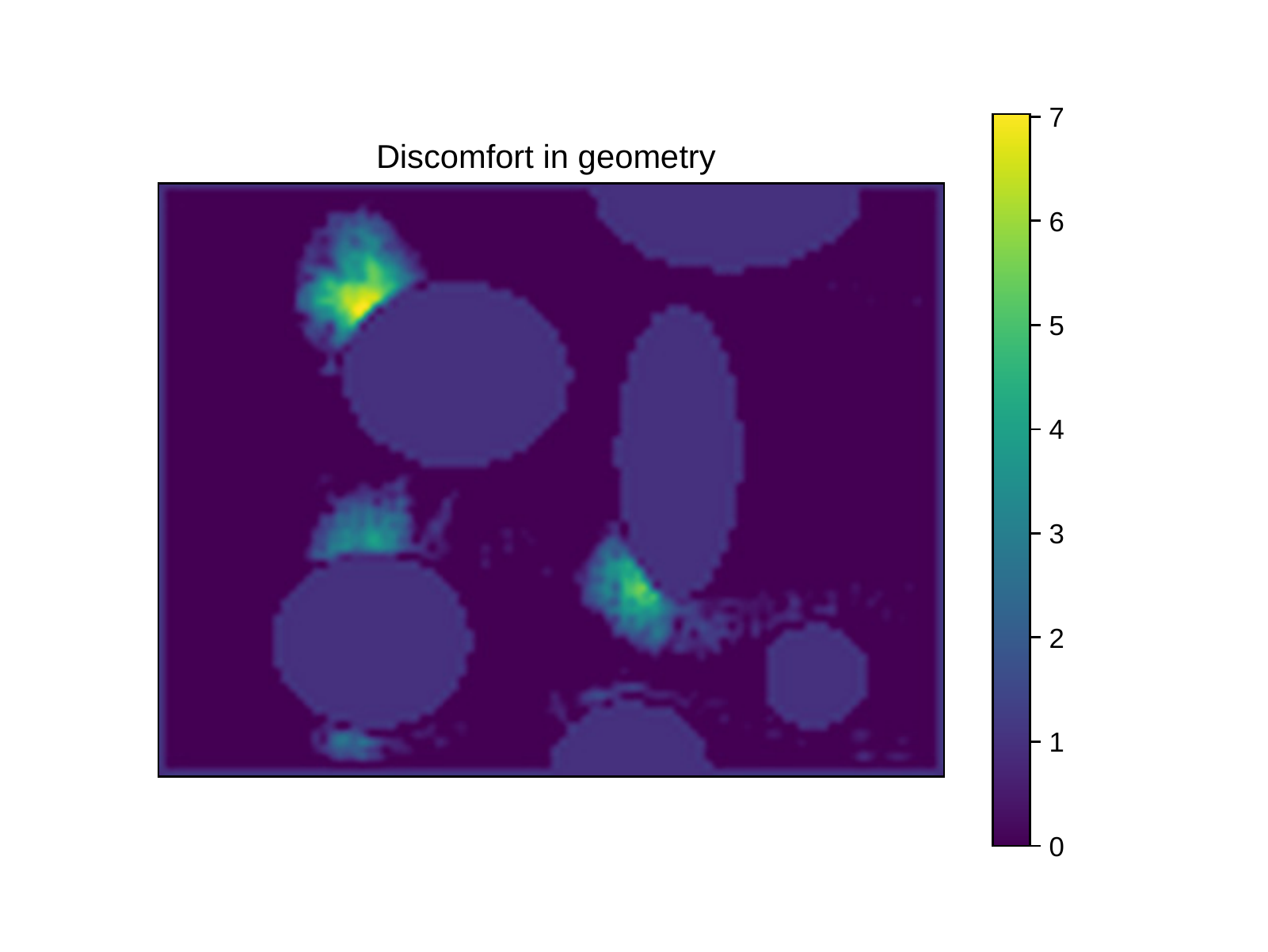}
        \caption{Discomfort observed in the simulation snapshot of Figure~\ref{fig:active_configuration}.}
        \label{fig:active_discomfort}
\end{minipage}
\end{figure}

\begin{figure}[ht]
\centering
\begin{minipage}{.45\textwidth}
    \centering
    \includegraphics[width=0.8\textwidth]{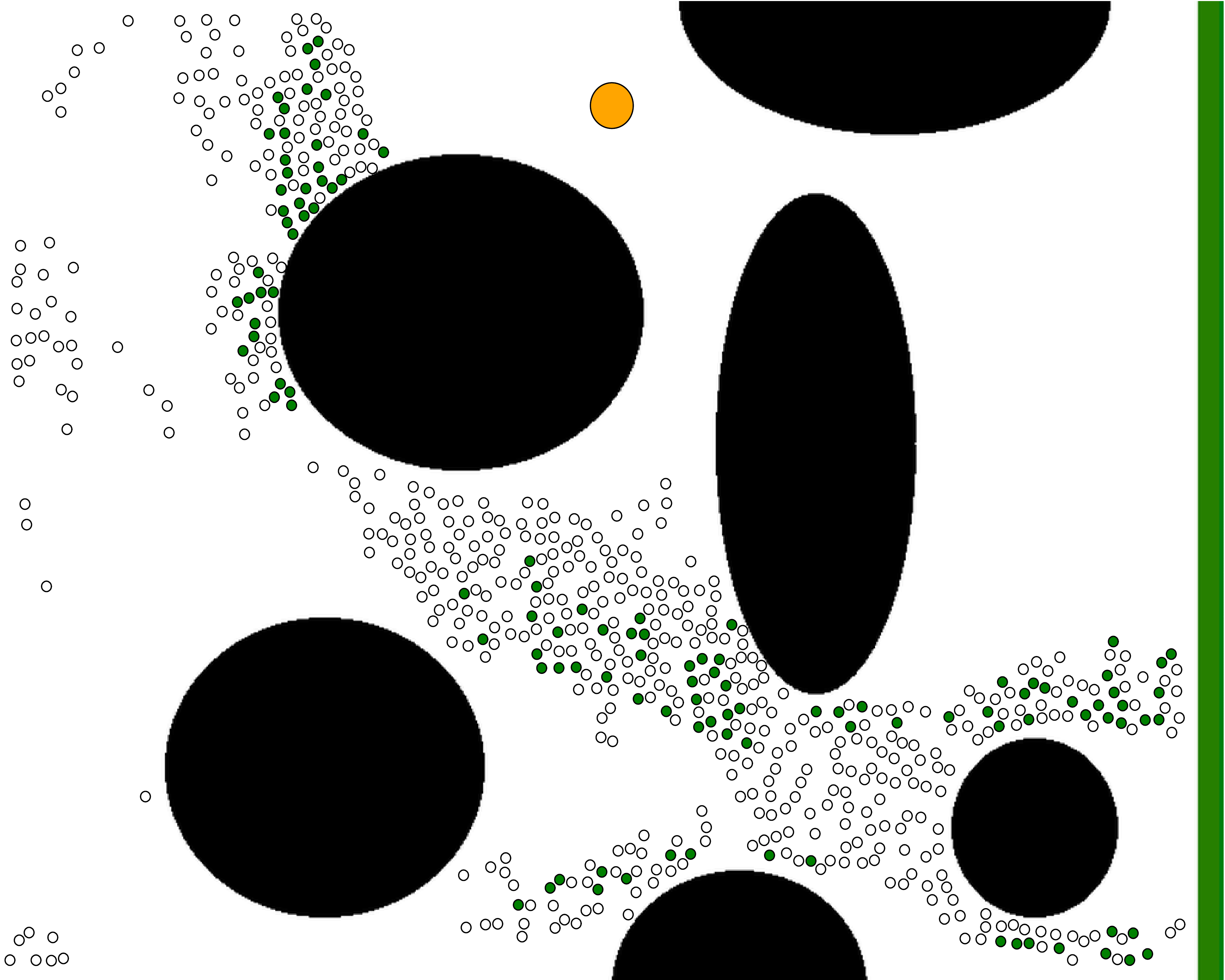}
     \caption{Snapshot after $t=37.5$ for a population with 20\% active agents. Active agents are displayed as filled circles, passive agents as open circles. The larger circle represents the center of the fire.}
     \label{fig:passive_configuration}
\end{minipage}%
\hfill
\begin{minipage}{.5\textwidth}
        \centering
        \includegraphics[width=\textwidth]{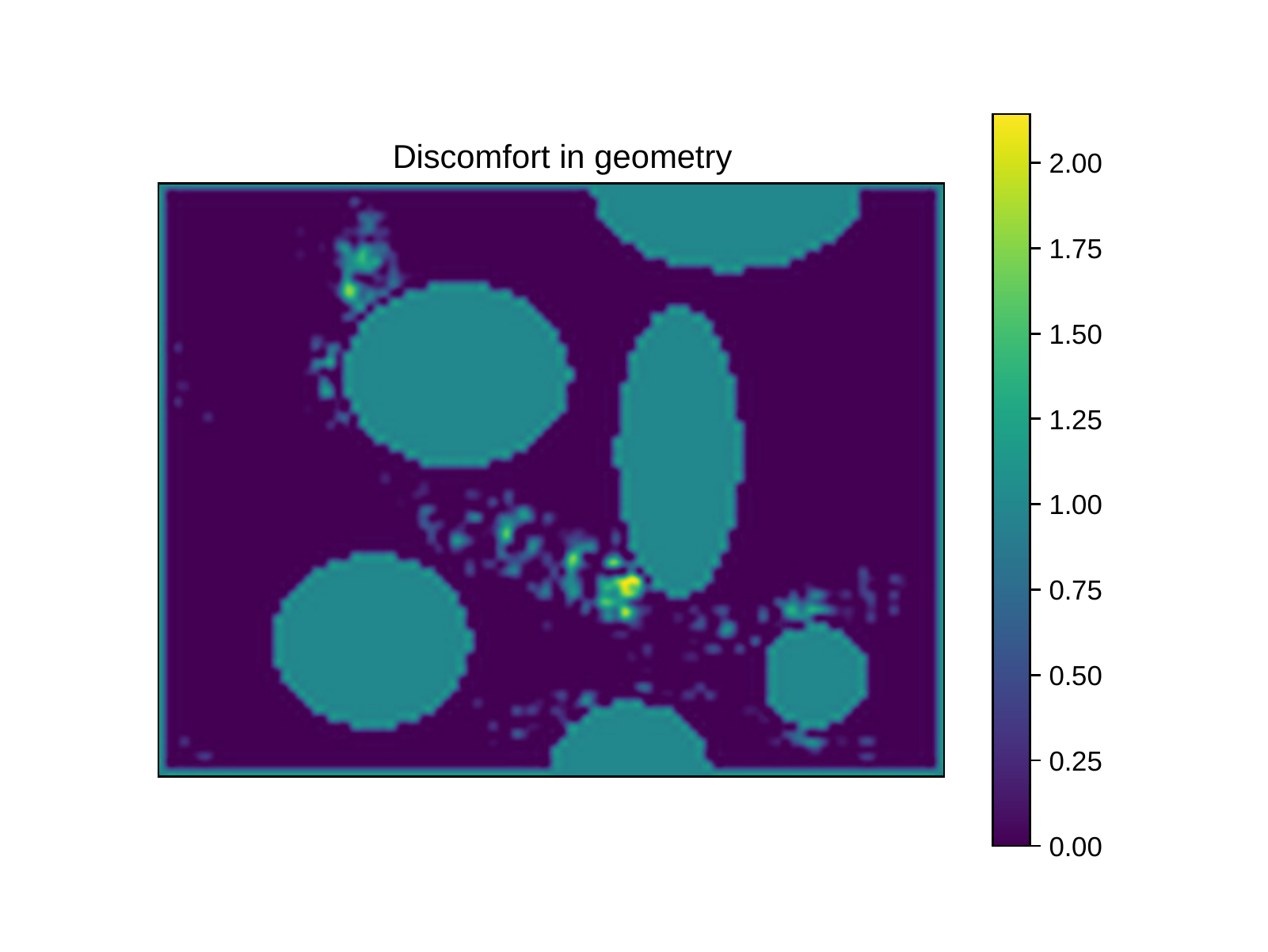}
         \caption{Discomfort observed in the simulation snapshot of Figure~\ref{fig:passive_configuration}.}
         \label{fig:passive_discomfort}
\end{minipage}
\end{figure}

Figure~\ref{fig:active_discomfort} and Figure~\ref{fig:passive_discomfort} illustrate a coherent discomfort field in an ongoing simulation due to a large number of agents with conflicting directions. The corresponding agents configuration (i.e. their spatial distribution) is displayed in Figure~\ref{fig:active_configuration} and Figure~\ref{fig:passive_configuration}. 
Visible is that close to the fire a lot of discomfort is generated. The main cause for this congestion is the conflict between active agents who have identified the location of the fire and want to move in different directions and active agents which are still unaware and want to exit the geometry through that particular corridor. This reminisces of the panic zone that occurs in crowd disasters close to the origin of the panic.
Notice how this zone is much more present in Case 1 than in Case 2, due to the lack of active agents in Case 2. While the passive agents take a lot longer to reach the exit, their following-dominated behaviour amounts to less discomfort in doing so. 

It would be interesting to have a partial differential equation describing at least approximately the macroscopic space-time evolution of such discomfort field available. Also, such an object would be very useful from a practical point of view -- it would allow a fast detection of zones of high discomfort, which could be helpful in taking management decisions to reduce the potential of risks and accidents. 

In Case 1 (Figure~\ref{fig:active_configuration}), we observe that all agents belong to a collective moving towards the exit, regardless of population. 
As one would expect, Case 2 (Figure~\ref{fig:passive_configuration}) displays less order than Case 1. Most agents move in smaller groups, either guided by active agents or randomly moving throughout the geometry. \\

Figure~\ref{fig:active_evac_times} and Figure~\ref{fig:passive_evac_times} depict the agents leaving the environment as a function of time. 
In Figure~\ref{fig:active_evac_times} we observe three stages: the first stage (from $t=0$ to $t\approx 100$) corresponds to the group of active agents that exit the geometry without any obstructions, guiding most of the passive agents while doing so. The second stage (from $t\approx 100$ to $t\approx 400$) has virtually no agents that reach the exit; all the remaining agents are trapped in the high discomfort panic-like zone close to the fire. The third stage shows the final active agents have escaped the panic zone, reaching the exit.

Figure~\ref{fig:passive_evac_times} displays a similar first stage, but because the discomfort zones are a lot less intensive, there is no pronounced second and third stage.
Notice how after the bulk of the active agents have left the geometry, the egress of the passive agents has reduced to a random walk. \\
\begin{figure}[ht]
\centering
\begin{minipage}{.45\textwidth}
    \centering
    \includegraphics[width=\textwidth]{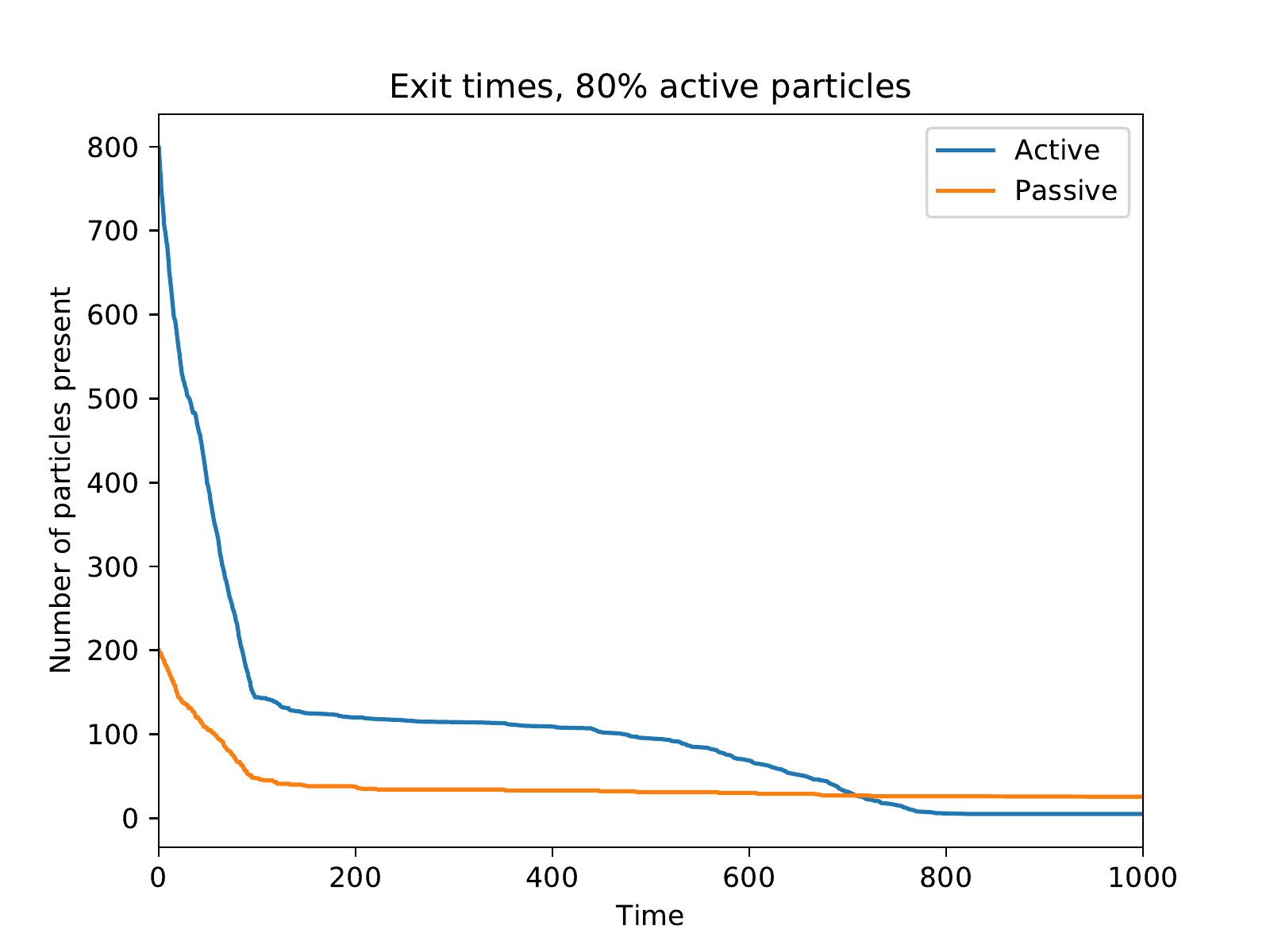}
    \caption{Agent exit times in Case 1.}
    \label{fig:active_evac_times}
\end{minipage}%
\hfill
\begin{minipage}{.45\textwidth}
    \centering
    \includegraphics[width=\textwidth]{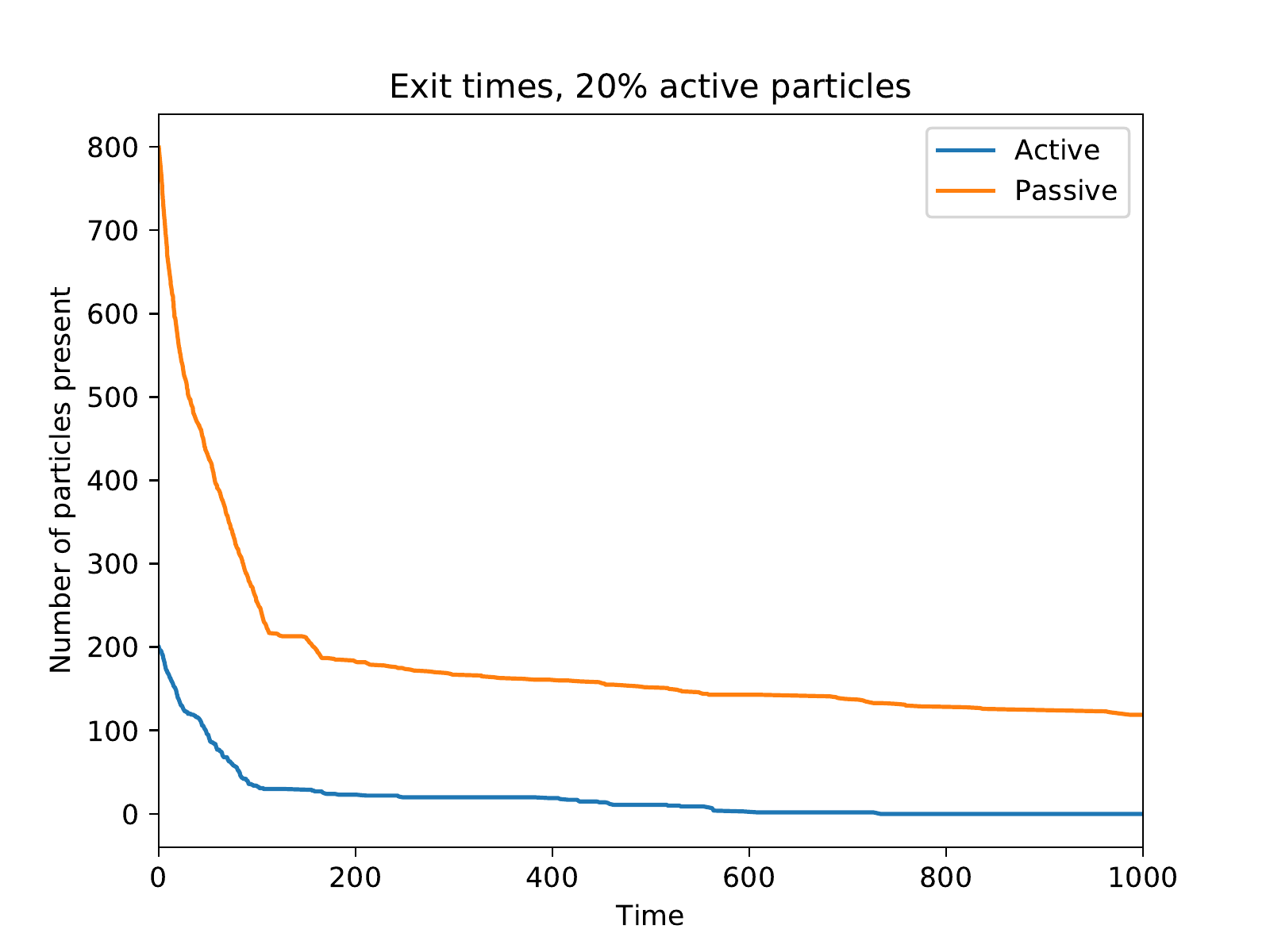}
    \caption{Agent exit times in Case 2.}
    \label{fig:passive_evac_times}
\end{minipage}

\end{figure}

These observations are supported by Figure~\ref{fig:active_cum_discomfort} and Figure~\ref{fig:passive_cum_discomfort}, where the cumulative discomfort for each location $\vec{x}$ in $\Omega$ is displayed.
Case 1 (Figure~\ref{fig:active_cum_discomfort}) shows significantly higher discomfort both near the fire and where the geometry narrows itself.
Case 2 (Figure~\ref{fig:passive_cum_discomfort}) shows a much higher usage of the space in the geometry, i.e. agents walking in locations that do not belong to any shortest path. However, the high discomfort zones are an order of magnitude lower than in Case 1, due to the \textquoteleft flexibility\textquoteright \ of passive agents.
\begin{figure}[ht]
\centering
\begin{minipage}{.45\textwidth}
        \centering
    \includegraphics[width=\textwidth]{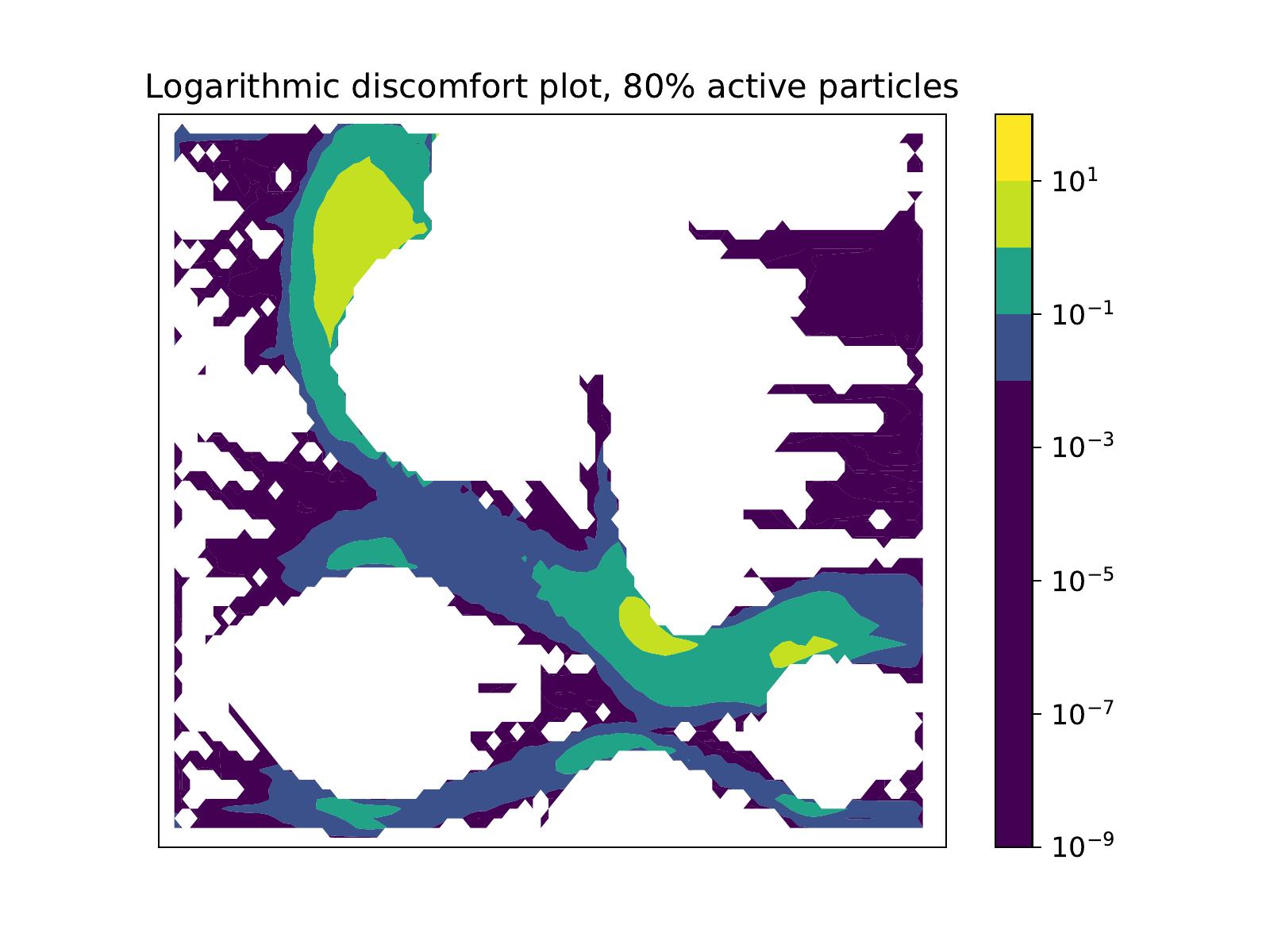}
    \caption{Logarithmic heat-map of discomfort zones that developed in the scenario with 80\% active agents. White regions indicate no experienced discomfort.}
    \label{fig:active_cum_discomfort}
\end{minipage}%
\hfill
\begin{minipage}{.45\textwidth}
        \centering
        \includegraphics[width=\textwidth]{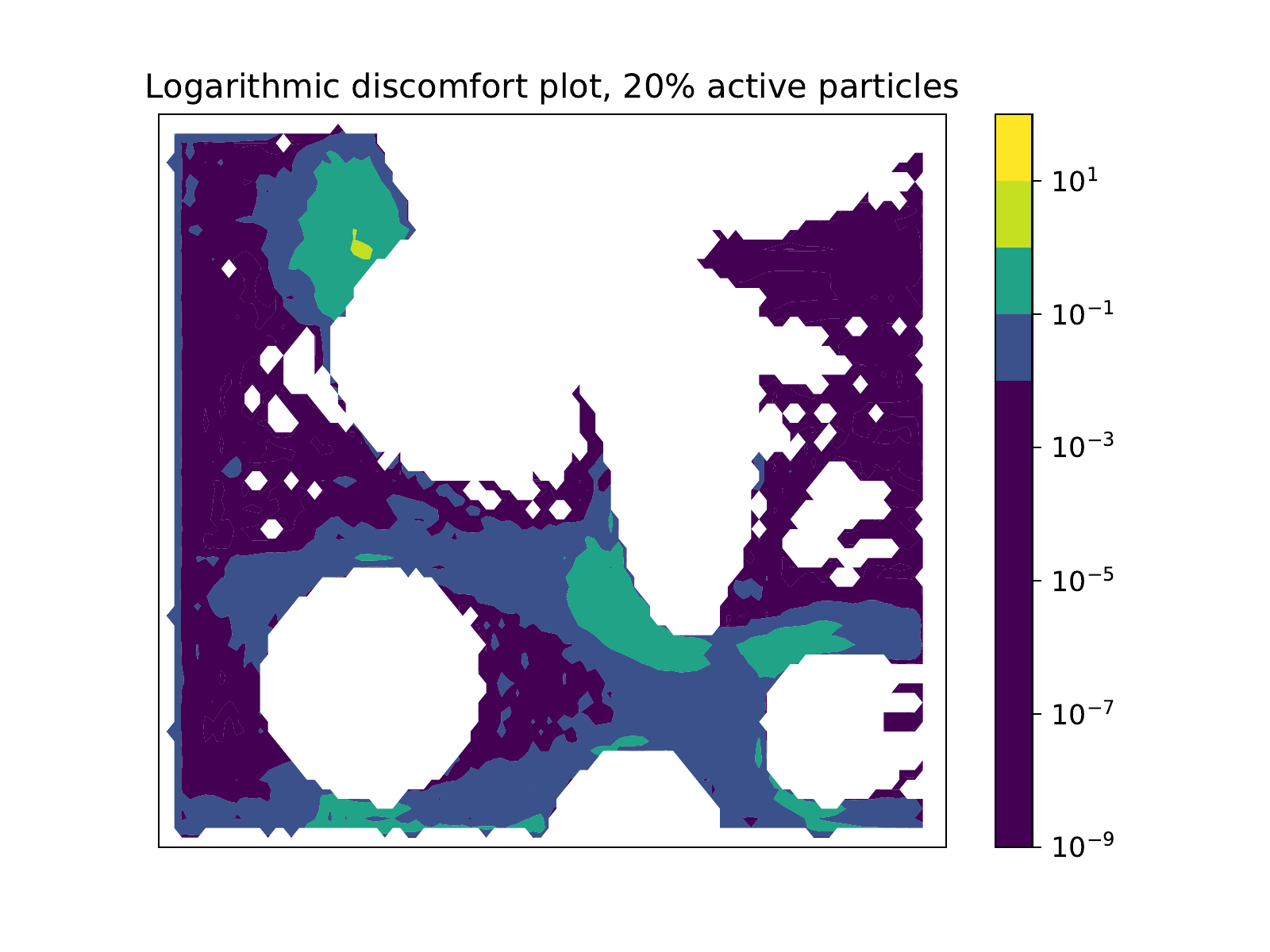}
        \caption{Logarithmic heat-map of discomfort zones that developed in the scenario with 20\% active agents. White regions indicate no experienced discomfort.}
        \label{fig:passive_cum_discomfort}
\end{minipage}
\end{figure}

Concluding, simulations support the following observations. Differences in environment knowledge can have a significant impact on several aspects on the dynamics of crowds in e.g. evacuations.
While it is true that additional knowledge decreases evacuation time, the autonomy of active agents can cause problems when their information turns out to be incorrect. When steering passive agents it is significantly more difficult to maintain order in the evacuation, but the fact that they can be guided can relieve discomfort and reduce congestion.

\clearpage

\section{Lattice gas dynamics (Model 2)}
\label{sec:model_2}

The second model we shall tackle here is a lattice gas model. Namely, we consider a Simple Exclusion Process (SEP) \cite{kipnis1998} on a two-dimensional lattice $\Lambda$:
\begin{equation}
    \Lambda=\{(i,j)\in \mathbb{Z}^2 :i=1,...,L_x \; \text{and} \; j=1,...,L_y\} \;.
\end{equation}
According to the basic tenets of the SEP dynamics, there can be only one particle per site, and particles jump independently towards one of the nearest neighbor sites on the lattice, provided that the arrival site is empty. 
We shall hereafter assume that the system is \textit{closed}, namely particles may not hop outwards from any of the boundary sites of $\Lambda$, except from a subset of lattice sites $\mathcal{D}=\{(i,j) \in \Lambda : j=L_y \; \text{and} \; i\in[i_{ex},i_{ex}+w_{ex}]\}$, called the ``exit door'': any particle located in $\mathcal{D}$ and hopping upwards is annihilated. Note that particles may just leave the system through the exit door: no inward flux of particles is considered in this model. The numerical investigation of the lattice gas model aims, indeed, to shed light on the characteristic time scales characterizing the particle evacuation from the system. \\
As in the case of Model 1, we distinguish between two species of particles, namely \textit{aware} or active particles,  and, respectively, \textit{unaware} or passive particles.  For simplicity of the notation, we shall refer to them as particles ``$A$'' and ``$U$'' in this section. While the species $U$ performs a symmetric simple exclusion dynamics on the lattice, particles of the species $A$ experience both a horizontal and a vertical drift, denoted below as $\epsilon_x$ and $\epsilon_y$, that enhance the rates at which particles of such species hop towards the exit door. The microscopic dynamics is defined as follows.
Call $\eta^{(U)}(x)$ and $\eta^{(A)}(x)$ the occupation number on the site $x$ (which is either $0$ or $1$) of the species, respectively, $U$ and $A$. Given two nearest neighbor sites $x,y \in \Lambda$, $|x-y|=1$, such that the \textit{bond} joining $x$ to $y$ is entirely contained in $\Lambda$, we define the hopping \textit{rate} from $x$ to $y$ of a particle of the species $U$ (no matter if the jump occurs along a horizontal or a vertical bond) as:
\begin{equation}
   c^{(U)}(x,y)=\eta^{(U)}(x)\left[1-\eta^{(U)}(y)-\eta^{(A)}(y)\right] \;.
   \label{cU}
\end{equation}
To define the corresponding hopping \textit{rate} for  particles of the species $A$ we shall distinguish between vertical and horizontal bonds. For the vertical bonds, i.e. when $i_y=i_x$, we set:
\begin{equation}
c^{(A)}(x,y)=\left\{\begin{array}{cc} (1+\epsilon_y)\eta^{(A)}(x)\left[1-\eta^{(A)}(y)-\eta^{(U)}(y)\right] & \quad \hbox{if $j_y>j_x$}\; ,\\ \\
0 & \quad \hbox{if $j_y<j_x$}\;.
  \label{cA1}
  \end{array}\right. 
\end{equation}
for bonds directed upwards and downwards, respectively. The microscopic dynamics of the species $A$, ruled by Eq. \eqref{cA1}, highlights the intrinsic bias of the species $A$ to move upwards, i.e. towards the exit, and prevents any redundant vertical motion in the opposite direction. Moreover, \eqref{cA1} also includes a drift term $\epsilon_y$ that marks the tendency of the species $A$ to reach the exit door with a rate that is higher than the unitary rate defining the unbiased dynamics of the species $U$, see Eq. \eqref{cU}. For the horizontal bonds, i.e. when $j_y=j_x$, we shall first consider the case $i_x \notin [i_{ex},i_{ex}+w_{ex}]$, for which we set:
\begin{equation}
c^{(A)}(x,y)=\left\{\begin{array}{cc}
(1+\epsilon_x)\eta^{(A)}(x)\left[1-\eta^{(A)}(y)-\eta^{(U)}(y)\right] & \quad \hbox{if $(i_y-i_x) (i_{ex}-i_y)\geq0$}\; ,\\ \\
\eta^{(A)}(x)\left[1-\eta^{(A)}(y)-\eta^{(U)}(y)\right] & \quad \hbox{if $(i_y-i_x) (i_{ex}-i_y)<0$}\;.
               \end{array}\right. 
               \label{cA2}
\end{equation}
The presence of a horizontal drift term $\epsilon_x$, in Eq. \eqref{cA2} reminds us that particles of the species $A$, unlike particles of the species $U$, move preferably towards the right (resp. the left), when the horizontal coordinate of the departure site is $i_x<i_{ex}$ (resp. $i_x>i_{ex}+w_{ex}$).
Instead, when $j_y=j_x$ and $i_x\in[i_{ex},i_{ex}+w_{ex}]$, we set:
\begin{equation}
c^{(A)}(x,y)=0  \;.
\label{cA3}
\end{equation}
Equation \eqref{cA3} says that, when $i_x\in [i_{ex},i_{ex}+w_{ex}]$, particles of the species $A$ may only hop upwards, namely they point directly towards the exit door without wandering along the horizontal direction.\\
The rates associated to those bonds joining any boundary site of $\Lambda$, that is not part of the exit door, to any external site, are all set equal to 0. Finally, the rates $c_{ex}$ associated to the vertical bonds joining a site $x\in \mathcal{D}$ with a site $y\notin \Lambda$, are defined as follows:

\begin{equation}
   c_{ex}(x)=\eta^{(U)}(x)(1-\eta^{(A)}(x))+(1+\epsilon_y)\eta^{(A)}(x)(1-\eta^{(U)}(x))\;.
\end{equation}

The proposed lattice gas model also accounts for the presence of fixed obstacles inside the domain, that correspond to a subset $\Lambda_{ob}\subset \Lambda$ of lattice sites that are inaccessible, represented by the black spots shown in Figure~\ref{fig:config}.
To complete the description of the microscopic dynamics we shall, hence, also set equal to zero all the rates associated to those bonds joining two sites, one (or even both) of them belonging to $ \Lambda_{ob}$.

The study of the evacuation of particles shall be pursued by considering, for each species, the behavior of the number of particles and the particle current (through the exit door) as a function of time. The particle current $J$ is defined as
\begin{equation}
    J(t)=\frac{N(0)-N(t)}{t}\;,
    \label{curr}
\end{equation}
where $N(0)$ and $N(t)$ denote, respectively, the number of particles of a given species at the initial time and at the time $t$. 

\section{Results Model 2: Lattice gas dynamics}
\label{sec:results_2}

This section contains our preliminary numerical results obtained using Model 2.\\
The dynamics was implemented by running a set of Kinetic Monte Carlo (KMC) simulations (\cite{Landau:2005}). KMC methods are notoriously suited to describe transient phenomena, in which physical time plays a crucial role in the microscopic evolution (\cite{Bortz,Voter}).
Note that, denoting by $T\in\mathbb{N}$ the number of time steps considered in the KMC simulation, the physical time $t\in\mathbb{R}$, considered in \eqref{curr}, is obtained as $t=\sum_{k=1}^{T} t_k$, where each $t_k$ (corresponding to the time elapsed between two consecutive particle jumps on the lattice) is an exponentially distributed random variable with a parameter given by the sum of all the rates associated to the lattice bonds, defined in Section \ref{sec:model_2}, see refs. \cite{CM,CCM} for details.

\begin{figure}[ht!]
\centering
\begin{tabular}{lll}
\includegraphics[width = 0.33\textwidth]{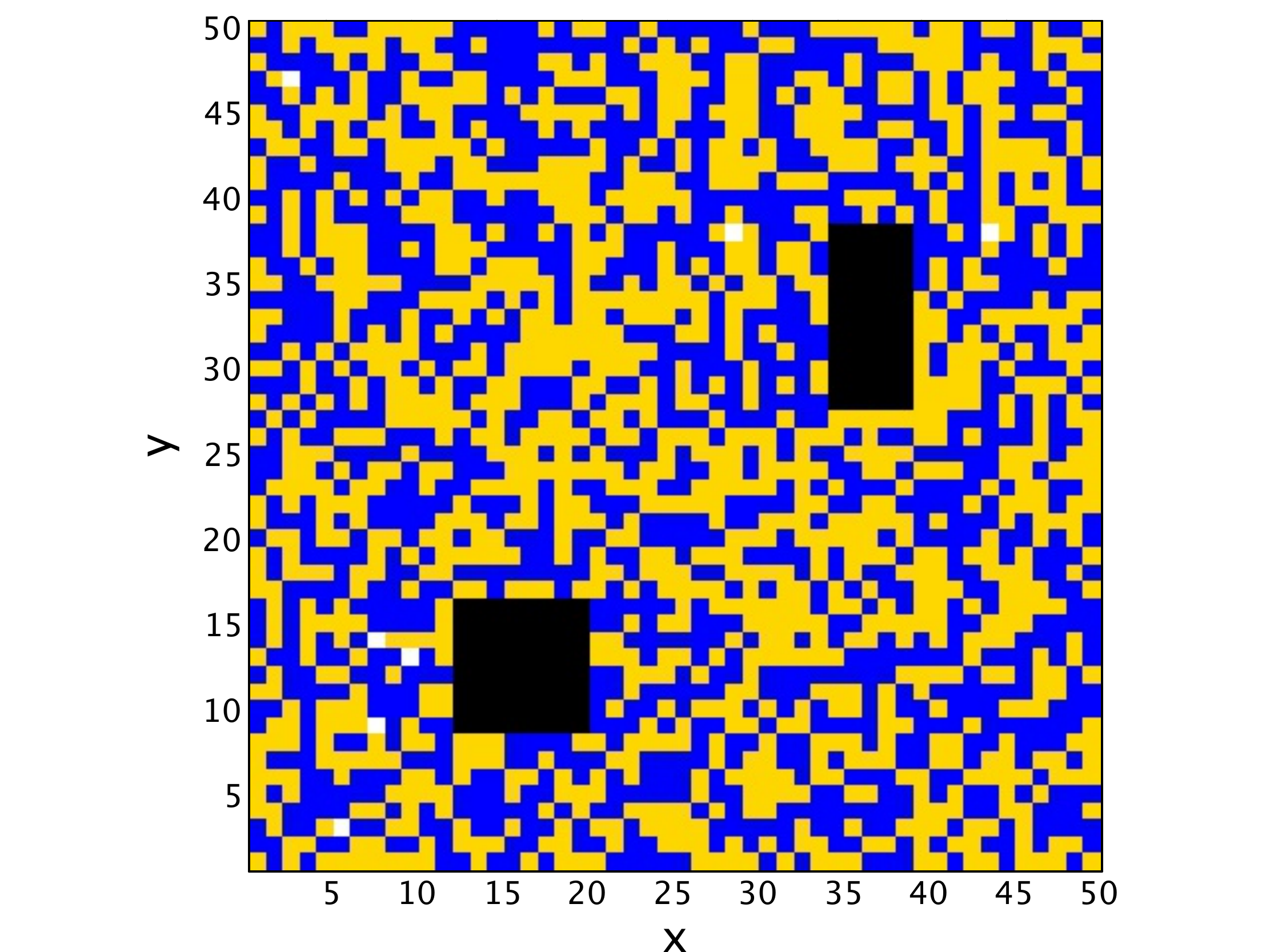} &
\includegraphics[width = 0.33\textwidth]{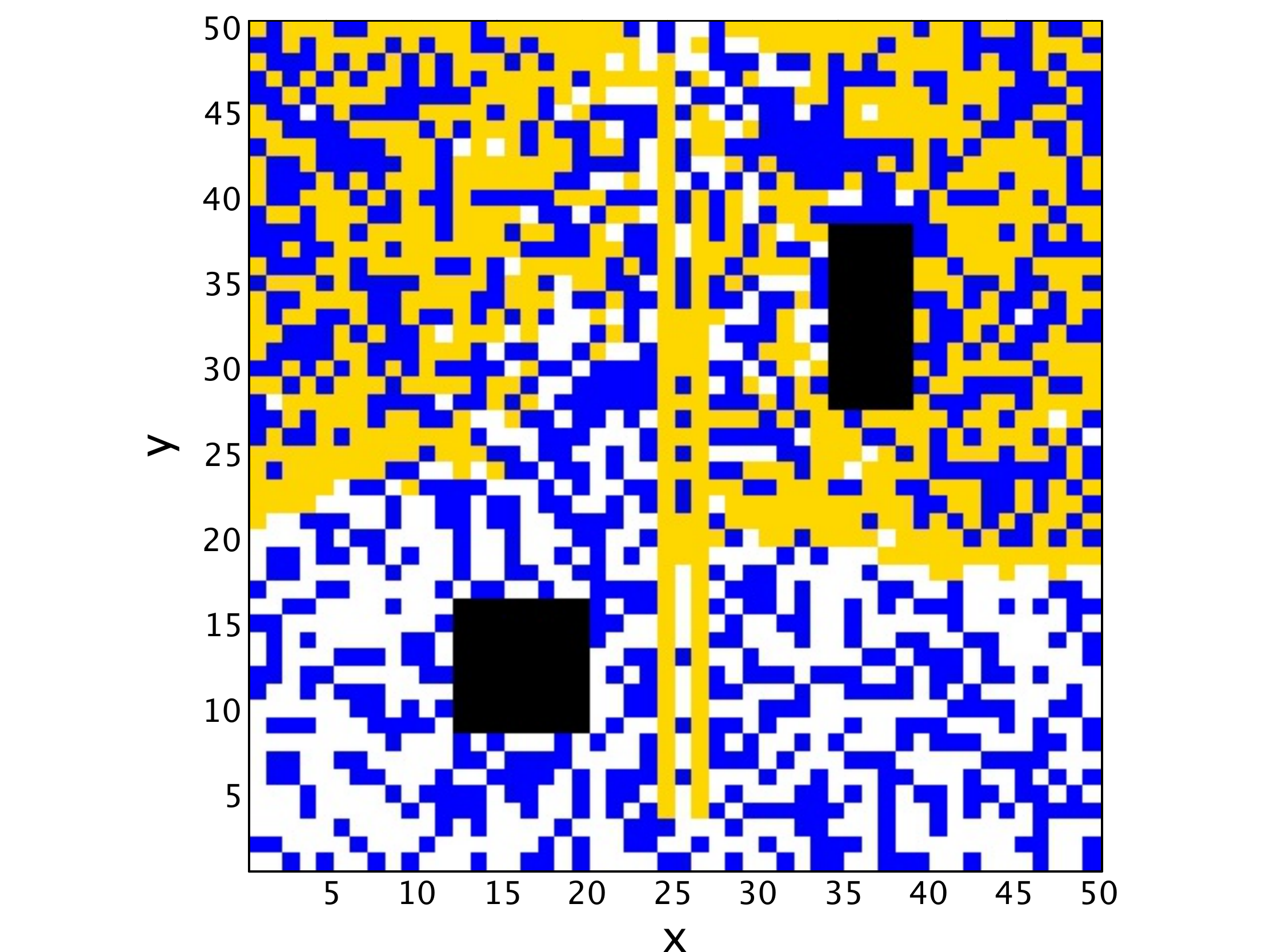} &
\includegraphics[width = 0.33\textwidth]{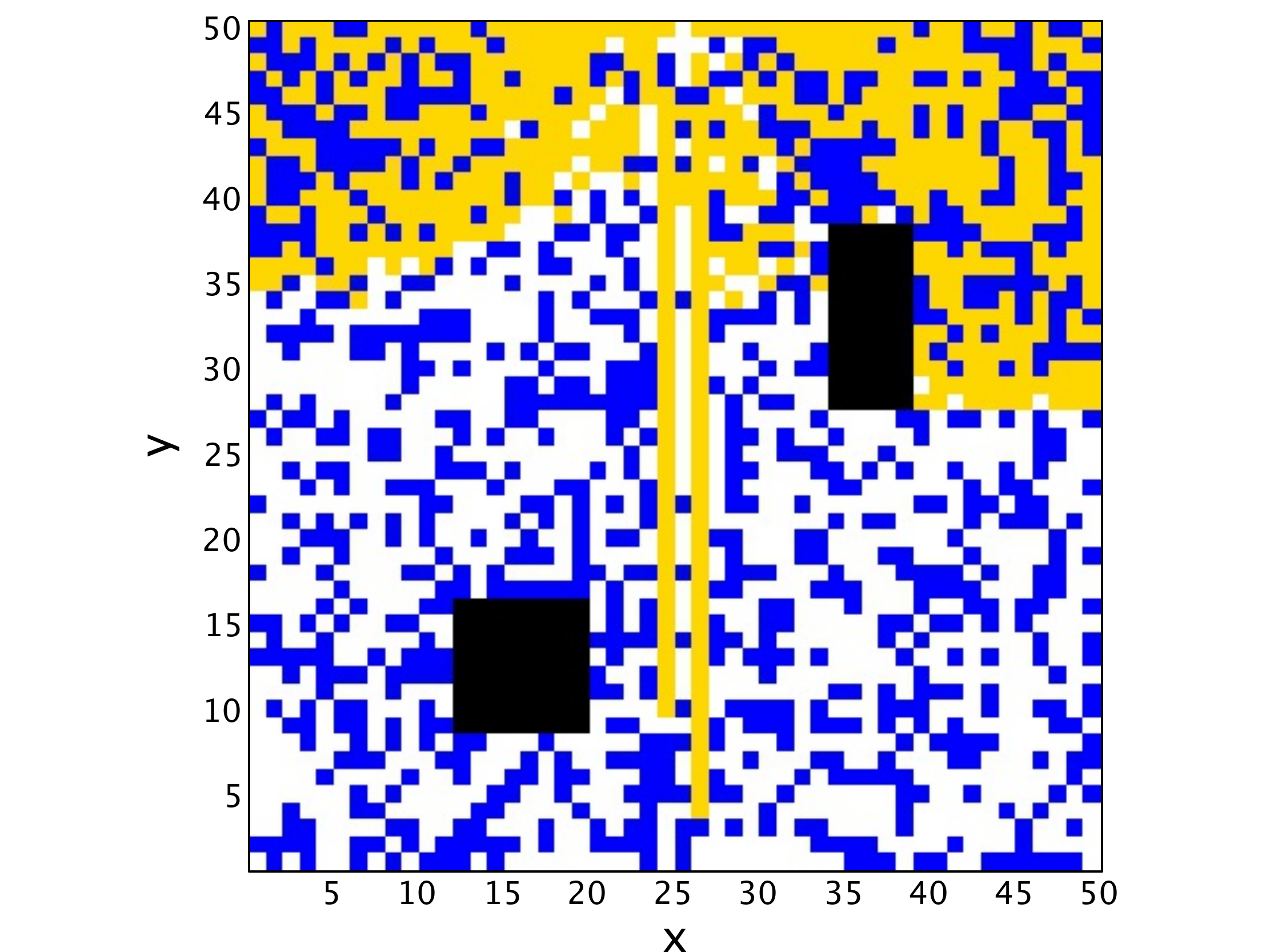}\\[0.2cm]
\includegraphics[width = 0.33\textwidth]{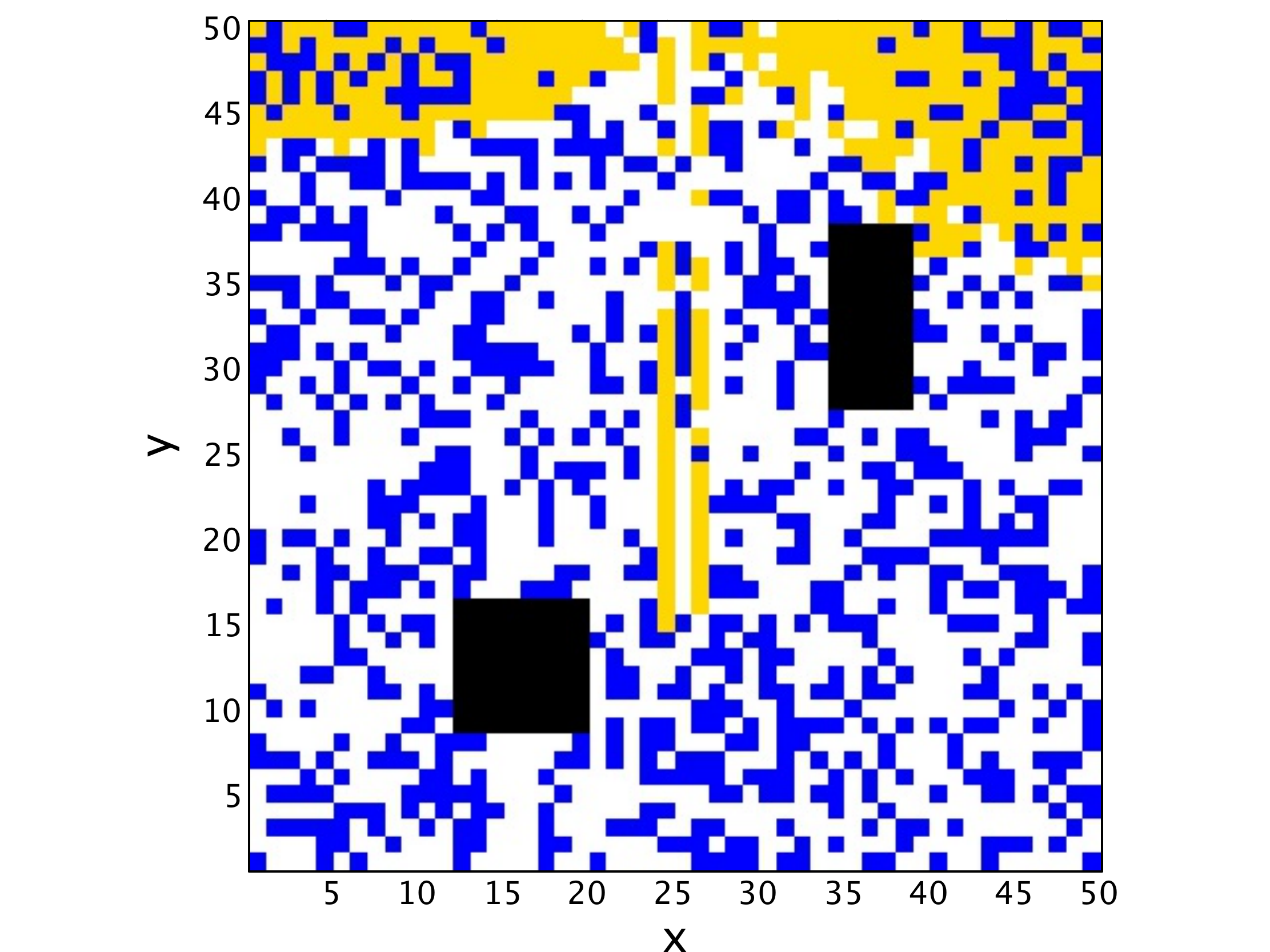} &
\includegraphics[width = 0.33\textwidth]{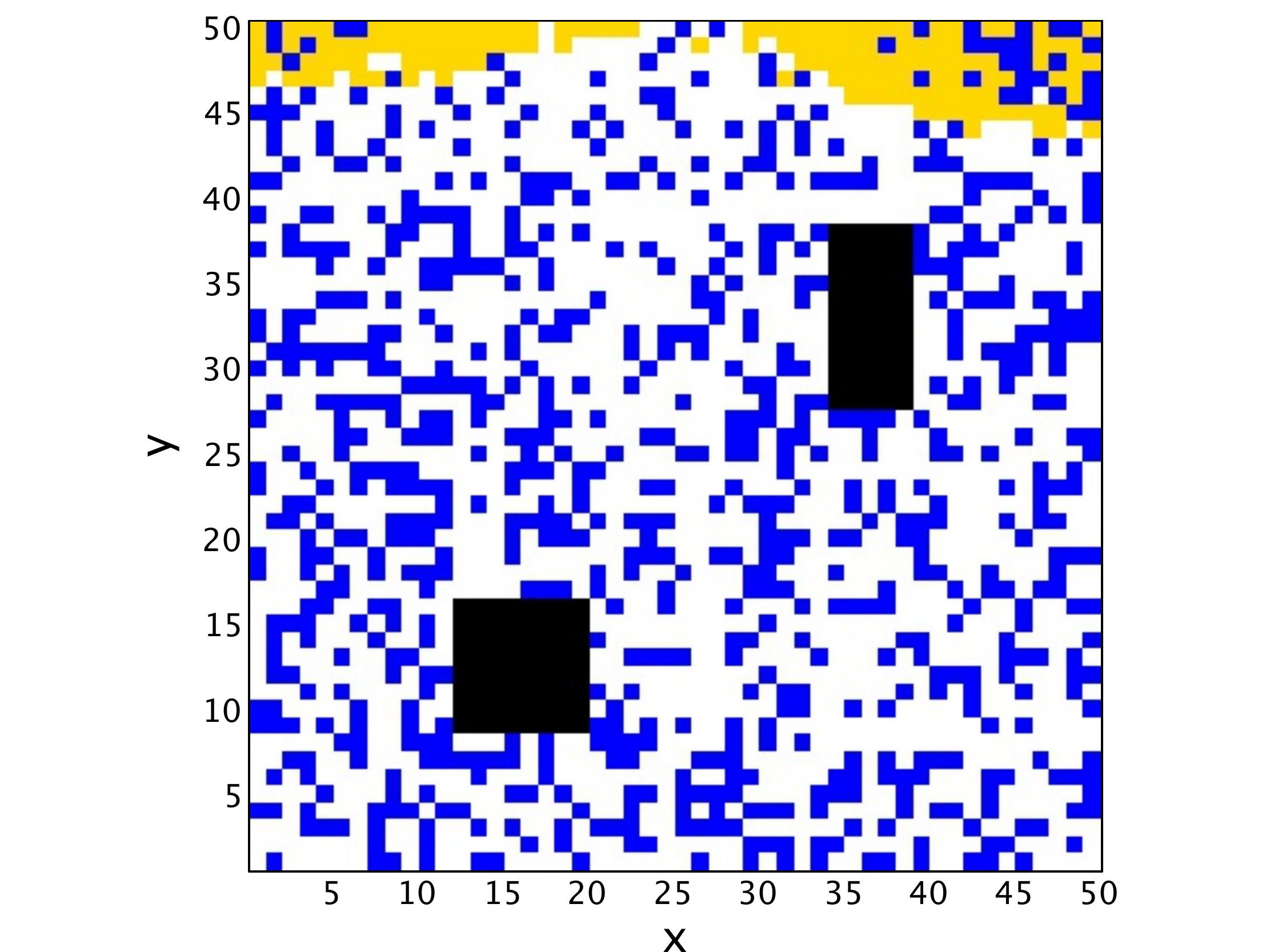}&
\includegraphics[width = 0.33\textwidth]{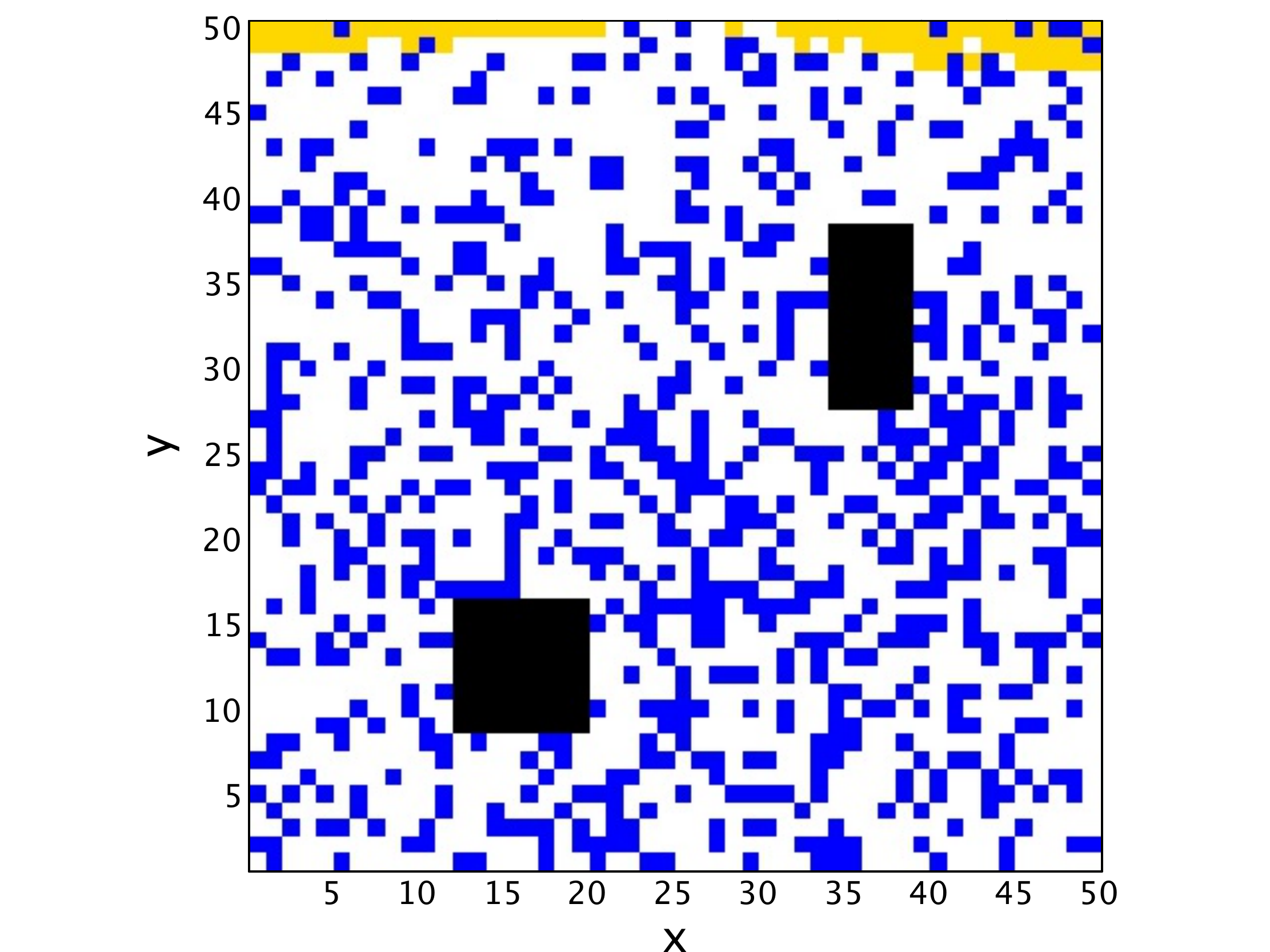}\\[0.2cm]
\includegraphics[width = 0.33\textwidth]{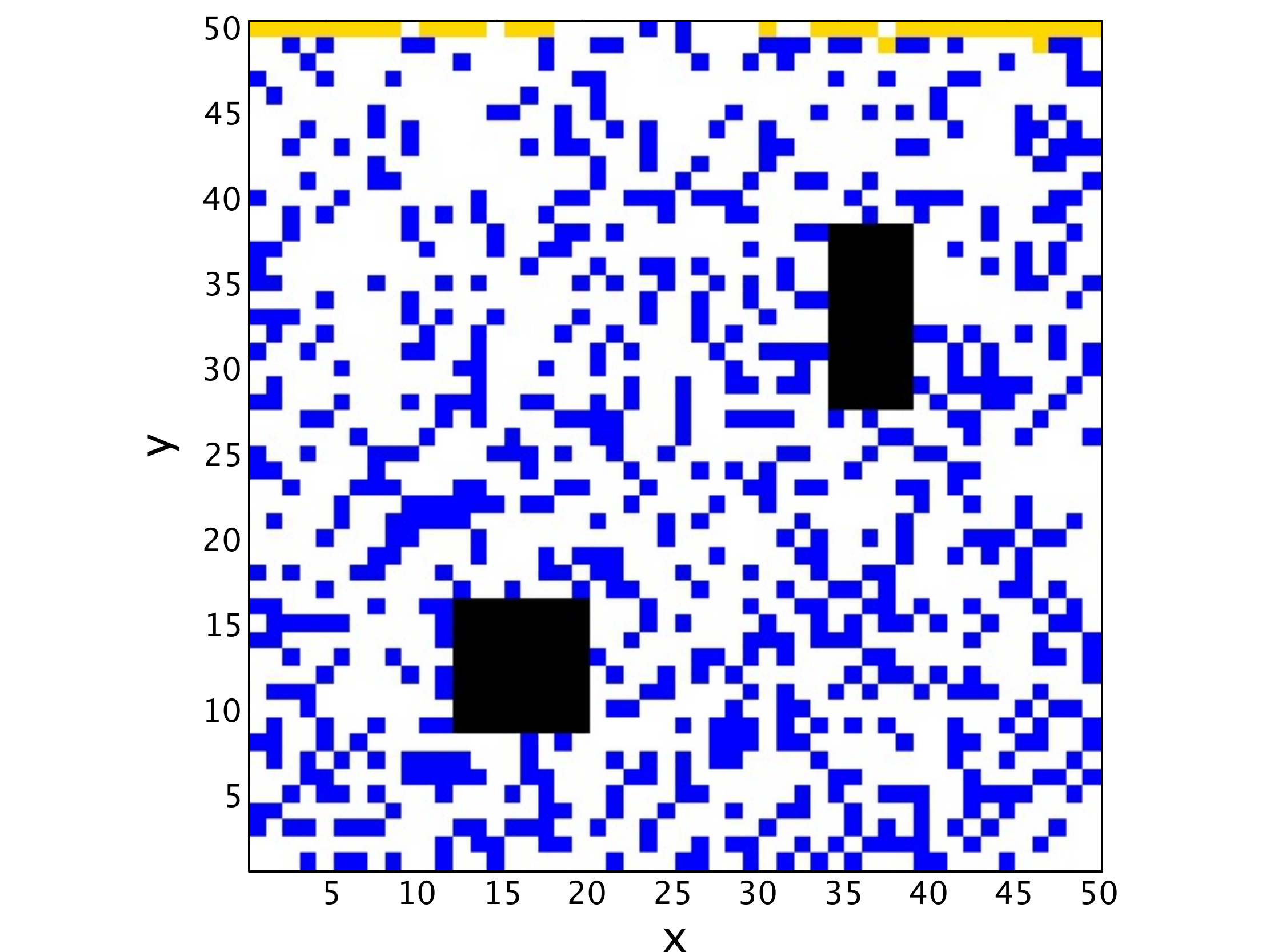} &
\includegraphics[width = 0.33\textwidth]{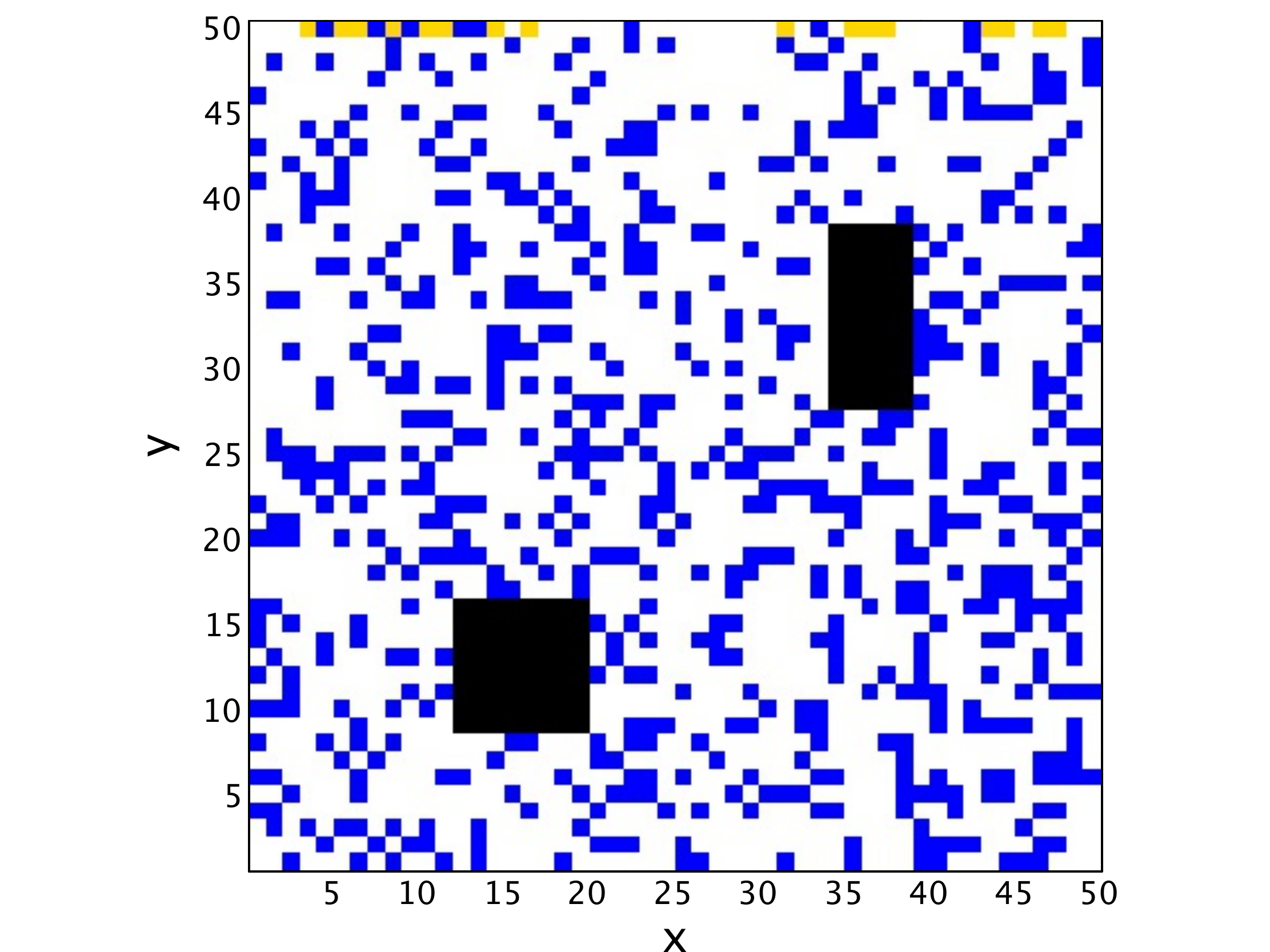} &
\includegraphics[width = 0.33\textwidth]{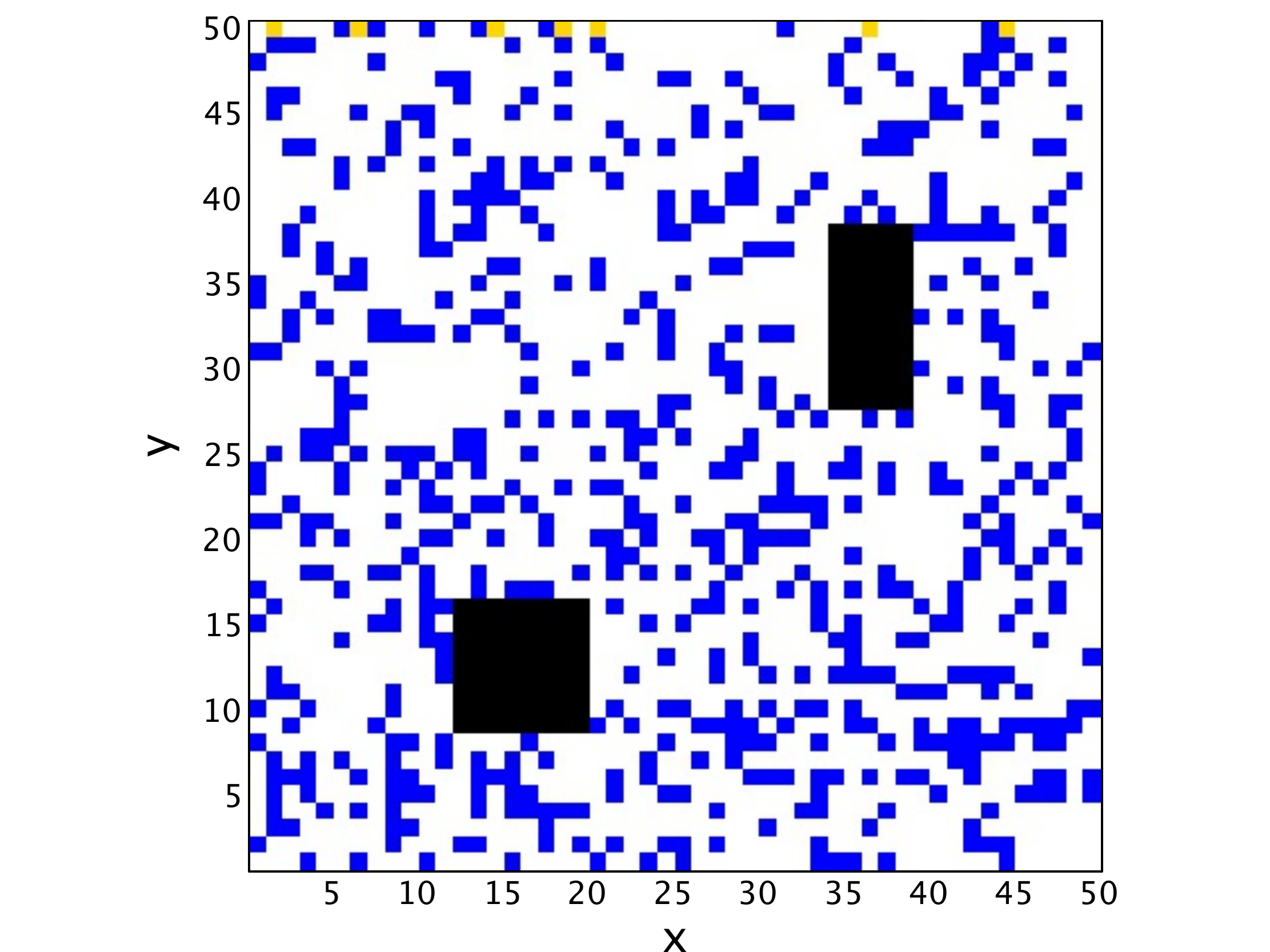}
\end{tabular}
\caption{Microscopic configurations of the particle system sampled at different times (time increases from top to bottom, left to tight). The configurations were obtained from Kinetic Monte Carlo simulations performed on a $2D$ square lattice $\Lambda$ with $L_x=L_y=50$. The yellow pixels correspond to particles of the species $A$, the blue pixels to particles of the species $U$, the white pixels represent empty spots on the lattice, whereas the black pixels denote the fixed obstacles (inaccessible sites). The width of the exit door, located at the center of the first horizontal row, is set equal to $w_{ex}=2$. The horizontal and vertical drifts affecting the dynamics of the species $A$ are set equal to $\epsilon_x=\epsilon_y=0.1$. The initial configuration (top left panel) sees an equal number of particles of the species $A$ and $U$ on the lattice, and the two species together occupy a fraction $98.5\%$ of all the accessible lattice sites.}\label{fig:config}
\end{figure}

The results of the KMC simulations for this model are portrayed in Figure~\ref{fig:config} and Figure~\ref{fig:part}. 
In Figure~\ref{fig:config} we show the microscopic configurations of the particles $A$ and $U$ at different times, namely from the initial configuration (top left panel) until the time when the evacuation of particle $A$ is essentially completed (cf. the bottom right panel, corresponding to some $3\times 10^6$ time steps of the dynamics).

\begin{figure}[ht!]
\centering
\includegraphics[width = 0.48\textwidth]{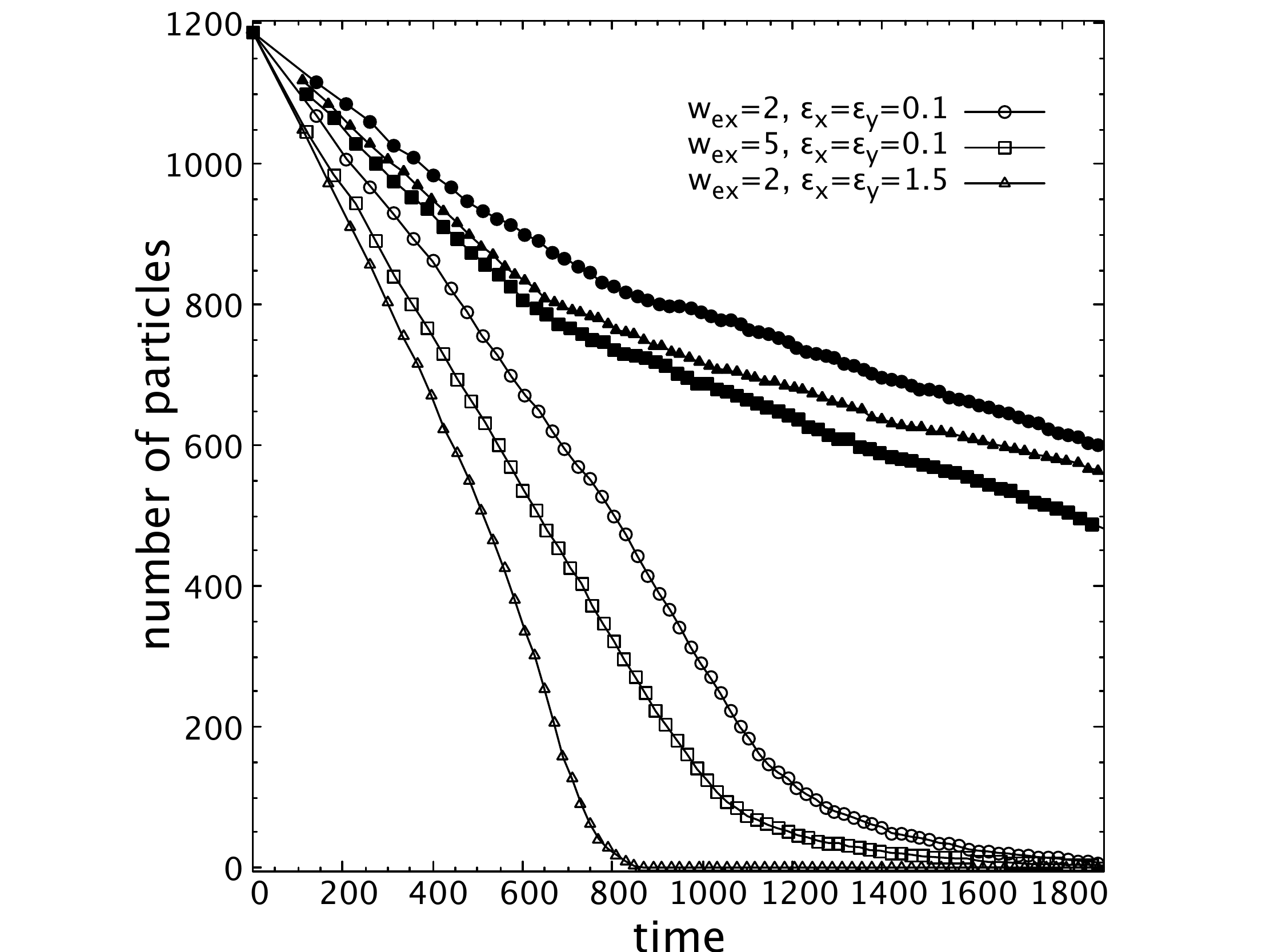} 
\hspace{0.3 cm}
\includegraphics[width = 0.48\textwidth]{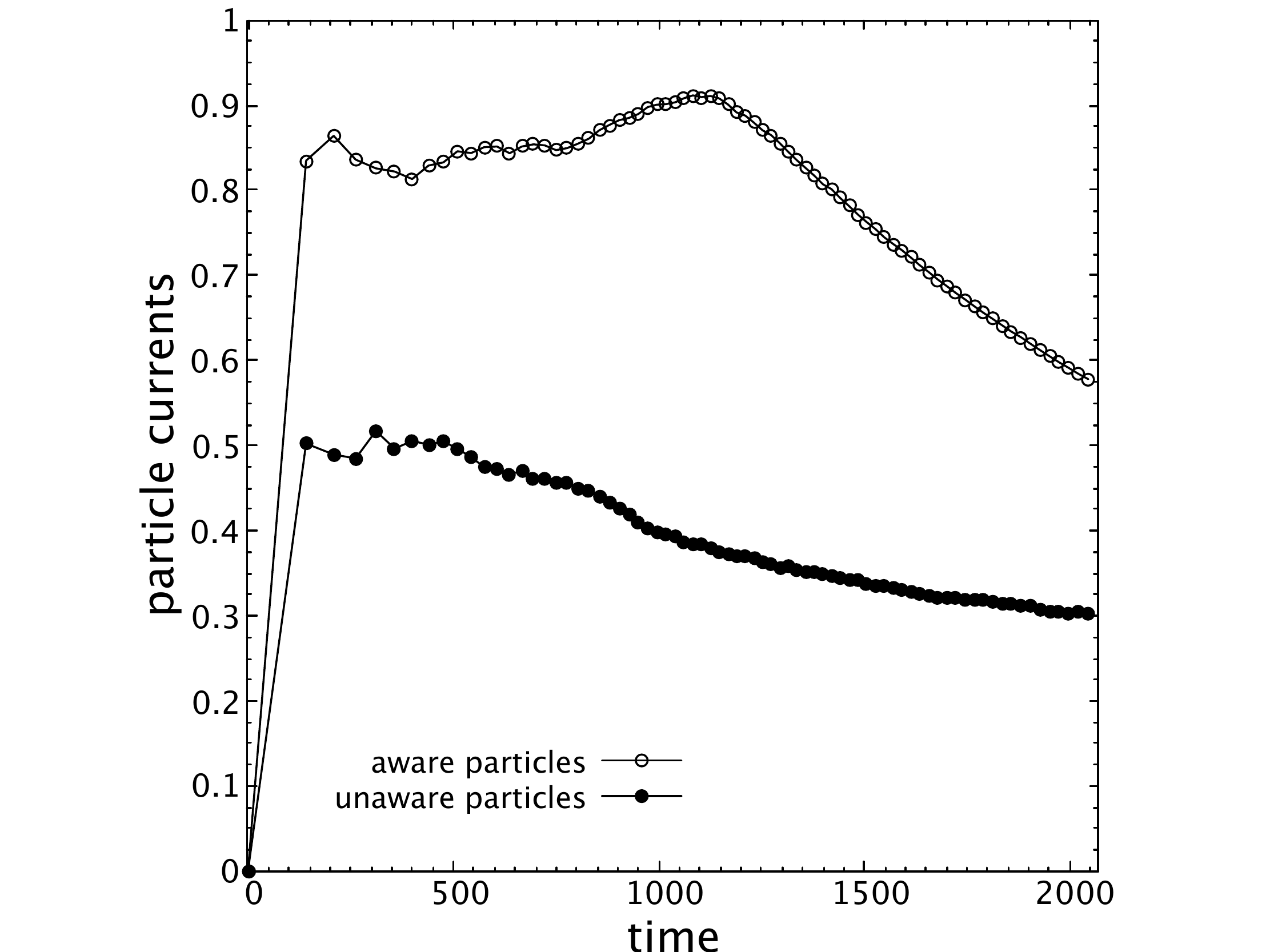} 
\caption{\textit{Left panel:} Behavior of the  number of particles of the species $A$ (empty symbols) and $B$ (filled symbols) vs. time, with $L_x=L_y=50$, and for different values of the exit width $w_{ex}$ and of the drift terms $\epsilon_x$ and $\epsilon_y$. \textit{Right panel:} behavior of the current vs. time for the two species $A$ and $U$ with $L_x=L_y=50$, $w_{ex}=2$ and $\epsilon_x=\epsilon_y=0.1$.}\label{fig:part}
\end{figure}

In the left panel of Figure~\ref{fig:part}, we show the effect of varying the width $w_{ex}$ of the exit door as well as the drifts $\epsilon_x$ and $\epsilon_y$ on the total number of particles $A$ and $U$ as a function of time. Clearly, by increasing the width of the exit door, particles of both species evacuate with a higher pace. The left panel of Figure~\ref{fig:part} also highlights an interesting effect that is obtained by varying $\epsilon_x$ and $\epsilon_y$: an increase of the drift terms induces a higher evacuation rate for particles of the species $A$, and also, consequently, for particles of the species $U$, which have access to a larger number of empty sites on the lattice. In the right panel of Figure~\ref{fig:part} we show the behavior of the particle current, defined in \eqref{curr}, for the two species $A$ and $U$ as a function of time, for fixed values of the parameters $w_{ex}$, $\epsilon_x$ and $\epsilon_y$. The higher evacuation rate observed for the aware particles stems directly from the definition of the rates for the two species $A$ and $U$, given in section \ref{sec:model_2}.

\section{Mathematical aspects of social dynamics in mixed populations}
\label{sec:analysis}
%
%
In this section, we discuss the solvability of a social dynamics model of mixed populations, resembling an overdamped version of Model 1. Note that Model 2 is well-posed by construction. Here, interesting questions would be pointing towards the rigorous derivation of the corresponding hydrodynamic limit equations \cite{DMP}, and/or the numerical evaluation of non-equilibrium collective effects (e.g. the inclusion of a reaction mechanism within the microscopic dynamics, allowing particles to switch from one species to the other, or the presence of long-range interactions between particles), but these aspects are not in our focus for the moment.  

This section contains a couple of technical preliminaries needed to state the evolution problem in a functional analytic framework. We use standard methods to handle the well-posedness of a coupled set of SDEs for the agents dynamics, also linked to a linear parabolic equation describing the motion of the smoke.

\subsection{Technical preliminaries, notation and assumptions}\label{sectio_pre}
\subsubsection{Geometry}

We consider a two dimensional domain, which we refer to as $\tilde{\Omega}$. This domain presents the geometry of the evacuation scenario. In addition, as a building geometry, parts of the domain are filled with obstacles ($G_1$ and $G_2$) denoted by $G := G_1 \cup G_2$ and the fire denoted by $\tilde{F}$. Moreover, the domain has exits denoted by $E$.
Let $\Omega:=\tilde{\Omega}\backslash(G \cup E \cup \tilde{F}) \subset \mathbb{R}^d$ for $d=2$ and $\partial \Omega$ be $C^2$, or at least satisfying the exterior sphere condition. A typical example of such $\Omega$ is depicted in Figure~\ref{fig:example_again}.

\begin{figure}[ht]
    \centering
    \includegraphics[width=0.5\textwidth]{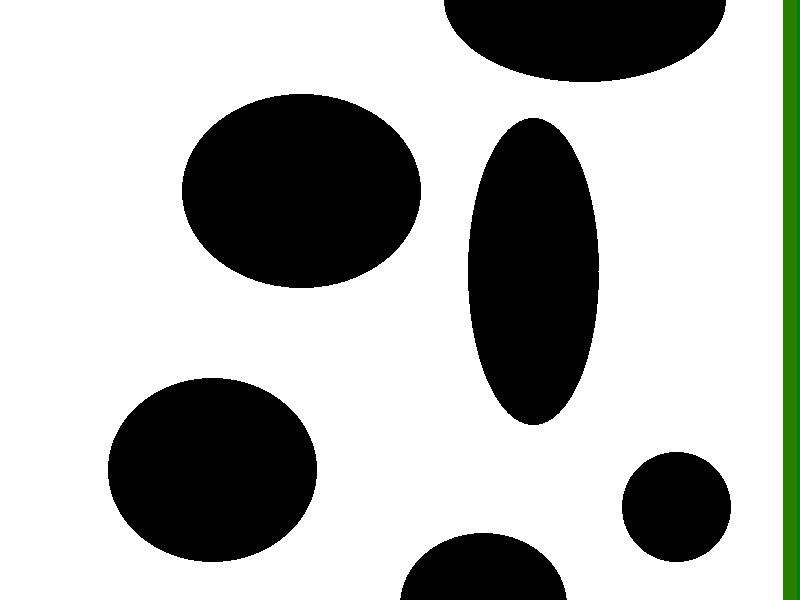}
    \caption{Basic geometry for our case study cf. Model 1. Obstacles are colored black, while the exit is colored green.}
    \label{fig:example_again}
\end{figure}


\subsubsection{Function spaces}

In this section, we employ a number of Sobolev spaces; see e.g.  \cite{Adams03}, \cite{Evans1997} 
for details on their definition and properties. 

The space $H^m(\Omega)$, $m\in \mathbb{N}$, is endowed with the norm
\begin{eqnarray}
\|v\|_{H^m(\Omega)}^2 :=\left(\sum_{|\alpha|\leq m}\int_{\Omega}|D^{\alpha}v|^2dx\right)^{1/2}  \textrm{ for all }v \in H^m(\Omega),\nonumber
\end{eqnarray}
while for the space 
$W^{m,\infty}(\Omega)$ we consider  the norm
\begin{eqnarray}
\|v\|_{m,\infty}(\Omega) := \sum_{|\alpha|\leq k}\textrm{ess} \sup_{\Omega}|D^{\alpha}v| \textrm{ for all } v \in W^{m,\infty}(\Omega), m \in \mathbb{N}; k = 0,\ldots,m.\nonumber
\end{eqnarray}


Our analysis of the stochastic differential equations (SDEs) describing the evolution of our populations follows the line of reasoning from \cite{Flandoli95} and \cite{Prato14} (the compactness method of SPDEs and martingale solutions). We refer to \cite{Evans2013} and \cite{Pavliotis2014} for the basic concepts and usual notations.



Let $\textbf{x}_t$ be a continuous-time stochastic process. We define the family of laws
\begin{eqnarray}\label{laws}
\left\{Q(\textbf{x}_t^n); t \geq 0, n\geq 1\right\}.
\end{eqnarray}
This is a family of probability distribution of $\textbf{x}_t^n$.

Recall the classical Ascoli-Arzel\`{a} theorem:\\
A family of functions $F \subset C([0,T]; \mathbb{R}^d)$ is relatively compact (in uniformly topology) if
\begin{enumerate}
	\item[i.] for every $t \in [0,T]$, the set $\{f(t);f \in F\}$ is bounded.
	\item[ii.] for every $\varepsilon>0$ and $t,s \in [0,T]$ there is $\delta > 0$ such that
	\begin{eqnarray}
	|f(t)-f(s)| \leq \varepsilon,
	\end{eqnarray}
	whenever $|t - s| \leq \delta$ for all $f \subset F$.
\end{enumerate}
We introduce the definition of H\"{o}lder seminorms, for $f: [0,T] \to \mathbb{R}^d$ as
\begin{eqnarray}\label{holder_norm}
[f]_\alpha = \sup_{t\neq s}\frac{|f(t)-f(s)|}{|t-s|}, 
\end{eqnarray}
and the supremum norm as
\begin{eqnarray}\label{infty_norm}
\|f\|_{\infty} = \sup_{t \in [0,T]}|f(t)|.
\end{eqnarray} Using  Ascoli-Arzel\`{a} theorem, starting from the facts:
\begin{enumerate}
	\item[i'.] there is $M_1>0$ such that $\|f\|_{\infty}\leq M_1$ for all $f \in F$.
	\item[ii'.] for some $\alpha \in (0,1)$, there is an $M_2>0$ such that $[f]_{C^\alpha} \leq M_2$ for all $f \in F$,
\end{enumerate}
we infer that the set
\begin{eqnarray}\label{relativelycompactKMP}
K_{M_1M_2} = \left\{f\in C([0,T];\mathbb{R}^d); \|f\|_{\infty}\leq M_1, [f]_{C^\alpha} \leq M_2\right\}
\end{eqnarray}
is relatively compact in $C([0,T]; \mathbb{R}^d)$.

For $\alpha \in (0,1)$, $T>0$ and $p>1$, the space 
$W^{\alpha,p}(0,T;\mathbb{R}^d)$ is defined as the set of all $f \in L^p(0,T;\mathbb{R}^d)$ such that
\begin{eqnarray}
[f]_{W^{\alpha,p}}:= \int_{0}^{T}\int_{0}^{T}\frac{|f(t) - f(s)|^p}{|t-s|^{1 + \alpha p}}dtds < \infty. \nonumber
\end{eqnarray}
This space is endowed with the norm
\begin{eqnarray}
\|f\|_{W^{\alpha,p}} = \|f\|_{L^p} + [f]_{W^{\alpha,p}}.\nonumber
\end{eqnarray}
Moreover, we know that if $\alpha p> 1$, then 
\begin{eqnarray}
W^{\alpha,p}(0,T; \mathbb{R}^d) \subset C^\gamma([0,T];\mathbb{R}^d) \quad \textrm{ for } (\alpha p - \gamma) >1\nonumber
\end{eqnarray}
and $[f]_{C^\gamma} \leq C_{\gamma,\alpha,p}\|f\|_{W^{\alpha,p}}$. Relying on the Ascoli-Arzel\`{a} theorem, we have the following situation:
\begin{enumerate}
	\item[ii''.] for some $\alpha \in (0,1)$ and $p>1$ with $\alpha p>1$, there is $M_2>0$ such that $[f]_{W^{\alpha,p}} \leq M_2$ for all $f \in F$. 
\end{enumerate}
If $\textrm{i'}$ and $\textrm{ii''}$ hold, then the set
\begin{eqnarray}\label{relativelycompactKMP_2}
K'_{M_1M_2} = \left\{f \in C([0,T]; \mathbb{R}^d); \|f\|_{\infty} \leq M_1, [f]_{W^{\alpha,p}} \leq M_2\right\}
\end{eqnarray}
is relatively compact in $C([0,T]; \mathbb{R}^d)$, if $\alpha p>1$.

\subsubsection{Hypotheses}

In this framework, we require the following assumptions:
\begin{enumerate}
	\item [($\textrm{A}_1$)] $\phi\in C^2(\Omega)$ (see also section \ref{eiko}).
	\item [($\textrm{A}_2$)] $p_{\max}=N|\Omega| <\infty$ (bounded maximal discomfort).
	\item [($\textrm{A}_3$)] The smoke matrix diffusion coefficient $D = \textbf{D}(x) \in W^{m,\infty}(\Omega)$ satisfies the uniform ellipticity condition, i.e. there exists positive constants $\underline{\theta}, \overline{\theta}$ such that
	$$\underline{\theta}|\xi|^2 \leq \textbf{D}(x)\xi_i\xi_j \leq \overline{\theta}|\xi|^2 \textrm{ for any } \xi \in \Omega.$$
	\item [($\textrm{A}_4$)] The smoke interface exchange coefficient on the boundary of our domain $\lambda:=\Lambda(x) \in W^{m,\infty}(\partial \Omega)$ is such that  there exist positive constants $\underline{\gamma}, \overline{\gamma}$ satisfying 
	$$-\underline{\gamma}|\xi|^2 < \lambda(x)\xi_i\xi_j \leq \overline{\gamma}|\xi|^2 \textrm{ for any } \xi \in \partial \Omega.$$
\end{enumerate}
Changing the functional framework will naturally lead to a reconsideration of these assumptions.

\subsubsection{First-order social agents dynamics}

We focus on the interaction between two groups of pedestrians, one familiar (active agents) and one unfamiliar (passive agents, visitors) with the geometry. To keep the presentation simple, we decide to tackle here the case when both active and passive agents follow a first-order dynamics (overdamped Langevin equations). To this end, we modify the dynamics of the passive agents, deviating this way from Model 1. 


Let $\textbf{x}_{a_i}$ denote the position of the agent $i$ from group $A$ at time $t$. The crowd dynamics in group A is expressed by the first-order differential equation encoding optimal environment knowledge within the domain $\Omega$, viz.
\begin{align}\label{micro_eqn1}
\begin{cases}
\frac{d\textbf{x}_{a_i}(t)}{dt} = -\Upsilon(s(\textbf{x}_{a_i},t))\left(\frac{\nabla \phi(\textbf{x}_{a_i})}{\|\nabla \phi(\textbf{x}_{a_i})\|}\right)(p_{\max} - p(\textbf{x}_{a_i},t)),\\
\textbf{x}_{a_i}(0) = \textbf{x}_{a_i,0},
\end{cases}
\end{align}
where $p_{\max}$ is a discomfort threshold proportional to the overall population size, say $p_{\max} = N|\Omega|$, with $N := N_A + N_B$ and $p(\textbf x,t)$ is the local discomfort (realization of social pressure) so that  
\begin{align}\label{pressure}
p(\textbf x,t)=\int_{\Omega\cap B(\textbf x,\tilde{\delta})}\sum_{j=1}^{N}\delta(z - \textbf{x}_{c_j}(t))dz.
\end{align}
In \eqref{pressure}, $\delta$ is the Dirac (point) measure and $B(\textbf{x},\tilde{\delta})$ is a ball center $\textbf{x}$ with small enough radius $\tilde{\delta}$ such that $\tilde{\delta} >0$. Hence, the discomfort $p(\textbf{x},t)$ represents a finite measure on the bounded set $\Omega \cap B(\textbf{x},\tilde{\delta})$.
In addition, we assume the following structural relation between the smoke extinction and the walking speed:
\begin{align}
\Upsilon(s(\textbf x,t)) = - a s(\textbf x,t) +  b,\nonumber
\end{align}
where $a, b$ are given positive numbers. Note that every member of this group wants to follow the motion path explicitly  given by $\nabla\phi$ (with $\phi$ the potential function solving the Eikonal equation), which minimizes the distance between particle positions $\textbf{x}_{a_i}$ and the exit location $E$.


As mentioned before, concerning the second population, since the agents do not know the geometry, they must rely on the information from their neighbour. The unfamiliarity with the local environment is expressed here by means of a Brownian motion term $B_i$. Moreover, the passive agents like to be stay away the fire -- for this to happen we use a repulsion term pinpointing to the location of the fire source $\nabla H_\epsilon$. Hence, the dynamics is described as a stochastic differential equation as follows
\begin{align}\label{micro_eqn2}
\begin{cases}
\frac{d\textbf{x}_{b_k}(t)}{dt} = \sum_{j = 1}^N\frac{(\textbf{x}_{c_j} - \textbf{x}_{b_k})}{\|\textbf{x}_{c_j} - \textbf{x}_{b_k}\|}w(\hat{\delta},s(\textbf{x}_{b_k},t)) - \nabla H_\epsilon(\textbf{x}_{b_k},t) + \tilde{D}B_k(t),\\
\textbf{x}_{b_k}(0) = \textbf{x}_{b_k,0}.
\end{cases}
\end{align}
Here $\tilde{D}$ is the constant diffusion coefficient matrix, while $\hat{\delta} \sim \|\textbf{x}_{c_j} - \textbf{x}_{b_k}\|$, and $w$ is a weight factor decreasing as a function of distance. They are defined as
\begin{align}\label{weight_factor}
w(x,y) \sim \frac{1}{r_s^2}\exp\left(-\frac{(x-y)^2}{r_s^2}\right).
\end{align}
In \eqref{weight_factor}, $r_s$ is the sight radius in the evacuees location. Since $H$ is in general not differentiable everywhere (cf. e.g. \eqref{H_def}), in order to be able to take the gradient of $H$, we consider from the start a mollified $H$, say $H_\epsilon$.
 Furthermore, note that $\tilde{D}$ can in principle also depend on the space position. This way the random effects can be skipped in the regions where the geometry is not available for walking. It is worth noting that we have many ways to express how the active agents sense the fire. We choose here to introduce the fire location as a region to be avoided and impose it in the definition domain of the Eikonal equation. It is worth comparing this model for the evolution of the passive agents and the one prescribed in Model 1. Notice here the following important aspects: not only the dynamics is over-damped, but also the expression of the social velocity is slightly adapted to avoid an implicit definition. 

\subsection{Well-posedness} 

Our evolution system consists of an ODE \eqref{micro_eqn1} coupled to an SDE \eqref{micro_eqn2}. Therefore, due to the randomness incorporated  in the SDE \eqref{micro_eqn2}, the ODE becomes an SDE after coupling. So,  we can consider \eqref{micro_eqn1} and \eqref{micro_eqn2} as a SDE system. Note that this system is one-way coupled with the reaction-diffusion-drift equation describing the smoke evolution. 

For convenience, we rephrase the solution to the system \eqref{micro_eqn1} and \eqref{micro_eqn2} in terms of the vector $\textbf{x}_t^n$ such that
\begin{align}
\textbf x_t^n &= ({\textbf x}_{a_i}^n(t),{\textbf x}_{b_k}^n(t)),\\ F_1({\textbf x}_t^n,t) &:= -\Upsilon(s({\textbf x}_{a_i}^n,t))\frac{\nabla \phi({\textbf x}_{a_i}^n)}{\|\nabla \phi({\textbf x}_{a_i}^n)\|}(p_{\max} - p({\textbf x}_{a_i}^n,t)),\\
F_2({\textbf x}_t^n,t) &:= \sum_{j = 1}^N\frac{({\textbf x}_{c_j}^n - {\textbf x}_{b_k}^n)}{\|{\textbf x}_{c_j}^n - {\textbf x}_{b_k}^n\|}w(\hat{\delta},s({\textbf x}_{b_k}^n,t)) - \nabla H_\epsilon({\textbf x}_{b_k}^n,t).
\end{align}
Furthermore, we set
\begin{align}
b_n(\textbf{x}_t^n,t):= \begin{bmatrix}
F_1(\textbf{x}_t^n,t)\\F_2(\textbf{x}_t^n,t)
\end{bmatrix} \textrm{ and } 
\tilde{\sigma}:=\begin{bmatrix}
\tilde{0}\\\tilde{D}
\end{bmatrix},
\end{align}
with
\begin{align}
\tilde{0}:= \begin{bmatrix}
0 &0\\
0 &0
\end{bmatrix} \quad \text{ and } \tilde{D}:=\begin{bmatrix}
D_{11} & D_{12}\\D_{21} & D_{22}
\end{bmatrix}
\end{align}
and initial datum 
\begin{align}
\textbf{x}_0^n := \begin{bmatrix}
\textbf{x}_{a_i,0}^n\\ \textbf{x}_{b_k,0}^n
\end{bmatrix}.
\end{align}
In this section, we use the compactness method for proving the existence of solutions; we follow the arguments  by G. Da Prato and J. Zabczyk  ($2014$) (cf. \cite{Prato14}, Section $8.3$) and a result of F. Flandoli (1995) (cf. \cite{Flandoli95}) for martingale solutions. The starting point of this argument is based on considering a sequence $\{\textbf{x}_t^n\}$ of solutions of the following stochastic differential equation
\begin{align}\label{approxi_formsde}
\begin{cases}
d\textbf{x}_t^n = b_n(\textbf{x}_t^n,t) dt + \tilde{\sigma} dB_t\\
\textbf{x}_0^n = \textbf{x}_0^n
\end{cases}
\end{align}
To ensure the applicability of the compactness argument, we need the following structural assumptions:
\begin{enumerate}
	\item[($\textrm{A}_5$)] $b_n$ be a consequence of continuous functions and uniformly Lipschitz in $x$.
\item[($\textrm{A}_6$)] $b_n$ be equi-bounded $\|b_n\|_{\infty} \leq C$.
\end{enumerate}

It is not difficult to see that in our case, Assumptions ($\textrm{A}_4$) and ($\textrm{A}_5$) are fulfilled. By $s\in C([0,T];C^1(\Omega))$ from Remark \ref{rm_C1}, we have $\tilde{v}_s$ Lipschitz in $x$. Moreover, by the Assumption $(\text{A}_1)$, we obtain $\nabla\phi$ is Lipschitz for $x\in \Omega$. On the other hand, the term $p_{\max} - p(\textbf{x}_{a_i},t)$ is a finite measure on bounded sets -- it is automatically Lipschitz. These considerations lead to the fact that $F_1$ is Lipschitz in $x\in \Omega$. In addition, by $(\text{A}_2)$ together with taking $H_\epsilon$ (as a mollified $H$) implies that  $\nabla H_\epsilon$ is uniformly Lipschitz in $x\in \Omega$. By the formula \eqref{weight_factor}, the weight factors are Lipschitz in $x \in \Omega$. Thus, $F_2$ inherites the Lipschitz property. Clearly, from these arguments, we obtain not only that $F_1$ and  $F_2$ are Lipschitz, but also that these functions are equibounded $\|F_1\|_{\infty} \leq C$ and $\|F_2\|_{\infty} \leq C$. Hence, we have $b_n$ satisfying both assumptions ($\textrm{A}_4$) and ($\textrm{A}_5$).

The compactness argument proceeds as follows. We begin with solutions $\textbf{x}_t^n, n \in \mathbb{N}$ of the system \eqref{micro_eqn1} and \eqref{micro_eqn2}, describing in \eqref{approxi_formsde}. The construction of these solutions can be investigated on a probability space $(\Omega,\mathcal{F}, P)$ with a filtration $\{\mathcal{F}_t\}$ and a Brownian motion $B(t)$. Next, let $Q^n$ be the laws of $\textbf{x}_t^n$ which is defined cf. \eqref{laws}. Then, by using Prokhorov's theorem, we show that the sequence of laws $\{Q^n(\textbf{x}_t^n)\}$ is weakly convergent to $Q(\textbf{x}_t)$ in $C([0,T]; \mathbb{R}^d)$. Then, by using the \textquotedblleft Skorohod representation Theorem\textquotedblright, the weak convergence is in a new probability space with a new stochastic process, for a new filtration. This leads to some arguments for weak convergence results of two stochastic processes in two different probability spaces that we need to use to obtain the existence of our SDE system. Finally, we prove the uniqueness of solutions to our system. 

Let us start with handling the tightness of the laws $\{Q^n\}$ through the following lemma.
\begin{lemma}\label{tightness}
Assume ($\textrm{A}_4$) and ($\textrm{A}_5$) hold.	The family of $\{Q^n\}$ is tight in $C([0,T];\mathbb{R}^d)$
\end{lemma}
\begin{proof}
	In order to prove the tightness, let us recall the following compact sets $K_{M,P}$ (as in the preliminaries section \ref{sectio_pre}):
	\begin{align}
	K_{M_1M_2}=\left\{f \in C([0,T];\mathbb{R}^d); \|f\|_{\infty} \leq M_1, [f]_{C^{\alpha}} \leq M_2\right\}\nonumber
	\end{align}	
	Now, we will show that for a given $\varepsilon >0$, there are $M_1, M_2 > 0$ such that
	\begin{align}
	P(\textbf{x}_{\cdot}^n \in K_{M_1M_2}^c) < \varepsilon, \text{ for all } n\in \mathbb{N}.\nonumber
	\end{align}
	This means that 
	\begin{align}
	P(\|\textbf{x}_{\cdot}^n\|_{\infty} > M_1 \text{ or } [\textbf{x}_{\cdot}^n]_{C^{\alpha}} > M_2) < \varepsilon.\nonumber
	\end{align}
	A sufficient condition is
	\begin{align}\label{pair1}
	P(\|\textbf{x}_{\cdot}^n\|_{\infty} > M_1) < \frac{\varepsilon}{2} \text{ and } P([\textbf{x}_{\cdot}^n]_{C^{\alpha}} > M_2) < \frac{\varepsilon}{2}.
	\end{align}
	Now, we consider the first one $P(\|\textbf{x}_{\cdot}^n\|_{\infty} > M_1) < \frac{\varepsilon}{2}$. Using Markov's inequality (cf. \cite{Jacod2004}, Corollary 5.1), we get
	\begin{align}
	P(\|\textbf{x}_{\cdot}^n\|_{\infty} > M_1) \leq \frac{1}{M_1}E\left[\sup_{t \in [0,T]}\left|\textbf{x}_t^n\right|\right],\nonumber
	\end{align}
	but
	\begin{align}
	\sup_{t \in [0,T]}\left|\textbf{x}_t^n\right| = \sup_{t \in [0,T]}\left\{\left|\textbf{x}_{a_i,0}^n + \int_{0}^{t}F_1(\textbf{x}_y^n,y)dy\right|,\left|\textbf{x}_{b_k,0}^n + \int_{0}^{t}F_2(\textbf{x}_y^n,y)dy+\int_{0}^{t}\tilde{\sigma} dB_y\right|\right\}.\nonumber
	\end{align}
	We estimate
	\begin{align}
	\sup_{t \in [0,T]}\left|\textbf{x}_t^n\right| \leq \sup_{t \in [0,T]}\left\{|\textbf{x}_{a_i,0}^n| + \left|\int_{0}^{t}F_1(\textbf{x}_y^n,y)dy\right|,|\textbf{x}_{b_k,0}^n|+ \left|\int_{0}^{t}F_2(\textbf{x}_y^n,y)dy\right|+\left|\int_{0}^{t}\tilde{\sigma} dB_y\right|\right\}\nonumber
	\end{align}
	Since $F_1,F_2$ bounded, then we have
	\begin{align*}
	\int_{0}^T\left|F_1(\textbf{x}_y^n,y)\right|dy &= \int_{0}^{T}\left|-\Upsilon(s({\textbf x}_{a_i}^n,y))\left(\frac{\nabla \phi({\textbf x}_{a_i}^n)}{\|\nabla \phi({\textbf x}_{a_i}^n)\|}\right)(p_{\max} - p({\textbf x}_{a_i}^n,y))\right|dy \leq C,\\
	\int_{0}^T\left|F_2(\textbf{x}_y^n,y)\right|dy &= \int_{0}^T\left|\sum_{j = 1}^N\frac{({\textbf x}_{c_j}^n - {\textbf x}_{b_k}^n)}{\|{\textbf x}_{c_j}^n - {\textbf x}_{b_k}^n\|}w(\hat{\delta},s({\textbf x}_{b_k}^n,y)) - \nabla H_\epsilon({\textbf x}_{b_k}^n,y)\right|dy \leq C.
	\end{align*}
	Taking the expectation, we have the following estimate
	\begin{align}
	E\left[\sup_{t \in [0,T]}|\textbf{x}_t^n|\right]\leq C.\nonumber
	\end{align}
	Hence, for $\varepsilon > 0$, we can choose $M_1>0$ such that $P(\|\textbf{x}_\cdot^n\|_{\infty} > M_1) < \frac{\varepsilon}{2}$.
	
	From now on, we consider the second inequality in \eqref{pair1}. This  reads
	\begin{align}
	P([\textbf{x}_{\cdot}^n]_{C^{\alpha}} > M_2)=P\left(\sup_{t \neq r} \frac{|\textbf{x}_t^n - \textbf{x}_t^r|}{|t - r|} > M_2\right) \leq \frac{\varepsilon}{2}. \nonumber
	\end{align}
	Let us introduce another class of compact sets now in the Sobolev space $W^{\alpha,p}(0,T; \mathbb{R}^d)$ (which for suitable exponents lies in  $C^{\gamma}([0,T], \mathbb{R}^d)$).
	Additionally, we recall the relatively compact sets $K'_{M_1M_2}$ in \eqref{relativelycompactKMP_2} such that
	\begin{align}
	K'_{M_1M_2} = \left\{f \in C([0,T];\mathbb{R}^d); \|f\|_{\infty} \leq M_1, [f]_{W^{\alpha,p}} \leq M_2 \right\}.\nonumber
	\end{align}
	A sufficient condition for $K'_{M_1M_2}$ to be relatively compact in the underlying space  is $\alpha p > 1$.
	Having this in mind, we wish to prove that there exist $\alpha \in (0,1)$ and $p > 1$ with $\alpha p > 1$ together with the property: given $\varepsilon > 0$, there is $M>0$ such that
	\begin{align}
	P([\textbf{x}_{\cdot}^n]_{W^{\alpha,p}}> M_2) < \frac{\varepsilon}{2},\nonumber
	\end{align}
	for every $n \in \mathbb{N}$. 
	
	Using Markov's inequality, 
	we obtain
	\begin{align}
	P([\textbf{x}_{\cdot}^n]_{W^{\alpha,p}} > M_2) &\leq \frac{1}{M}E\left[\int_0^T\int_0^T \frac{|\textbf{x}_t^n - \textbf{x}_r^n|^p}{|t - r|^{1 + \alpha p}}dtdr\right] \nonumber\\
	&= \frac{C}{M}\int_0^T\int_0^T\frac{E\left[|\textbf{x}_t^n - \textbf{x}_r^n|^p\right]}{|t - r|^{1+\alpha p}}dtdr.\nonumber
	\end{align}
	For $t \geq r$, we have
	\begin{align}
	\textbf{x}_t^n - \textbf{x}_r^n = \begin{pmatrix}
	\int_{r}^{t}F_1(\textbf{x}_y^n,y)dy \\
	\int_{r}^{t}F_2(\textbf{x}_y^n,y)dy
	\end{pmatrix} + \begin{pmatrix}
	0\\ \int_{r}^{t}\tilde{\sigma} dB_y
	\end{pmatrix}.\nonumber
	\end{align}
	Let us introduce some further notation.  For a vector $u = (u_1,u_2)$, we set  $|u| := |u_1|+|u_2|$. At this moment, we consider the following expression:
	\begin{align}\label{modulus_def}
	|\textbf{x}_t^n - \textbf{x}_r^n| = \left|\int_{r}^{t}F_1(\textbf{x}_y^n,y)dy\right| + \left|\int_{r}^{t}F_2(\textbf{x}_y^n,y)dy+\int_{r}^{t}\tilde{\sigma} dB_y\right|.
	\end{align}
	Taking the modulus up to the power $p>1$, \eqref{modulus_def} reads
	\begin{align}\label{modulus_p}
	|\textbf{x}_t^n - \textbf{x}_r^n|^p &= \left( \left|\int_{r}^{t}F_1(\textbf{x}_y^n,y)dy\right| + \left|\int_{r}^{t}F_2(\textbf{x}_y^n,y)dy+\int_{r}^{t}\tilde{\sigma} dB_y\right|\right)^p\nonumber\\
	&\leq \left|\int_{r}^{t}F_1(\textbf{x}_y^n,y)dy\right|^p + \left|\int_{r}^{t}F_2(\textbf{x}_y^n,y)dy\right|^p + \left|\int_{r}^{t}\tilde{\sigma} dB_y\right|^p\nonumber\\
	&\leq \int_{r}^{t}\left|F_1(\textbf{x}_y^n,y)\right|^pdy+\int_{r}^{t}\left|F_2(\textbf{x}_y^n,y)\right|^pdy+\left|\int_{r}^{t}\tilde{\sigma} dB_y\right|^p\nonumber\\
	&\leq C(t-r)^p +\left|\int_{r}^{t}\tilde{\sigma} dB_y\right|^p.
	\end{align}
	Taking the expectation on \eqref{modulus_p}, we obtain the following estimate
	\begin{align}\label{expectation_modu}
	E[|\textbf{x}_t^n - \textbf{x}_r^n|^p] \leq C(t-r)^p +E\left[\left|\int_{r}^{t}\tilde{\sigma} dB_y\right|^p\right].
	\end{align}
	Now, we consider the second term of the right hand side of \eqref{expectation_modu}. By using the Burkholder-Davis-Gundy inequality (cf. \cite{Prato14}, Hypothesis 6.4), we obtain
	\begin{align}\label{Bur_ine}
	E\left[\left|\int_{r}^{t} \tilde{\sigma} dB_y\right|^p\right] \leq CE\left[\left(\int_{r}^{t} dy\right)^{p/2}\right] \leq C(t - r)^{p/2}.
	\end{align} 
	Combining \eqref{expectation_modu} and \eqref{Bur_ine}, we have the upper bound 
	\begin{align}
	E[|\textbf{x}_t^n - \textbf{x}_r^n|^p] \leq C(t - r)^{p/2}.\nonumber 
	\end{align}
	On the other hand, the integral $$ \int_{0}^{T}\int_{0}^{T} \frac{1}{|t - r|^{1 + (\alpha - \frac{1}{2})p}}dtdr$$ is finite if $\alpha < \frac{1}{2}$. Consequently,  we can pick $\alpha < \frac{1}{2}$. Taking now $p>2$ together with the constraint $\alpha p > 1$, we can find $M_2 > 0$ such that 
	\begin{align}
	P\left([\textbf{x}_{\cdot}^n]_{W^{\alpha,p}} > M_2\right) < \frac{\varepsilon}{2}.\nonumber
	\end{align}
	This complete the proof of the Lemma.
\end{proof}
\begin{theorem}
	Assume $(\text{A}_1)$ and $(\text{A}_2)$ hold. There exits a solution of the microscopic dynamics SDE system \eqref{micro_eqn1} and \eqref{micro_eqn2}.
\end{theorem}
\begin{proof}
	From Lemma \ref{tightness}, we have obtained that the sequence $\{Q^n\}$ is tight in $C([0,T];\mathbb{R}^d)$. Applying the Prokhorov's Theorem (cf. \cite{Billingsley1999}, Theorem $5.1$), there are subsequences $\{Q^{n_k}\}$ which converge weakly. For simplicity of the notation, we denote these subsequences by $\{Q^{n}\}$. This means that we have $\{Q^n\}$ converges weakly to some probability measure $Q$ on Borel sets in $C([0,T];\mathbb{R}^d)$.
	
	
	Since we have that $Q(\textbf{x}_t^n)$ converges weakly to $Q(\textbf{x}_t)$, by using the \textquotedblleft Skorohod Representation Theorem\textquotedblright \ (cf. \cite{Prato14}, Theorem $2.4$) , there exists a probability space $(\widetilde{\Omega}, \tilde{\mathcal{F}},\tilde{P})$ with the filtration $\{\tilde{\mathcal{F}}_t\}$ and $\tilde{\textbf{x}}_t^n$, $\tilde{\textbf{x}}_t$ belong to $C([0,T]; \mathbb{R}^d)$ with $n\in \mathbb{N}$, such that $Q(\tilde{\textbf{x}}) = Q(\textbf{x})$, $Q(\tilde{\textbf{x}}_t^n) = Q(\textbf{x}_t^n)$ with $n \in \mathbb{N}$, and $\tilde{\textbf{x}}_t^n \to \tilde{\textbf{x}}_t$ as $n \to \infty$, $\tilde{P}-$a.s. 
	
	By using this argument, we  get that $\tilde{\textbf{x}}_t^n$ converges to $\tilde{\textbf{x}}_t$ a.s. in the uniform topology on compacts sets, and then $\tilde{\textbf{x}}_t^n$ converges in probability towards $\tilde{\textbf{x}}_t$. It leads to
	\begin{align*}
	\int_{r}^{t}\tilde{b}_n(\tilde{\textbf{x}}_y^n,y)dy &\rightarrow \int_{r}^{t}\tilde{b}(\tilde{\textbf{x}}_y,y)dy
	\end{align*}
	in probability.
	To prove that these new processes satisfy the SDEs, we rely on an argument of Bensoussan cf. \cite{Bensoussan1995}. Essentially,  we need to check that the pair $(\tilde{\textbf{x}}_{\cdot}^n, \tilde{B}_{\cdot})$ satisfies the following equation 
	\begin{align}\label{equalweak1}
	\tilde{\textbf{x}}_t^n = \tilde{\textbf{x}}_0^n + \int_{0}^tb_n(\tilde{\textbf{x}}_y^n,y)dy + \int_0^t\tilde{\sigma}d\tilde{B}_y.
	\end{align}
	Let us call
	\begin{align}
	\tilde{\mathcal{M}}_t^n:=\tilde{\textbf{x}}_t^n - \tilde{\textbf{x}}_0^n - \int_{0}^tb_n(\tilde{\textbf{x}}_y^n,y)dy - \int_0^t\tilde{\sigma}d\tilde{B}_y.\nonumber
	\end{align}
	To prove \eqref{equalweak1}, we define the following equation
	\begin{align}
	\mathcal{M}_t^n:=\textbf{x}_t^n - \textbf{x}_0^n - \int_{0}^tb_n(\textbf{x}_y^n,y)dy - \int_0^t\tilde{\sigma}dB_y.\nonumber
	\end{align}
	Clearly, this definition implies $\mathcal{M}_t^n=0 \quad P \text{ a.s.}$. Hence, we have 
	$
	E\left[\frac{\mathcal{M}_t^n}{\mathcal{M}_t^n + 1}\right] = 0$. 
	Now, we want to check that
	\begin{align}
	\tilde{E}\left[\frac{\tilde{\mathcal{M}}_t^n}{\tilde{\mathcal{M}}_t^n + 1}\right] = 0.\label{MM}
	\end{align}
	Consider the fact that
	\begin{align}
	\frac{\mathcal{M}_t^n}{\mathcal{M}_t^n + 1} = \phi^n(\textbf{x}_{\cdot}^n,B_{\cdot}) \quad \text{ and } \frac{\tilde{\mathcal{M}}_t^n}{\tilde{\mathcal{M}}_t^n + 1} = \phi^n(\tilde{\textbf{x}}_{\cdot}^n,\tilde{B}_{\cdot}),\nonumber
	\end{align}
	where $\phi^n$ belong to $\mathcal{B}(E)$ which is the Borel sets of $E$ with $E := C([0,T]; \mathbb{R}^d)$.\\
	We note that
	\begin{align}
	\tilde{E}\left[\frac{\tilde{\mathcal{M}}_t^n}{\tilde{\mathcal{M}}_t^n + 1}\right] = \tilde{E}[\phi^n(\tilde{\textbf{x}}_{\cdot}^n,\tilde{B}_{\cdot})] = \int_{\mathcal{B}(E)}\phi^n dQ^n = E[\phi^n(\textbf{x}_{\cdot}^n,B_{\cdot})] = E\left[\frac{M_t^n}{\mathcal{M}_t^n + 1}\right].\nonumber
	\end{align}
	Thus, \eqref{MM} holds. 
	This implies $\tilde{\mathcal{M}}_t^n = 0 \quad \tilde{P} \text{ a.s.}$ Therefore, the new process, posed in the  new probability space, satisfies the SDE.\qed
\end{proof}
\begin{proposition}
	The solution of SDE system \eqref{micro_eqn1} and \eqref{micro_eqn2} is unique.
\end{proposition}
\begin{proof}
	Assume that we have two distinct solutions $\textbf{x}_1$ and $\textbf{x}_2$ belonging to $C([0,T];\mathbb{R}^d)$ with continuous sample paths almost surely.  Then it also holds
	\begin{align}
	\textbf{x}_1(t) - \textbf{x}_2(t) = \int_{0}^{t}(b(\textbf{x}_1,y) - b(\textbf{x}_2,y))dy,\nonumber
	\end{align}
	and hence, 
	\begin{align}\label{2solminus}
	E(|\textbf{x}_1(t) - \textbf{x}_2(t)|) \leq E\left(\left|\int_{0}^{t}b(\textbf{x}_1(y),y) - b(\textbf{x}_2(y),y)dy\right|\right).
	\end{align}
	For a detailed check, we  consider 
	\begin{align}
	E\begin{pmatrix}\label{expect_1}
	|\int_{0}^{t}F_1(\textbf{x}_1(y),y) - F_1(\textbf{x}_2(y),y)dy|\\
	|\int_{0}^{t}F_2(\textbf{x}_1(y),y) - F_2(\textbf{x}_2(y),y)dy|
	\end{pmatrix}.
	\end{align}
	Since the terms of $F_1$ is Lipschitz, the first line of \eqref{expect_1} reads
	\begin{align}\label{1st_line}
	&\left|\int_{0}^{t}F_1(\textbf{x}_1(y),y) - F_1(\textbf{x}_2(y),y)dy\right| \nonumber\\&=\Bigg|\int_{0}^{t}\Bigg(-\tilde{v}_{s}({\textbf x}_{a_i}^1(y),y)\left(\frac{\nabla \phi({\textbf x}_{a_i}^1(y))}{\|\nabla \phi({\textbf x}_{a_i}^1(y))\|}\right)(p_{\max} - p({\textbf x}_{a_i}^1(y),y)) \nonumber\\&+ \tilde{v}_{s}({\textbf x}_{a_i}^2(y),y)\left(\frac{\nabla \phi({\textbf x}_{a_i}^2(y))}{\|\nabla \phi({\textbf x}_{a_i}^2(y))\|}\right)(p_{\max} - p({\textbf x}_{a_i}^2(y),y))\Bigg)dy\Bigg|\nonumber\\
	&\leq C\int_{0}^{t}|\textbf{x}_{a_i}^2(y) - \textbf{x}_{a_i}^1(y)|dy.
	\end{align}
	By the same argument, the second line of \eqref{expect_1} becomes
	\begin{align}\label{2nd_line}
	&\left|\int_{0}^{t}F_2(\textbf{x}_1(y),y) - F_2(\textbf{x}_2(y),y)dy\right|=\nonumber\\
	&\Bigg|\int_{0}^{t}\Bigg(\sum_{j = 1}^N\frac{({\textbf x}_{c_j}^1 - {\textbf x}_{b_k}^1)}{\|{\textbf x}_{c_j}^1 - {\textbf x}_{b_k}^1\|}w(\hat{\delta},s({\textbf x}_{b_k}^1,y)) - \nabla H_\epsilon({\textbf x}_{b_k}^1,y) \nonumber\\&- \sum_{j = 1}^N\frac{({\textbf x}_{c_j}^2 - {\textbf x}_{b_k}^2)}{\|{\textbf x}_{c_j}^2 - {\textbf x}_{b_k}^2\|}w(\hat{\delta},s({\textbf x}_{b_k}^2,y)) + \nabla H_\epsilon({\textbf x}_{b_k}^2,y)\Bigg)dy\Bigg|= \nonumber\\
	\nonumber\\
	&\Bigg|\int_{0}^t\sum_{j=1}^{N_A}\frac{\textbf{x}_{a_j}^1 - \textbf{x}_{b_k}^1}{\|\textbf{x}_{a_j}^1 - \textbf{x}_{b_k}^1\|}w(\hat{\delta},s(\textbf{x}_{b_k}^1,y))-\sum_{j=1}^{N_A}\frac{\textbf{x}_{a_j}^2 - \textbf{x}_{b_k}^2}{\|\textbf{x}_{a_j}^2 - \textbf{x}_{b_k}^2\|}w(\hat{\delta},s(\textbf{x}_{b_k}^2,y))\nonumber\\
	&+\sum_{j=1}^{N_B}\frac{\textbf{x}_{b_j}^1 - \textbf{x}_{b_k}^1}{\|\textbf{x}_{b_j}^1 - \textbf{x}_{b_k}^1\|}w(\hat{\delta},s(\textbf{x}_{b_k}^1,y))-\sum_{j=1}^{N_B}\frac{\textbf{x}_{b_j}^2 - \textbf{x}_{b_k}^2}{\|\textbf{x}_{b_j}^2 - \textbf{x}_{b_k}^2\|}w(\hat{\delta},s(\textbf{x}_{b_k}^2,y)) \nonumber\\
	&\nabla H_\epsilon({\textbf x}_{b_k}^2,y) - \nabla H_\epsilon({\textbf x}_{b_k}^1,y)dy
	\Bigg|\nonumber\\
	&= \left|\int_{0}^{t}(A_1 + A_2 + A_3)dy\right|,
	\end{align}
	where
	\begin{align*}
	A_1&:= \sum_{j=1}^{N_A}\frac{\textbf{x}_{a_j}^1 - \textbf{x}_{b_k}^1}{\|\textbf{x}_{a_j}^1 - \textbf{x}_{b_k}^1\|}w(\hat{\delta},s(\textbf{x}_{b_k}^1,y))-\sum_{j=1}^{N_A}\frac{\textbf{x}_{a_j}^2 - \textbf{x}_{b_k}^2}{\|\textbf{x}_{a_j}^2 - \textbf{x}_{b_k}^2\|}w(\hat{\delta},s(\textbf{x}_{b_k}^2,y)),\\
	A_2&:= \sum_{j=1}^{N_B}\frac{\textbf{x}_{b_j}^1 - \textbf{x}_{b_k}^1}{\|\textbf{x}_{b_j}^1 - \textbf{x}_{b_k}^1\|}w(\hat{\delta},s(\textbf{x}_{b_k}^1,y))-\sum_{j=1}^{N_B}\frac{\textbf{x}_{b_j}^2 - \textbf{x}_{b_k}^2}{\|\textbf{x}_{b_j}^2 - \textbf{x}_{b_k}^2\|}w(\hat{\delta},s(\textbf{x}_{b_k}^2,y)),\\
	A_3&:= \nabla H_\epsilon({\textbf x}_{b_k}^2,y) - \nabla H_\epsilon({\textbf x}_{b_k}^1,y).
	\end{align*}
	By the Lipschitz property of the weight factors, the term $\left|\int_{0}^{t}A_1dy\right|$ reads
	\begin{align}\label{A_1}
	\left|\int_{0}^{t}A_1dy\right| &= \Bigg|\int_{0}^t \sum_{j=1}^{N_A}\frac{\textbf{x}_{a_j}^1 - \textbf{x}_{b_k}^1}{\|\textbf{x}_{a_j}^1 - \textbf{x}_{b_k}^1\|}w(\hat{\delta},s(\textbf{x}_{b_k}^1,y))-\sum_{j=1}^{N_A}\frac{\textbf{x}_{a_j}^1 - \textbf{x}_{b_k}^1}{\|\textbf{x}_{a_j}^1 - \textbf{x}_{b_k}^1\|}w(\hat{\delta},s(\textbf{x}_{b_k}^2,y))\nonumber\\ &+ \sum_{j=1}^{N_A}\frac{\textbf{x}_{a_j}^1 - \textbf{x}_{b_k}^1}{\|\textbf{x}_{a_j}^1 - \textbf{x}_{b_k}^1\|}w(\hat{\delta},s(\textbf{x}_{b_k}^2,y))-\sum_{j=1}^{N_A}\frac{\textbf{x}_{a_j}^2 - \textbf{x}_{b_k}^2}{\|\textbf{x}_{a_j}^2 - \textbf{x}_{b_k}^2\|}w(\hat{\delta},s(\textbf{x}_{b_k}^2,y))dy\Bigg|\nonumber\\
	&=\Bigg|\int_{0}^t\sum_{j=1}^{N_A}\frac{\textbf{x}_{a_j}^1 - \textbf{x}_{b_k}^1}{\|\textbf{x}_{a_j}^1 - \textbf{x}_{b_k}^1\|}\left(w(\hat{\delta},s(\textbf{x}_{b_k}^1,y)) - w(\hat{\delta},s(\textbf{x}_{b_k}^2,y))\right) \nonumber\\
	&+\sum_{j=1}^{N_A}\left(\frac{\textbf{x}_{a_j}^1 - \textbf{x}_{b_k}^1}{\|\textbf{x}_{a_j}^1 - \textbf{x}_{b_k}^1\|} -\frac{\textbf{x}_{a_j}^2 - \textbf{x}_{b_k}^2}{\|\textbf{x}_{a_j}^2 - \textbf{x}_{b_k}^2\|} \right)w(\hat{\delta},s(\textbf{x}_{b_k}^2,y))dy \Bigg|\nonumber\\
	&\leq C_1\int_{0}^{t}\left|\textbf{x}_{b_k}^1(y) - \textbf{x}_{b_k}^2(y)\right|dy + C_2\int_{0}^{t}\Bigg(\left|\textbf{x}_{a_j}^1(y) - \textbf{x}_{a_j}^2(y)\right| \nonumber\\&+ \left|\textbf{x}_{b_k}^1(y) - \textbf{x}_{b_k}^2(y) \right|\Bigg)dy.
	\end{align}
	In \eqref{A_1}, $C_1, C_2$ are constants with $C_2$ defined as
	\begin{align}
	C_2 := \max\left\{\frac{1}{\|\textbf{x}_{a_j}^1 - \textbf{x}_{b_k}^1\|},\frac{1}{\|\textbf{x}_{a_j}^2 - \textbf{x}_{b_k}^2\|}\right\}. \nonumber 
	\end{align}
	Thus, \eqref{A_1} can be written as
	\begin{align}\label{ineA1}
	\left|\int_{0}^{t}A_1dy\right|\leq C\int_{0}^t\left|\textbf{x}_{a_i}^1(y) - \textbf{x}_{a_i}^2(y)\right| + \left|\textbf{x}_{b_k}^1(y) - \textbf{x}_{b_k}^2(y) \right|dy.
	\end{align}
	By the same argument, we consider $\left|\int_{0}^{t}A_2dy\right|$, then we obtain
	\begin{align}\label{ineA2}
	\left|\int_{0}^{t}A_2dy\right| &\leq C\int_{0}^{t}\left|\textbf{x}_{b_j}^1(y) - \textbf{x}_{b_j}^2(y)\right| + \left|\textbf{x}_{b_k}^1(y) - \textbf{x}_{b_k}^2(y)\right|dy\nonumber\\
	&\leq C\int_{0}^{t}\left|\textbf{x}_{b_k}^1(y) - \textbf{x}_{b_k}^2(y)\right|dy.
	\end{align}
	Now, using the Lipschitz property of $\nabla H_\varepsilon$, we estimate the last term $\left|\int_{0}^{t}A_3dy\right|$ by 
	\begin{align}\label{ineA3}
	\left|\int_{0}^{t}A_3dy\right|\leq C\int_{0}^{t}\left|\textbf{x}_{b_k}^1(y) - \textbf{x}_{b_k}^2(y)\right|dy.
	\end{align}
	Combining \eqref{ineA1}, \eqref{ineA2} and \eqref{ineA3}, gives
	\begin{align}
	\left|\int_{0}^{t}(A_1 + A_2 + A_3)dy\right| \leq C\int_{0}^t\left|\textbf{x}_{a_i}^1(y) - \textbf{x}_{a_i}^2(y)\right| + \left|\textbf{x}_{b_k}^1(y) - \textbf{x}_{b_k}^2(y) \right|dy.\nonumber
	\end{align}
	Therefore, from \eqref{1st_line}, \eqref{2nd_line}, together with taking expectation, we obtain 
	\begin{align}
	&E\begin{pmatrix}
	|\int_{0}^{t}F_1(\textbf{x}_1(y),y) - F_1(\textbf{x}_2(y),y)dy|\\ |\int_{0}^{t}F_2(\textbf{x}_1(y),y) - F_2(\textbf{x}_2(y),y)dy|
	\end{pmatrix} \nonumber\\&\leq C \begin{pmatrix}\int_{0}^t|\textbf{x}_{a_i}^1(y) - \textbf{x}_{a_i}^2(y)|dy\\
	\int_{0}^t\left|\textbf{x}_{a_i}^1(y) - \textbf{x}_{a_i}^2(y)\right| + \left|\textbf{x}_{b_k}^1(y) - \textbf{x}_{b_k}^2(y) \right|dy
	\end{pmatrix}\nonumber
	\end{align}
	From \eqref{2solminus}, we get the following estimate
	\begin{align}
	\begin{pmatrix}
	E\left(\left|\textbf{x}_{a_i}^1 - \textbf{x}_{a_i}^2\right|\right)\\
	E\left(\left|\textbf{x}_{b_k}^1 - \textbf{x}_{b_k}^2\right|\right)
	\end{pmatrix} \leq C \begin{pmatrix}\int_{0}^t|\textbf{x}_{a_i}^1(y) - \textbf{x}_{a_i}^2(y)|dy\\
	\int_{0}^t\left|\textbf{x}_{a_i}^1(y) - \textbf{x}_{a_i}^2(y)\right| + \left|\textbf{x}_{b_k}^1(y) - \textbf{x}_{b_k}^2(y) \right|dy
	\end{pmatrix}\nonumber
	\end{align}
	Thanks to the Gr\"{o}nwall lemma, we obtain
	\begin{align}
	\begin{pmatrix}
	E\left(\left|\textbf{x}_{a_i}^1 - \textbf{x}_{a_i}^2\right|\right)\\
	E\left(\left|\textbf{x}_{b_k}^1 - \textbf{x}_{b_k}^2\right|\right)
	\end{pmatrix} = \begin{pmatrix}
	0\\0
	\end{pmatrix}\nonumber
	\end{align}
	This implies that $\textbf{x}_1(t) = \textbf{x}_2(t)$ almost surely that 
	\begin{align}
	P\left(\sup_{t \in [0,T]}\left|\textbf{x}_1(t) - \textbf{x}_2(t) \right|=0\right)=1.\nonumber
	\end{align}
	\qed
\end{proof}

\subsection{Background results}

This section contains a few remarks about the regularity of the agents's paths as well as of the concentration of smoke.  These results are fairly  standard; we add them here for the sake of completeness of our arguments.  

\subsubsection{A regularized Eikonal equation}\label{eiko}

In this section, we regularize the Eikonal equation introduced for Model 1; see  \eqref{eq:eikonal}. This is often referred to as a \textquotedblleft viscous\textquotedblright \ Eikonal equation. 

For $\varepsilon>0$, we introduce the following semilinear viscous problem  approximating as $\varepsilon\to 0$ our Eikonal equation: 

Find $\phi_\varepsilon\in C(\overline{\Omega})\cap C^2(\Omega)$ satisfying
\begin{align}
\begin{cases}
-\varepsilon\Delta \phi_\varepsilon + |\nabla\phi_\varepsilon|^2 = f^2 \quad &\text{ in } \Omega,\\
\phi_\varepsilon(x) = 0 \quad &\text{ at } \partial \Omega \cup \partial G,\\
\nabla\phi_\varepsilon\cdot \textbf{n} = g \quad &\text{ at } E,
\end{cases}
\end{align}
With suitable assumptions on $f,g,\Omega$, this problem with mixed Dirichlet-Neumann boundary conditions can be shown to be well-posed; see e.g. Theorem 2.1 , p.10, in \cite{Schieborn:2006} for the case of the Dirichlet problem. 
Note also that it is sometimes convenient to transform this semilinear PDE via
\begin{align}
w_a := \exp(-\varepsilon^{-1}\phi_\varepsilon) - 1,
\end{align}
where $a = \frac{1}{\varepsilon}$.
Then $w_a$ becomes a solution of the following linear PDE with mixed Dirichlet-Robin boundary conditions:
\begin{align}
\begin{cases}\label{transform:eqn}
-\Delta w_a + f^2 a^2 w_a + a^2 = 0 \quad &\text{ in } \Omega,\\
w_a = 0 \quad &\text{ at } \partial \Omega \cup \partial G,\\
\nabla w_a\cdot \textbf{n} = \tilde{g}(w_a) \quad &\text{ at } E,
\end{cases}
\end{align}
where $\tilde{g}(w_a)= -\varepsilon^{-1}(w_a + 1)g.$


\subsubsection{Higher regularity estimates for the smoke concentration}

We introduce the evolution of fire throughout a diffusion-dominated  convection process. The production and spreading of smoke, with the smoke density $s(\textbf{x},t)$, are described as the following diffusion-drift-reaction equation:  
\begin{eqnarray}
\begin{cases}\label{paraboliceqn1}
\partial_t s + \textrm{div} (-D\nabla s + \textbf{v}_ds) = y_{s}H(\textbf{x},t) &\textrm{ in } \Omega \times (0,T],\\
\left(-D\nabla s + \textbf{v}_ds\right)\cdot\textbf{n} = 0 &\textrm{ on } \partial \Omega \cup \partial G\times (0,T], \\
\left(-D\nabla s + \textbf{v}_ds\right)\cdot\textbf{n} = \lambda s &\textrm{ at } \partial E \times (0,T],\\
s(x,0) = s_0 &\textrm{ in } \Omega \times \{t = 0\},
\end{cases}
\end{eqnarray}
where $D$ is the smoke diffusive coefficient, $\textbf{v}_d$ is a given drift corresponding (e.g. wind's velocity,\ldots), $y_{s}$ is a smoke production coefficient, while $H$ represents the shape and intensity of the fire. The center of the fire location is denoted by $\textbf{x}_0$ with radius $r_0$. $H$ reads
\begin{align}\label{H_def}
H(\textbf x,t) = \begin{cases}
R(\textbf x,t) \quad \text{ if } |\textbf x-{\textbf x}_0| <r_0,\\
0 \quad \text{ otherwise },
\end{cases}
\end{align}
where $R(\textbf x,t)$ is defined by
\begin{align}
R(\textbf{x},t) = c(t)\exp\left(-\kappa\frac{|\textbf x-{\textbf x}_0|}{L}\right).\nonumber
\end{align}
Here, $\kappa$ is the convection heat transfer constant coefficient, $c(t)$ is a constant function depending on $t$, $L$ is the typical length of a stationary temperature distribution within the geometry and $\lambda$ is an interface exchange smoke coefficient. For convenience, in order to take the gradient of $H$, we consider $H_\epsilon$ a suitable mollification of $H$. 
In our case, from now on, we consider the coefficient $y_s$ as a constant $c_y$ and put $f(x,t) := c_yH_\epsilon(x,t)$, then \eqref{paraboliceqn1} becomes
\begin{align}
\begin{cases}\label{paraboliceqn}
\partial_t s + \text{div} (-D\nabla s + \textbf{v}_ds) = f(\textbf{x},t) &\text{ in } \Omega \times (0,T],\\
\left(-D\nabla s + \textbf{v}_ds\right)\cdot\textbf{n} = 0 &\text{ on } \partial \Omega \cup \partial G\times (0,T], \\
\left(-D\nabla s + \textbf{v}_ds\right)\cdot\textbf{n} = \lambda s &\text{ at } \partial E \times (0,T],\\
s(\textbf{x},0) = s_0 &\text{ in } \Omega \times \{t = 0\},
\end{cases}
\end{align}
In order to have a well-posed dynamics of pedestrians model, we need the solution of \eqref{paraboliceqn} to belong to $C([0,T];C^1(\Omega))$. Since the pedestrian dynamics system couple one way with the smoke equation, the solution $s$ of \eqref{paraboliceqn} should be Lipschitz to guarantee the well-posedness of the system. In the next part, we adapt the approach in \cite{Pankavich15} to get a short proof of increased parabolic regularity for a bounded domain $\Omega$ in $\mathbb{R}^d$. Moreover, from now on, we assume the boundaries $\partial \Omega \cup \partial G$ and $\partial E$ are $C^2$ (or, at least, they satisfy the exterior sphere condition).
\begin{theorem}\label{lower_regularity}[Lower-order regularity]
	Assume Assumptions $(\text{A}_3)$, $(\text{A}_4)$ to hold. Suppose $f\in H^1(\Omega)$ and $\textbf{v}_d \in W^{1,\infty}(\Omega)$. Then, for any $T>0$, $t \in (0,T]$, there exists a unique 
	\begin{align}
	s \in C([0,T];H^1(\Omega)) \text{ and } s' \in L^2(0,T;H^{-1}(\Omega))\nonumber
	\end{align}
	that solves \eqref{paraboliceqn}. Furthermore, the following a priori estimates hold
	\begin{align}
	\sup_{t \in [0,T]}\|s\|_{L^2(\Omega)}^2 \leq C_T\left(\|s_0\|_{L^2(\Omega)}^2 + \|f\|_{L^2(\Omega)}^2\right) \text{ and } \|\nabla s\|_{L^2(\Omega)}^2 \nonumber\\\leq \frac{C_T}{t}\left(\|s_0\|_{L^2(\Omega)}^2+\|f\|_{H^1(\Omega)}^2\right).\nonumber
	\end{align}
\end{theorem}
\begin{proof} We adapt the arguments from \cite{Pankavich15} to our setting and split the proof into fourth steps:
	\begin{itemize}
		\item Step 1: {\it Galerkin approximation}
		
		Firstly, we assume that the functions $w_k = w_k(x) (k \in \mathbb{N})$ are smooth and that
		\begin{align}
		\{w_k\}_{k=1}^{\infty} \text{ is an orthonormal basis of } H^1(\Omega).
		\end{align}
		We are looking for an approximation of \eqref{paraboliceqn} in the form
		\begin{align}\label{approxima_galerkin1}
		s_m(t) := \sum_{k=1}^{m}d_n^k(t)w_k,
		\end{align}
		where the coefficients $d_m^k$ satisfy the following system
		\begin{align}\label{weakformgalerkin1}
		\begin{cases}
		\langle s'_m, w_k\rangle_{L^2(\Omega)} + \langle D\nabla s_m,\nabla w_k\rangle_{L^2(\Omega)} - \langle\textbf{v}_ds_m,\nabla w_k\rangle_{L^2(\Omega)} \\+ \langle \lambda s_m,w_k\rangle_{L^2(\partial E)} = \langle f,w_k \rangle_{L^2(\Omega)}, \\
		s_m(0) = s_{0m} \text{ with } k = 1\ldots m,
		\end{cases}
		\end{align}
		where 
		\begin{align}\label{s0approximateform1}
		s_{0m} = \sum_{k = 1}^{m}c_{m}^{k}w_k \to s_0
		\end{align} 
		strongly in $L^2(\Omega)$.
		\item Step 2: {\it A priori estimates}
		
		The goal of this step is to obtain some useful a priori estimates. Multiplying \eqref{weakformgalerkin1} by $d_m^k(t)$, taking the summation for $k\in \{1,\ldots,m\}$. Then recalling \eqref{approxima_galerkin1}, using Green's formula together with the mixed boundary condition, we arrive at
		\begin{align}\label{weakform_test1}
		\frac{1}{2}\frac{d}{dt}\|s_m(t)\|_{L^2(\Omega)}^2 + \int_{\Omega}\nabla s_m\cdot D\nabla s_mdx +  \int_{\partial E}\lambda s_m^2d\sigma(E) \nonumber\\= \int_{\Omega}s_m\textbf{v}_d\cdot\nabla s_mdx + \int_{\Omega}fs_mdx.
		\end{align}
		Thanks to Cauchy-Schwarz's inequality $\varepsilon$ for an $\varepsilon >0$, we have
		\begin{align}
		\frac{1}{2}\frac{d}{dt}\|s_m(t)\|_{L^2(\Omega)}^2 + \int_{\Omega}\nabla s_m\cdot D\nabla s_mdx + \int_{\partial E}\lambda s_m^2d\sigma(E) \leq  + \|\textbf{v}_d\|_{1,\infty}\Big(\varepsilon\|s_m\|_{L^2(\Omega)}^2 \nonumber\\+\frac{1}{\varepsilon}\|\nabla s_m\|_{L^2(\Omega)}^2 \Big) + \frac{1}{2}\left(\|f\|_{L^2(\Omega)}^2 + \|s_m\|_{L^2(\Omega)}^2\right).\nonumber
		\end{align}
		Next, by using the ellipticity property of the diffusion coefficient $D$ and the assumption on the interface exchange coefficient, we obtain
		\begin{align}\label{prior_1}
		\frac{1}{2}\frac{d}{dt}\|s_m(t)\|_{L^2(\Omega)}^2 \leq -\underline{\theta}\|\nabla s_m\|_{L^2(\Omega)}^2 + \underline{\gamma}\|s_m\|_{L^2(\partial E)}^2 + \|\textbf{v}_d\|_{1,\infty}\Bigg(\frac{1}{\varepsilon}\|s_m\|_{L^2(\Omega)}^2 \nonumber\\+\varepsilon\|\nabla s_m\|_{L^2(\Omega)}^2 \Bigg) + \frac{1}{2}\left(\|f\|_{L^2(\Omega)}^2 + \|s_m\|_{L^2(\Omega)}^2\right).
		\end{align}
		By the trace inequality applied to $\|s_m\|_{L^2(\partial E)}^2$, \eqref{prior_1} reads
		\begin{align}
		\frac{1}{2}\frac{d}{dt}\|s_m(t)\|_{L^2(\Omega)}^2 \leq -\underline{\theta}\|\nabla s_m\|_{L^2(\Omega)}^2 + C(\underline{\gamma})\|s_m\|_{H^1(\Omega)}^2 + \|\textbf{v}_d\|_{1,\infty}\Bigg(\frac{1}{\varepsilon}\|s_m\|_{L^2(\Omega)}^2 \nonumber\\+\varepsilon\|\nabla s_m\|_{L^2(\Omega)}^2 \Bigg) + \frac{1}{2}\left(\|f\|_{L^2(\Omega)}^2 + \|s_m\|_{L^2(\Omega)}^2\right).\nonumber
		\end{align}
		By choosing $\varepsilon = \underline{\theta}(2\|\textbf{v}_d\|_{1,\infty})^{-1}$, we get the following estimate
		\begin{align}\label{priori_1}
		\frac{1}{2}\frac{d}{dt}\|s_m(t)\|_{L^2(\Omega)}^2 \leq C\left(\|f\|_{L^2(\Omega)}^2 + \|s_m\|_{H^1(\Omega)}^2\right) + \left(C(\underline{\gamma}) - \frac{\underline{\theta}}{2}\right)\|\nabla s_m\|_{L^2(\Omega)}^2.
		\end{align}
	Multiplying with $\varphi = \partial_x s_m$ \eqref{paraboliceqn} differentiated with respect to $x$, we obtain after integrating by part that
		\begin{align}
		\frac{1}{2}\frac{d}{dt}\|\partial_x s_m\|_{L^2(\Omega)}^2 + \int_{\Omega}\nabla \partial_x s_m\cdot D \nabla \partial_x s_mdx + \int_{\Omega}\nabla\partial_x s_m \cdot \partial_x D\nabla s_mdx \nonumber\\- \int_{\Omega}\nabla\partial_x s_m\cdot\textbf{v}_d\partial_xs_mdx -\int_{\Omega}\nabla\partial_xs_m\cdot \partial_x \textbf{v}_ds_mdx+ \int_{\partial E}\lambda|\partial_x s_m|^2d\sigma(E)\nonumber\\ + \int_{\partial E}\partial_x(\lambda s_m)\partial_xs_md\sigma (E) = \int_{\Omega}\partial_x f \partial_x s_mdx.\nonumber
		\end{align}
		This leads to 
		\begin{align}\label{priori_test2}
		\frac{1}{2}\frac{d}{dt}\|\partial_x s_m\|_{L^2(\Omega)}^2=-\int_{\Omega}\nabla \partial_x s_m\cdot D \nabla \partial_x s_mdx - \int_{\Omega}\nabla\partial_x s_m \cdot \partial_x D\nabla s_mdx\nonumber\\ + \int_{\Omega}\nabla\partial_x s_m\cdot\textbf{v}_d\partial_xs_mdx +\int_{\Omega}\nabla\partial_xs_n\cdot \partial_x \textbf{v}_ds_mdx - \int_{\partial E}\lambda|\partial_x s_m|^2d\sigma(E)\nonumber\\ - \int_{\partial E}\partial_x(\lambda s_m)\partial_xs_md\sigma(E) + \int_{\Omega}\partial_x f \partial_x s_mdx.
		\end{align}
		Using the assumptions on $D$ and $\lambda$ as well as  Cauchy-Schwarz's inequality for the RHS of \eqref{priori_test2}, we obtain the following estimate
		\begin{align}
		\frac{1}{2}\frac{d}{dt}\|\partial_x s_m\|_{L^2(\Omega)}^2 \leq -\underline{\theta} \|\nabla \partial_x s_m\|_{L^2(\Omega)}^2 + \|D\|_{W^{1,\infty}}\Bigg(\varepsilon_1\|\nabla\partial_xs_m\|_{L^2(\Omega)}^2 \nonumber\\+ \frac{1}{\varepsilon_1}\|\nabla s_m\|_{L^2(\Omega)}^2\Bigg)
		+ |\textbf{v}_d|\left(\varepsilon_{2'}\|\nabla \partial_xs\|_{L^2(\Omega)}^2+\frac{1}{\varepsilon_{2'}}\|\partial_xs\|_{L^2(\Omega)}^2\right) \nonumber\\+ \|\textbf{v}_d\|_{1,\infty}\left(\varepsilon_2\|\nabla \partial_xs\|_{L^2(\Omega)}^2 + \frac{1}{\varepsilon_2}\|s\|_{L^2(\Omega)}^2\right) +\underline{\gamma}\|\partial_xs_m\|_{L^2(\partial E)}^2 \nonumber\\+ \|\lambda\|_{1,\infty}\|\partial_xs_m\|_{L^2(\partial E)}^2 + \frac{1}{2}\left(\|\partial_xf\|_{L^2(\Omega)}^2 + \|\partial_xs_m\|_{L^2(\Omega)}^2\right).\nonumber
		\end{align}
		By choosing $\varepsilon_1= \underline{\theta}(4\|D\|_{1,\infty})^{-1}$, $\varepsilon_{2'}=\underline{\theta}(8C)^{-1}$, $\varepsilon_2 = \underline{\theta}(8\|\textbf{v}_d\|_{1,\infty})^{-1}$ together with the use of the trace inequality to handle the boundary terms, we arrive at
		\begin{align}\label{priori_2}
		\frac{1}{2}\frac{d}{dt}\|\partial_x s_m\|_{L^2(\Omega)}^2 &\leq -\frac{\underline{\theta}}{2}\|\nabla\partial_xs_m\|_{L^2(\Omega)}^2 + C\|\nabla s_m\|_{L^2(\Omega)}^2 + C\|\partial_xs_m\|_{L^2(\Omega)}^2\nonumber\\&+ C\|s_m\|_{L^2(\Omega)}^2  +C(\underline{\gamma})\left(\|\partial_x s_m\|_{L^2(\Omega)}^2 + \|\nabla \partial_x s_m\|_{L^2(\Omega)}^2\right)
		\nonumber\\&+\frac{1}{2}\left(\|\partial_x f\|_{L^2(\Omega)}^2 + \|\partial_x s_m\|_{L^2(\Omega)}^2\right)\nonumber\\
		&\leq \left(-\frac{\underline{\theta}}{2} + C(\underline{\gamma})\right)\|\nabla\partial_xs_m\|_{L^2(\Omega)}^2 \nonumber\\&+ C\left(\|\nabla s_m\|_{L^2(\Omega)}^2 + \|\partial_xs_m\|_{L^2(\Omega)}^2 + \|s_m\|_{L^2(\Omega)}^2\right).
		\end{align}
		Taking the summation over all first order derivatives, we have
		\begin{align}
		\frac{1}{2}\frac{d}{dt}\|\nabla s_m\|_{L^2(\Omega)}^2 \leq \left(C(\underline{\gamma})- \frac{\underline{\theta}}{2}\right)\|\nabla^2s_m\|_{L^2(\Omega)}^2 + C\left(\| s_m\|_{H^1(\Omega)}^2 + \|f\|_{H^1(\Omega)}^2\right).\nonumber
		\end{align}
		Let us introduce a linear expansion in $t$ as follow
		\begin{align}\label{expansion}
		\zeta_1(t) = \|s_m\|_{L^2(\Omega)}^2 + \frac{C(\underline{\theta},\underline{\gamma})t}{2}\|\nabla s_m\|_{L^2(\Omega)}^2.
		\end{align}
		Taking the derivative of \eqref{expansion} with respect to $t$, we obtain
		\begin{align}
		\zeta_1'(t) = \frac{d}{dt}\|s_m\|_{L^2(\Omega)}^2 + \frac{C(\underline{\theta},\underline{\gamma})}{2}\|\nabla s_m\|_{L^2(\Omega)}^2 + \frac{C(\underline{\theta},\underline{\gamma})t}{2}\frac{d}{dt}\|\nabla s_m\|_{L^2(\Omega)}^2.\nonumber
		\end{align}
		Combining \eqref{priori_1} and \eqref{priori_2}, we are led to the following estimate
		\begin{align}
		\zeta'(t) \leq 2C\left(\|f\|_{L^2(\Omega)}^2 + \|s_m\|_{H^1(\Omega)}^2\right) - 2\left(C(\underline{\gamma})- \frac{\underline{\theta}}{2}\right)\|\nabla s_m\|_{L^2(\Omega)}^2 \nonumber\\+ \frac{C(\underline{\theta},\underline{\gamma})}{2}\|\nabla s_m\|_{L^2(\Omega)}^2 + \frac{C(\underline{\theta},\underline{\gamma})t}{2}\Bigg[\left(C(\underline{\gamma})- \frac{\underline{\theta}}{2}\right)\|\nabla^2s_m\|_{L^2(\Omega)}^2 \nonumber\\+ C\left(\| s_m\|_{H^1(\Omega)}^2 + \|f\|_{H^1(\Omega)}^2\right)\Bigg].\nonumber
		\end{align}
		Choosing $\underline{\theta},\underline{\gamma}$ such that
		$- \frac{\underline{\theta}}{2} + C(\underline{\gamma})< 0$ and put $- \frac{\underline{\theta}}{2} + C(\underline{\gamma}) =: -C(\underline{\theta},\underline{\gamma})$, we obtain
		\begin{align}\label{bef_gronwall}
		\zeta'(t) \leq C_T\left(\|f\|_{H^1(\Omega)}^2 + \zeta(t)\right) \textrm{ for a.e. } t \in (0,T).
		\end{align}
		Applying Gr\"{o}nwall's inequality to \eqref{bef_gronwall}, we have the following estimate
		\begin{align}\label{Gronwall}
		\zeta(t) \leq C_{T}\left(\zeta(0) + \|f\|_{H^1(\Omega)}^2\right) = C_T\left(\|s_0\|_{L^2(\Omega)}^2 + \|f\|_{H^1(\Omega)}^2\right).
		\end{align}
		Combining \eqref{expansion} and \eqref{Gronwall} gives
		\begin{align}\label{pri_resl1}
		\|s_m(t)\|_{L^2(\Omega)}^2 \leq C_T\left(\|s_0\|_{L^2(\Omega)}^2 + \|f\|_{H^1(\Omega)}^2\right)
		\end{align}
		and 
		\begin{align}\label{pri_resl2}
		\|\nabla s_m\|_{L^2(\Omega)}^2 \leq \frac{C_T}{Ct}\left(\|s_0\|_{L^2(\Omega)}^2 + \|f\|_{H^1(\Omega)}^2\right).
		\end{align}
		The estimates \eqref{pri_resl1} and \eqref{pri_resl2} imply that $s_m$ is a bounded sequence in $H^1(\Omega)$ and a.e $t \in (0,T)$.
		\item Step 3: { \it Passage to the limit $m \to \infty$}
		
		Using the {\em a priori} estimates \eqref{pri_resl1} and \eqref{pri_resl2}, we obtain the following inequality
		\begin{align}
		\int_{0}^{T}\frac{1}{2}\frac{d}{dt}\|s_m(t)\|_{L^2(\Omega)}^2dt + \int_{0}^{T}\|\nabla s_m\|_{L^2(\Omega)}^2dt \leq C_T\int_{0}^{T}\left(\|f\|_{H^1(\Omega)}^2+\|s_0\|_{L^2(\Omega)}^2\right).\nonumber
		\end{align}
		This implies that $(s_m)$ is a bounded sequence in $L^2(0,T;H^1(\Omega))$.
		
		On the other hand, in order to use Aubin-Lions's lemma, we additionally need to prove $s_m' \in L^2(0,T;H^{-1}(\Omega))$. Take an arbitrary $v \in H^1(\Omega)$, with $\|v\|_{H^1(\Omega)} \leq 1$. We can deduce for a.e. $0 < t < T$ that
		\begin{align}
		\langle s_m', v\rangle_{L^2(\Omega)} = \langle f,v\rangle_{L^2(\Omega)} + \langle \textbf{v}_ds_m,\nabla v\rangle_{L^2(\Omega)} - \langle D\nabla s_m,\nabla v\rangle_{L^2(\Omega)} - \langle\lambda s_m,v\rangle_{L^2(\partial E)}.\nonumber
		\end{align}
		Then, we get
		\begin{align}\label{est_st11}
		|\langle s_m', v\rangle| \leq  C\|s_m\|_{H^1(\Omega)} + C\|f\|_{L^2(\Omega)}.
		\end{align}
		for $\|v\|_{W^{1,2}(\Omega)} \leq 1$. Moreover, \eqref{est_st11} implies that
		\begin{align}\label{est_st12}
		\|s_m'\|_{H^{-1}(\Omega)} \leq C\left(\|s_m\|_{H^1(\Omega)} + \|f\|_{L^2(\Omega)}\right) .
		\end{align}
		Integrating \eqref{est_st12} on $(0,T)$, we obtain the following estimate
		\begin{align}\label{est_st13}
		\int_{0}^{T}\|s_m'\|_{H^{-1}(\Omega)}^2dt &\leq C\int_{0}^{T} \|s_m\|_{H^1(\Omega)}+\|f\|_{L^2(\Omega)}dt 
		\nonumber\\&\leq C\left(\|s_0\|_{L^2(\Omega)}^2 + \|f\|_{L^2(0,T;L^2(\Omega))}\right).
		\end{align}
		Thus, $s_m' \in L^2(0,T;H^{-1}(\Omega))$.
		Therefore, we conclude that
		\begin{align}
		\begin{cases}
		s_m \rightharpoonup s \text{ weakly in } L^2(0,T;H^1(\Omega)),\\
		{s'}_m \rightharpoonup s' \text{ weakly in } L^2(0,T;H^{-1}(\Omega)).
		\end{cases}\nonumber
		\end{align}
		Relying on Aubin-Lions lemma in \cite{Boyer2013} with $p,q = 2$,
		$$E_0 = H^1(\Omega), \quad E = L^2(\Omega), \quad E_1 = H^{-1}(\Omega)$$
		together with Rellich theorem (cf. \cite{Evans1997}, Section 5.7, Theorem 1) for the compactness embedding $H^1(\Omega) \subset L^2(\Omega)$, we have the sequence $\{s_m\}$ is relatively compact in $L^2(0,T;L^2(\Omega))$ in the strong topology. This sequence also weakly relatively compact in $L^2(0,T;H^1(\Omega))$ and weakly star relatively compact in $C([0,T];L^2(\Omega))$. 
		Hence, there exists a subsequence $s_{m_k}$ (just for simplicity of notation, let us 
		denote it by $s_m$) which converges to a function $s$ belonging to $L^2(0,T;H^1(\Omega))$ and $C([0,T];L^2(\Omega))$. Therefore, we can conclude that there exists a solution $s \in L^2(0,T;H^1(\Omega)) \cup C([0,T];L^2(\Omega))$ satisfying equation \eqref{paraboliceqn}.
		\item Step 4: {\it Uniqueness of solutions}
		
		 Assume that equation \eqref{paraboliceqn} admits $2$ solutions $s_1$ and $s_2$ belonging to\\ $L^2(0,T;H^1(\Omega)) \cup C([0,T];L^2(\Omega))$. Denote $w = s_1 - s_2$. Then equation \eqref{paraboliceqn} becomes
		\begin{align}
		\begin{cases}
		\partial_t w + \text{div}(-D\nabla w + \textbf{v}_dw) = 0\quad \text{ in } \Omega \times (0,T],\\
		(-D\nabla w + \textbf{v}_dw)\cdot \textbf{n} = 0\quad \text{ on } \partial \Omega \cup \partial G\times (0,T],\\
		(-D\nabla w + \textbf{v}_dw)\cdot \textbf{n} = \lambda w \quad\text{ at } \partial E \times (0,T],\\
		w(t = 0) = 0 \quad \text{ in } \Omega \times \{t=0\},
		\end{cases}\nonumber
		\end{align}
		Recalling \eqref{weakformgalerkin1}, we note that
		\begin{align}
		\frac{1}{2}\frac{d}{dt}\int_{\Omega}w^2dx + \int_{\Omega}D|\nabla w|^2dx + \int_{\partial E}\lambda w^2d\sigma(E) = \int_{\Omega}w\textbf{v}_d\cdot\nabla wdx ,\nonumber
		\end{align}
		which leads to
		\begin{align}
		\frac{d}{dt}\left(\|w\|_{L^2(\Omega)}^2\right) + \overline{\theta}\|\nabla w\|_{L^2(\Omega)}^2 + \overline{\gamma}\|w\|_{L^2(\partial E)}^2 \leq C\|w\|_{L^2(\Omega)}^2.\nonumber
		\end{align}
		This also implies 
		\begin{align}\label{gronwall_12}
		\frac{d}{dt}\left(\|w\|_{L^2(\Omega)}^2\right) \leq C\|w\|_{L^2(\Omega)}^2.
		\end{align}
		Integrating \eqref{gronwall_12} on $(0,T)$, gives
		\begin{align}
		\|w\|_{L^2(\Omega)}^2 \leq \|w(0)\|_{L^2(\Omega)}^2 + C\int_{0}^t\|w\|_{L^2(\Omega)}^2.\nonumber
		\end{align}
		Gr\"{o}nwall's lemma ensure
		\begin{align}
		\|w\|_{L^2(\Omega)}^2 \leq \|w(0)\|_{L^2(\Omega)}^2(1 + Cte^{Ct}),\nonumber
		\end{align}
		which for $w(0) = 0$, gives $\|w\|_{L^2(\Omega)} = 0$. So, $w = 0$ a.e. in $\Omega$ and everywhere in $[0,T]$, which ensures the desired uniqueness.
		
		Now, let us show that $s \in C([0,T];H^1(\Omega))$. We consider $w_r(t) = s(t+r) - s(t)$, then $w_r(t)$ satisfies the equation \eqref{paraboliceqn} with $f = 0$, $w(0) = s_0 - s(r)$ and $\lambda s(t+r) - \lambda s(t) = \lambda w_r(t)$. By using similar argument, we obtain
		\begin{align}
		\|w_r(t)\|_{L^2(\Omega)}^2 + \frac{C(\underline{\theta},\underline{\gamma}) t}{2}\|\nabla w_r(t)\|_{L^2(\Omega)}^2 \leq C_T\left(\|s_0 - s(r)\|_{L^2(\Omega)}^2\right).\nonumber
		\end{align}
		Since we have $s \in C([0,T];L^2(\Omega))$, then $\lim_{r \to \infty}\|s(t+r)-s(t)\| = 0$ and \\ $\lim_{r \to \infty}\|\nabla s(t+r)-\nabla s(t)\| = 0$ for $t>0$.
		Therefore, we obtain $s \in C([0,T];H^1(\Omega))$.
	\end{itemize}
\qed
\end{proof}
\begin{theorem}\label{Higer_order}[High-order regularity]
	Assume $(A_3)$, $(A_4)$ to hold. Suppose $f\in H^m(\Omega)$ and  $\textbf{v}_d \in W^{m,\infty}(\Omega)$ for every $m \in \mathbb{N}$, and $s_0 \in L^2(\Omega)$. Then, for any $T>0$, $t \in [0,T]$, the solution of \eqref{paraboliceqn} satisfies the following estimate
	\begin{align}
	\|\nabla^k s\|_{L^2(\Omega)}^2 \leq \frac{C_T}{t^k}\left(\|s_0\|_{L^2(\Omega)}^2 + \|f\|_{H^k(\Omega)}^2\right) \text{ for } k=0,1,\ldots,m.\nonumber
	\end{align}	
\end{theorem}
\begin{proof}
	We use the method of induction on $m \in \mathbb{N}$, using the fact that we have done the first case $m=1$ of induction mathematically in Theorem \ref{lower_regularity}. 
	As a notation for derivatives, we define
	\begin{align}
	\|\nabla^k s\|_{L^2(\Omega)}^2 := \sum_{|\alpha|\leq k}\|\partial_x^{\alpha}s\|_{L^2(\Omega)}^2.\nonumber
	\end{align}

	Now, taking the $k-$order derivative with respect to $x$ for $k\in \mathbb{N}$ which is denoted by $\partial_x^{\alpha}$ of the equation \eqref{paraboliceqn}, multiplying by $\partial_x^\alpha s$ and integrating the results by parts together with using Green's theorem for the equation, we obtain
	\begin{align}
	\frac{1}{2}\frac{d}{dt}\|\partial_x^\alpha s(t)\|_{L^2(\Omega)}^2 + \int_{\Omega}\nabla\partial_x^{\alpha}s\cdot \sum_{j=0}^{k}\sum_{|\beta| = j,\beta+\gamma = \alpha}\mybinom{\alpha}{\beta}\partial_x^{\beta}D\nabla\partial_x^{\gamma}sdx \nonumber\\-\int_{\Omega}\nabla\partial_x^{\alpha}s\cdot \sum_{j=0}^{k}\sum_{|\beta| = j,\beta+\gamma = \alpha}\mybinom{\alpha}{\beta}\partial_x^{\beta}\textbf{v}_d\partial_x^{\gamma}sdx + \int_{\partial E}\partial_x^{\alpha}s\lambda\partial_x^{\alpha}sd\sigma(E)\nonumber\\+\int_{\partial E}\partial_x^{\alpha}(\lambda s)\cdot \partial_x^{\alpha}sd\sigma(E)
	=\int_{\Omega}\partial_x^{\alpha}f\partial_x^{\alpha}sdx,\nonumber
	\end{align}
	and thus
	\begin{align}\label{high_derivative}
	\frac{1}{2}\frac{d}{dt}\|\partial_x^\alpha s(t)\|_{L^2(\Omega)}^2 =- \int_{\Omega}\nabla\partial_x^{\alpha}s\cdot \sum_{j=0}^{k}\sum_{|\beta| = j,\beta+\gamma = \alpha}\mybinom{\alpha}{\beta}\partial_x^{\beta}D\nabla\partial_x^{\gamma}sdx\nonumber\\+\int_{\Omega}\nabla\partial_x^{\alpha}s\cdot\sum_{j=0}^{k}\sum_{|\beta| = j,\beta+\gamma = \alpha}\mybinom{\alpha}{\beta}\partial_x^{\beta}\textbf{v}_d\partial_x^{\gamma}sdx -\int_{\partial E}\partial_x^{\alpha}s\lambda\partial_x^{\alpha}sd\sigma(E)\nonumber\\ - \int_{\partial E}\partial_x^{\alpha}(\lambda s)\cdot \partial_x^{\alpha}sd\sigma(E)
	+\int_{\Omega}\partial_x^{\alpha}f\partial_x^{\alpha}sdx.
	\end{align}
    Denote
	\begin{align}
	A&:=- \int_{\Omega}\nabla\partial_x^{\alpha}s\cdot \sum_{j=0}^{k}\sum_{|\beta| = j,\beta+\gamma = \alpha}\mybinom{\alpha}{\beta}\partial_x^{\beta}D\nabla\partial_x^{\gamma}sdx\nonumber\\
	&= -\int_{\Omega}\nabla\partial_x^{\alpha}s\cdot D \nabla\partial_x^{\alpha}sdx - \int_{\Omega}\nabla\partial_x^{\alpha}s\cdot \sum_{j=1}^{k}\sum_{|\beta| = j,\beta+\gamma = \alpha}\mybinom{\alpha}{\beta}\partial_x^{\beta}D\nabla\partial_x^{\gamma}sdx.\nonumber
	\end{align}
	We can estimate $|A|$ from above, it follows
	\begin{align}\label{est_A}
	|A| \leq -\underline{\theta}\|\nabla\partial_x^{\alpha}s\|_{L^2(\Omega)}^2 + C\|D\|_{1,\infty}\left(\varepsilon_1\|\nabla \partial_x^{\alpha}s\|_{L^2(\Omega)}^2 + \frac{1}{\varepsilon_1}\|s\|_{H^{k-1}(\Omega)}^2\right)\nonumber\\
	\leq \frac{-\underline{\theta}}{2}\|\nabla\partial_x^{\alpha}s \|_{L^2(\Omega)}^2 + C\|s\|_{H^{k-1}(\Omega)}^2,
	\end{align}
	where we choose $\varepsilon_1 = \underline{\theta}(4C\|D\|_{m,\infty})^{-1}$.
	Set
	\begin{align}
	B:=\int_{\Omega}\nabla\partial_x^{\alpha}s\cdot\sum_{j=0}^{k}\sum_{|\beta| = j,\beta+\gamma = \alpha}\mybinom{\alpha}{\beta}\partial_x^{\beta}\textbf{v}_d\partial_x^{\gamma}sdx,\nonumber
	\end{align}
	and obtain the upper bound
	\begin{align}\label{est_B}
	|B|\leq C\|\textbf{v}_d\|_{m,\infty}\left(\varepsilon_2 \|\nabla \partial_x^{\alpha}s\|_{L^2(\Omega)}^2 + \frac{1}{\varepsilon_2}\|\partial_x^{\alpha}s\|_{H^{k-1}(\Omega)}^2\right).
	\end{align}
	Now, let us label the third and fourth terms in the right hand side of \eqref{high_derivative} as follow
	\begin{align}
	\tilde{C}:=-\int_{\partial E}\partial_x^{\alpha}s\lambda\partial_x^{\alpha}sd\sigma(E) - \int_{\partial E}\partial_x^{\alpha}(\lambda s)\cdot \partial_x^{\alpha}sd\sigma(E).\nonumber
	\end{align}
	Using the assumptions on $\lambda$ together with applying Cauchy's inequality, trace inequality for $\tilde{C}$, we have the following estimate:
	\begin{align}\label{est_C}
	\tilde{C} \leq \underline{\gamma}\|\partial_x^{\alpha}s\|_{L^2(\partial E)}^2 + \|\lambda\|_{m,\infty}\|\partial_x^\alpha s\|_{L^2(\partial E)}^2
	\leq C(\underline{\gamma})\left(\|\nabla\partial_x^{\alpha}s\|_{L^2(\Omega)}^2 + \|\partial_x^{\alpha}s\|_{L^2(\Omega)}^2\right).
	\end{align}
	Finally, we estimate the last term of \eqref{high_derivative}, by using Cauchy's inequality, we obtain
	\begin{align}\label{est_D}
	\int_{\Omega}\partial_x^{\alpha}f\partial_x^{\alpha}sdx 
	\leq \frac{1}{2}\left(\|\partial_x^{\alpha}f\|_{L^2(\Omega)}^2 + \|\partial_x^{\alpha}s\|_{L^2(\Omega)}^2\right).
	\end{align}
	Combining \eqref{est_A}-\eqref{est_D} and choosing $\varepsilon_2 = \underline{\theta}(4C\|\textbf{v}_d\|_{1,\infty})^{-1}$, we have the following estimate
	\begin{align}
	\frac{1}{2}\frac{d}{dt}\|\partial_x^\alpha s(t)\|_{L^2(\Omega)}^2 \leq \frac{-\underline{\theta}}{2}\|\nabla\partial_x^{\alpha}s \|_{L^2(\Omega)}^2 + C\|s\|_{H^{k-1}(\Omega)}^2 + C\|\partial_x^\alpha s\|_{H^{k-1}(\Omega)}^2 \nonumber\\+ C(\underline{\gamma})\|\nabla \partial_x^{\alpha}s\|_{L^2(\Omega)}^2 + \|f\|_{H^k(\Omega)}^2.\nonumber
	\end{align}
	Now, summing all of first-order derivatives, we obtain
	\begin{align}\label{nablak_est}
	\frac{1}{2}\frac{d}{dt}\|\nabla^k s(t)\|_{L^2(\Omega)}^2 \leq \left(C(\underline{\gamma}) - \frac{\underline{\theta}}{2}  \right)\|\nabla^{k+1}s\|_{L^2(\Omega)}^2 + C\left(\|f\|_{H^k(\Omega)}^2 + \|s\|_{H^k(\Omega)}^2\right).
	\end{align}
	Now, we aim to find $s \in C([0,T],H^{m-1}(\Omega))$ using the induction hypothesis under the assumptions $f \in H^{m-1}(\Omega)$ and $D,\lambda,\textbf{v}_d \in W^{m,\infty}(0,T;\Omega)$. Using the same argument as in the case $m=1$, we define
	\begin{align}\label{expansion_2}
	\zeta_2(t) := \sum_{k=1}^{n}\frac{(C(\underline{\theta},\underline{\gamma})t)^k}{2^kk!}\|\nabla^k s\|_{L^2(\Omega)}^2.
	\end{align} 
	Taking the derivative of \eqref{expansion_2} with respect to $t$, we obtain
	\begin{align}
	\zeta'_2(t) &= \sum_{k=1}^m\frac{(C(\underline{\theta},\underline{\gamma}))^kt^{k-1}}{2^k(k-1)!}\|\nabla^ks\|_{L^2(\Omega)}^2 + \sum_{k=0}^{m}\frac{(C(\underline{\theta},\underline{\gamma}))^k}{2^kk!}\frac{d}{dt}\|\nabla^ks\|_{L^2(\Omega)}^2\nonumber\\
	&:= G_1+G_2.\nonumber
	\end{align}	
	\begin{align}\label{G_2}
	G_2 &\leq \sum_{k=0}^m\frac{(C(\underline{\theta},\underline{\gamma})t)^k}{2^kk!}\left(-C(\underline{\theta},\underline{\gamma})\|\nabla^{k+1}s\|_{L^2(\Omega)}^2 + C\left(\|f\|_{H^k(\Omega)}^2 + \|s\|_{H^k(\Omega)}^2\right)\right)\nonumber\\
	&=-2\sum_{k=0}^m\frac{(C(\underline{\theta},\underline{\gamma}))^{k+1}t^k}{2^{k+1}k!}\|\nabla^{k+1}s\|_{L^2(\Omega)}^2 \nonumber\\&+ C\sum_{k=0}^{m}\frac{(C(\underline{\theta},\underline{\gamma}))^k}{2^kk!}\left(\|f\|_{H^k(\Omega)}^2 + \|s\|_{H^k(\Omega)}^2\right)\nonumber\\
	&\leq -2G_1 - \frac{2(C(\underline{\theta},\underline{\gamma}))^{m+1}t^m}{2^{m+1}m!}\|\nabla^{m+1}s\|_{L^2(\Omega)}^2 \nonumber\\&+ C\sum_{k=0}^{m}\frac{(C(\underline{\theta},\underline{\gamma}))^k}{2^kk!}\left(\|f\|_{H^k(\Omega)}^2 + \|s\|_{H^k(\Omega)}^2\right).
	\end{align}
	On the other hand, the induction hypothesis gives the following inequality
	\begin{align}\label{induction_hypo}
	\|s\|_{H^{k-1}(\Omega)}^2 \leq \frac{C_T}{t^{k-1}}\left(\|f\|_{H^{k-1}(\Omega)}^2+\|s_0\|_{L^2(\Omega)}^2\right).
	\end{align}
	Combining \eqref{G_2} and \eqref{induction_hypo}, we obtain
	\begin{align}
	\zeta'_2(t) &\leq C\sum_{k=0}^{m}\frac{(C(\underline{\theta},\underline{\gamma}))^k}{2^kk!}\left(\|f\|_{H^k(\Omega)}^2 + \|s\|_{H^k(\Omega)}^2\right)\nonumber\\
	&\leq C_T\left(\|f\|_{H^m(\Omega)}^2 + \sum_{k=0}^m\frac{(C(\underline{\theta},\underline{\gamma})t)^k}{2^kk!}\left[\|\nabla^k s\|_{L^2(\Omega)}^2 + \|s\|_{H^{k-1}(\Omega)}^2\right]\right)\nonumber\\
	&\leq C_T\left(\|f\|_{H^m(\Omega)}^2 + \zeta_2(t) +\sum_{k=0}^m\frac{(C(\underline{\theta},\underline{\gamma})t)^k}{2^kk!}\frac{C_T}{t^{k-1}}\left[\|f\|_{H^{k-1}(\Omega)}^2 + \|s_0\|_{L^2(\Omega)}^2\right] \right)\nonumber\\
	&\leq C_T\left(\|f\|_{H^m(\Omega)}^2 + \|s_0\|_{L^2(\Omega)}^2 + \zeta_2(t)\right).\nonumber
	\end{align}
	Gr\"{o}nwall's inequality yields
	\begin{align}
	\zeta_2(t) \leq C_T\left(\|f\|_{H^m(\Omega)}^2 + \|s_0\|_{L^2(\Omega)}^2 + \zeta_2(0)\right) \leq C_T\left(\|f\|_{H^m(\Omega)}^2 + \|s_0\|_{L^2(\Omega)}^2\right).\nonumber
	\end{align}
	The bound on $\zeta_2(t)$ gives the following estimate
	\begin{align}
	\|\nabla^m s\|_{L^2(\Omega)}^2 \leq \frac{C_T}{(C(\underline{\theta},\underline{\gamma})t)^m},\nonumber
	\end{align}
	which completes the induction proof. \qed
\end{proof}
\begin{remark}\label{rm_C1}
From Theorem \ref{Higer_order}, for $m = 3$, $\Omega \subset \mathbb{R}^d$ with $d=2$, there exists a unique solution $s \in C([0,T]; C^1(\Omega))$ and $s' \in L^2(0,T;H^{-1}(\Omega))$ that solves the equation \eqref{paraboliceqn}. 

By the same arguments as in Theorem \ref{lower_regularity}, this also implies that $s \in C([0,T];H^m(\Omega))$. On the other hand, in our model, we consider our domain in $\Omega \subset \mathbb{R}^d$ with $d = 2$. Moreover, assume $\Omega$ satisfies the strong locally Lipschitz condition (cf. \cite{Adams03}, Theorem 4.12), taking $m=3$, hence $H^3(\Omega)$ compact embedding into $C^1(\Omega)$, i.e. $H^3(\Omega) \subset C^1(\Omega)$. As a conclusion, we obtain $s \in C([0,T]; C^1(\Omega))$. This property ensures that the smoke concentration $s$ is Lipschitz with respect to the space variable -- a fact needed to handle the well-posedness of our  SDEs.
\end{remark}

\section{Discussion}
\label{sec:discussion}

In this chapter, we presented various models aimed at modelling crowds of mixed populations (active and passive) moving inside heterogeneous environments.

Based on our numerical experiments, we observed the impact of passive agents on the residence times of the population and conclude that the lack of environment knowledge can have a substantial impact on the evacuation. 
Additionally, we notice that the size of the obstacles and doors have a significant influence on the overall dynamics.

While the presence of passive agents increases the evacuation time, we speculate that by manipulating the spacial distribution of active particles, it is possible to optimize the residence time of the passive agents. We plan to investigate these aspects in a forthcoming publication.

From the mathematical point of view, the situation becomes a lot more challenging when there is a feedback mechanism between the agent-based dynamics and the environment (fire, smoke, geometry). Formulating this relationship mathematically would allow for an optimization approach, eventually in a multiscale setting.
The main advantage of such a mathematical framework would be to contribute to an intelligent design of building interiors and to provide a basis for smart evacuation signaling systems.

\bibliographystyle{spmpsci}
\bibliography{literature}
\end{document}